\documentclass[11pt,reqno]{amsart}

\usepackage{amssymb}
\usepackage{amsmath}
\usepackage{amsthm}
\usepackage{graphicx,psfrag}
\usepackage{enumerate}
 \usepackage{overpic}

\usepackage[margin=1.2in]{geometry}

\usepackage[all]{xy}
\usepackage[dvipsnames]{xcolor}
\usepackage{hyperref}

\theoremstyle{plain}
\newtheorem{theorem}{Theorem}[section]
\newtheorem{lemma}[theorem]{Lemma}
\newtheorem{corollary}[theorem]{Corollary}

\newtheorem{conjecture}[theorem]{Conjecture}

\theoremstyle{definition}
\newtheorem{definition}[theorem]{Definition}
\newtheorem*{notation}{Notation}

\theoremstyle{remark}
\newtheorem{remark}[theorem]{Remark}

\newtheorem*{remark*}{Remark}
\newtheorem*{claim*}{Claim}
\newtheorem{claim}{Claim}
\newtheorem{claims}{Claim}[section]

\newtheorem*{acknowledgments}{Acknowledgments}

\numberwithin{figure}{section}

\newcommand{\Int}{\mathrm{int}}

\begin{document}

\title{Heegaard genus, degree-one maps, and amalgamation of 3-manifolds}
\author{Tao Li}
\address{Department of Mathematics \\
 Boston College \\
 Chestnut Hill, MA 02467}
\email{taoli@bc.edu}
\thanks{Partially supported by an NSF grant}

\begin{abstract}
Let $M=\mathcal{W}\cup_\mathcal{T} \mathcal{V}$ be an amalgamation of two compact 3-manifolds along a torus, where $\mathcal{W}$ is the exterior of a knot in a homology sphere. 
Let $N$ be the manifold obtained by replacing $\mathcal{W}$ with a solid torus such that the boundary of a Seifert surface in $\mathcal{W}$ is a meridian of the solid torus.  This means that there is a degree-one map $f\colon M\to N$, pinching $\mathcal{W}$ into a solid torus while fixing $\mathcal{V}$. We prove that $g(M)\ge g(N)$, where $g(M)$ denotes the Heegaard genus.  An immediate corollary is that the tunnel number of a satellite knot cannot be smaller than the tunnel number of its pattern knot.
\end{abstract}

\maketitle

%\tableofcontents

\section{Introduction}\label{Sintro} 

Degree-one maps are fundamental objects in topology.   For 3-manifolds, such maps have close relations with the geometrization of 3-manifolds as well as many topological properties, e.g.~see \cite{BW, HWZ,  Rong, So, WZ}. 
It has been known for a long time that maps of nonzero degree between surfaces are standard \cite{E}.  However, many important questions remain open for maps between 3-manifolds.  
One of the most fundamental questions on degree-one maps between 3-manifolds is the relation between their Heegaard genera.  

\begin{conjecture}\label{Qgenus}
 Let $M$ and $N$ be closed orientable 3-manifolds and suppose there is a degree-one map $f\colon M\to N$.  Then $g(M)\ge g(N)$, where $g(M)$ is the Heegaard genus of $M$.
\end{conjecture}

Conjecture~\ref{Qgenus} is an old and difficult question in 3-manifold topology. 
It implies the Poincar\'{e} Conjecture: If a closed 3-manifold $N$ is homotopy equivalent to $S^3$, since a homotopy equivalence is a degree-one map, Conjecture~\ref{Qgenus} implies that $0=g(S^3)\ge g(N)$ and $N$ must be $S^3$.  

There is a general strategy of proving Conjecture~\ref{Qgenus} dated back to Haken and Waldhausen. 
Suppose that $f\colon M\to N$ is a degree-one map.  Let $N=\mathcal{W}_N\cup \mathcal{V}_N$ be a minimal genus Heegaard splitting of $N$, where $\mathcal{W}_N$ and $\mathcal{V}_N$ are genus-$g$ handlebodies.  
Let $\mathcal{V}=f^{-1}(\mathcal{V}_N)$ and $\mathcal{W}=f^{-1}(\mathcal{W}_N)$. 
By a theorem of Haken \cite{H1} and Waldhausen \cite{W1} (also see \cite{RW}), after some homotopy on the degree-one map $f$, we may assume that  $f|_{\mathcal{V}}\colon \mathcal{V}\to \mathcal{V}_N$ is a homeomorphism.    So we have 
 $M=\mathcal{W}\cup \mathcal{V}$, where $\partial \mathcal{W}=\partial \mathcal{V}=\mathcal{W}\cap \mathcal{V}$,
   $\mathcal{V}$ is a genus-$g$ handlebody, and  $f|_{\mathcal{V}}$ is a homeomorphism.

Consider $\mathcal{W}$, $\mathcal{W}_N$ and the map $f|_\mathcal{W}\colon \mathcal{W}\to \mathcal{W}_N$.  Let $\mathcal{T}=\partial \mathcal{W}=\partial \mathcal{V}$. For any compressing disk $D$ in the handlebody $\mathcal{W}_N$, after some homotopy on $f|_\mathcal{W}$, we may assume that $f^{-1}(D)$ is an incompressible surface in $\mathcal{W}$.  Thus there is a collection of $g$ non-separating simple closed curves $\gamma_1,\dots,\gamma_g$ in $\mathcal{T}$, such that 
\begin{enumerate}
  \item  $\mathcal{T}-\cup_{i=1}^g \gamma_i$ is connected and 
  \item  each $\gamma_i$ is the boundary of an incompressible surface in $\mathcal{W}$ and these incompressible surfaces are disjoint. 
\end{enumerate}
Note that, given any 3-manifold $\mathcal{W}$ with connected genus-$g$ boundary, if $\mathcal{W}$ satisfies these two conditions, one can always construct a degree-one map $f\colon \mathcal{W}\to \mathcal{W}_N$ by pinching each incompressible surface into a disk and pinching the complement of the incompressible surfaces into a 3-ball.  So we can formulate the question in Conjecture~\ref{Qgenus} as a question about surgery.

\begin{conjecture}\label{Qgenus3}
Suppose $M=\mathcal{W}\cup_\mathcal{T} \mathcal{V}$ with $\mathcal{T}=\partial \mathcal{W}=\partial \mathcal{V}=\mathcal{W}\cap \mathcal{V}$ a genus-$g$ surface.  Suppose $\mathcal{W}$ satisfies the two conditions above.  Let $N$ be the closed 3-manifold obtained by replacing $\mathcal{W}$ with a genus-$g$ handlebody $H$ such that each $\gamma_i$ in the conditions above bounds a disk in $H$.  Then $g(M)\ge g(N)$.
\end{conjecture}

Let $M$ and $N$ be the 3-manifolds in Conjecture~\ref{Qgenus3}. There is a degree-one map $f\colon M\to N$ pinching $\mathcal{W}$ into a handlebody while fixing $\mathcal{V}$.  Thus Conjecture~\ref{Qgenus3} follows from Conjecture~\ref{Qgenus}.  
Conversely, if Conjecture~\ref{Qgenus3} holds, then by the theorem of Haken \cite{H1} and Waldhausen \cite{W1} explained earlier, Conjecture~\ref{Qgenus} holds.  So the two conjectures are equivalent and 
a key to understanding a degree-one map between two closed 3-manifolds is the pinching map $f\colon \mathcal{W}\to \mathcal{W}_N$ which pinches $\mathcal{W}$ into a handlebody.  

In this paper, we study the genus one case, i.e.~the case that $\mathcal{T}$ is a torus.  Let $\mathcal{W}$ be the exterior of a knot in a homology sphere.  The Seifert surface of the knot gives a non-separating surface in $\mathcal{W}$.  
Hence $\mathcal{W}$ satisfies the two conditions above.   
We prove that Conjecture~\ref{Qgenus3} holds if $\mathcal{W}$ be a knot exterior in a homology sphere.

\begin{theorem}\label{Ttorus}
Let $M=\mathcal{W}\cup_\mathcal{T} \mathcal{V}$, where $\mathcal{W}$ is the exterior of a knot in a homology sphere and $\mathcal{T}$ is a torus.  Let $N=\widehat{\mathcal{T}}\cup_\mathcal{T} \mathcal{V}$ be the manifold obtained by replacing $\mathcal{W}$ with a solid torus such that the boundary of a Seifert surface in $\mathcal{W}$ is a meridian of the solid torus.  Then $g(M)\ge g(N)$.
\end{theorem}

As mentioned above, the degree-one map $f\colon \mathcal{W}\to \widehat{\mathcal{T}}$ from $\mathcal{W}$ to a solid torus $\widehat{\mathcal{T}}$ extends to a degree-one map from $M$ to $N$. 
Thus Theorem~\ref{Ttorus} gives some evidence for Conjectures~\ref{Qgenus} and \ref{Qgenus3}. Moreover, Theorem~\ref{Ttorus} has some interesting corollaries on the tunnel numbers of satellite knots.  Let $k$ be a satellite knot in $S^3$. So there is a knotted solid torus $V\subset S^3$ such that $k$ is a nontrivial knot in $V$.  The core curve of the solid torus $V$ is the companion knot for $k$.  If we re-embed $V$ into an unknotted solid torus in $S^3$, then the image of $k\subset V$ becomes a knot $k'$ in $S^3$ and $k'$ is called the pattern knot for $k$. 
Note that we can view $\mathcal{W}=S^3\setminus\Int(V)$ as a knot exterior. If we pinch $S^3\setminus\Int(V)$ into a solid torus, then the resulting manifold is still $S^3$ and $k$ becomes the pattern knot $k'$ after the pinching operation. Thus the following is an immediate corollary of Theorem~\ref{Ttorus}.

\begin{corollary}\label{C1}
The tunnel number of a satellite knot is larger than or equal to the tunnel number of its pattern knot.
\end{corollary}

If $k=k_1\# k_2$ is a connected sum of two knots, then a swallow-follow torus is an essential torus in $S^3\setminus N(k)$. So we may view $k$ as a satellite knot with $k_1$ as its companion and $k_2$ its pattern knot.  Thus Corollary~\ref{C1} implies the following theorem of Schirmer \cite{Schi}.

\begin{corollary}[\cite{Schi}]\label{C2} 
Let $k_1$ and $k_2$ be knots in $S^3$, then $t(k_1\# k_2)\ge\max\{t(k_1), t(k_2)\}$, where $t(k)$ denotes the tunnel number of a knot $k$. 
\end{corollary}

By a theorem in \cite{L18}, there are knots $k_1$ and $k_2$ with $t(k_1\# k_2)=t(k_1)$, so 
the inequalities in Corollary~\ref{C1} and Corollary~\ref{C2} are sharp.

\begin{acknowledgments}
I would like to thank Trent Schirmer for helpful conversations.  Some of the arguments in this paper are repackaging of arguments in \cite{Schi}. I would also like to thank the referees for many suggestions for improving the exposition of the paper.
\end{acknowledgments}

\section{Generalized Heegaard splittings}\label{SGHS}

\begin{notation}
For any topological space $X$, we use $\overline{X}$, $\Int(X)$, $|X|$, and $N(X)$ to denote the closure, interior, number of components, and a small neighborhood of $X$ respectively.  If $X$ is a surface, $g(X)$ denotes the genus of $X$, and if $X$ is a 3-manifold, $g(X)$ denotes the Heegaard genus of $X$.  Moreover, the genus of a disconnected surface is defined to be the sum of the genera of its components.
Throughout the paper, we use $I$ to denote the unit interval $[0,1]$.
\end{notation}

Since Heegaard genus is additive under connected sum \cite{H}, we only need to consider the case that $\mathcal{W}$ and $\mathcal{V}$ in Theorem~\ref{Ttorus} are irreducible.  If $\mathcal{W}$ or $\mathcal{V}$ is a solid torus, then  Theorem~\ref{Ttorus} holds trivially.  So we suppose neither $\mathcal{W}$ nor $\mathcal{V}$ is a solid torus.  As $\mathcal{W}$ and $\mathcal{V}$ are irreducible, this means that the torus $\mathcal{T}$ is incompressible, which implies that $M$ is also irreducible.

A Heegaard splitting $M=H_1\cup H_2$ can be viewed as a handle decomposition of $M$ where $H_1$ is the union of $0$- and $1$-handles and $H_2$ is the union of $2$- and $3$-handles.  
Scharlemann and Thompson \cite{ST} observed that one can rearrange these handles and obtain a sequence of nice surfaces.  This process is called untelescoping.  These surfaces give rise to a decomposition of $M$ into a collection of submanifolds $\mathcal{N}_0,\dots, \mathcal{N}_m$ along surfaces $\mathcal{F}_1,\dots,\mathcal{F}_m$, where $\mathcal{F}_i$ may not be connected.  Each $\mathcal{N}_i$ has a Heegaard surface $\mathcal{P}_i$ that decomposes $\mathcal{N}_i$ into two compression bodies $\mathcal{A}_i$ and $\mathcal{B}_i$, see Figure~\ref{Funtel} for a schematic picture. 

The 1-handles of $H_1$ are rearranged as the 1-handles of the $\mathcal{A}_i$'s and the 2-handles of $H_2$ are rearranged as the 2-handles of the $\mathcal{B}_i$'s.  This decomposition is called a {\bf generalized Heegaard splitting}.  Note that the Heegaard splittings of the $\mathcal{N}_i$'s can be amalgamated along the $\mathcal{F}_i$'s into the original Heegaard splitting $M=H_1\cup H_2$, see \cite{Sch}.  The genus of the generalized Heegaard splitting is defined to be the genus of splitting $M=H_1\cup H_2$.

\begin{figure}[h]
\begin{center}
\begin{overpic}[width=3in]{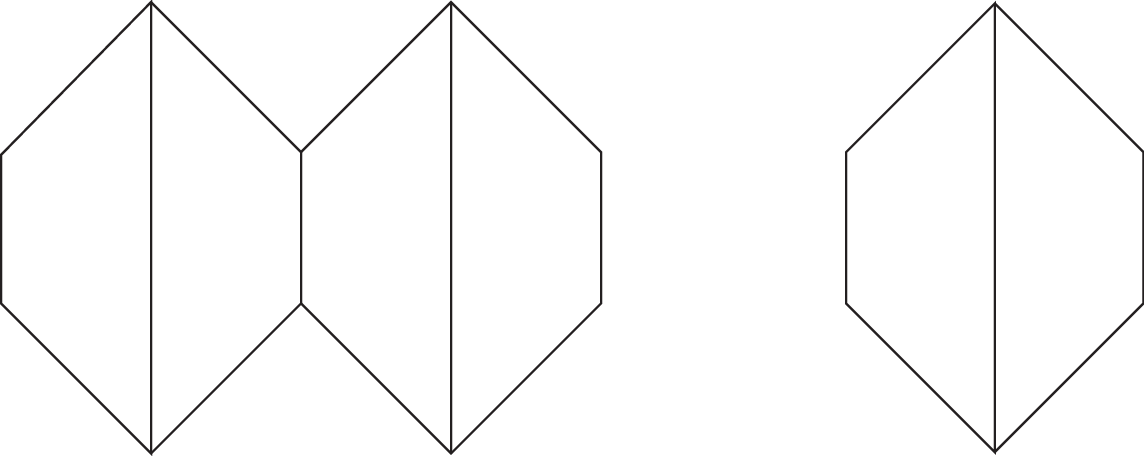}
\put(10,-5){$\mathcal{N}_0$}
\put(58,20){$\dots\dots$}
\put(37,-5){$\mathcal{N}_1$}
\put(85,-5){$\mathcal{N}_m$}
\put(24,8){$\mathcal{F}_1$}
\put(51,8){$\mathcal{F}_2$}
\put(71,8){$\mathcal{F}_m$}
\put(14,15){$\mathcal{P}_0$}
\put(40,15){$\mathcal{P}_1$}
\put(87,15){$\mathcal{P}_m$}
\put(4,25){$\mathcal{A}_0$}
\put(32,25){$\mathcal{A}_1$}
\put(77,25){$\mathcal{A}_m$}
\put(17,25){$\mathcal{B}_0$}
\put(43,25){$\mathcal{B}_1$}
\put(90,25){$\mathcal{B}_m$}
\includegraphics[width=3in]{untel.eps}
\end{overpic}
\vspace{10pt}
\caption{The diagram of a generalized Heegaard splitting}\label{Funtel}
\end{center}
\end{figure}

Scharlemann and Thompson \cite{ST} proved that if the Heegaard splitting $M=H_1\cup H_2$ is unstabilized, then one can rearrange the handles so that each $\mathcal{F}_i$ is incompressible in $M$ and each $\mathcal{P}_i$ is a strongly irreducible Heegaard surface of $\mathcal{N}_i$.  See \cite{CG, He, L8, L4} for the definition and some properties of strongly irreducible Heegaard splittings.

Suppose $M=H_1\cup H_2$ is a minimal genus Heegaard splitting of $M$. By \cite{ST}, we may assume each $\mathcal{F}_i$ is incompressible and each $\mathcal{P}_i$ is strongly irreducible.  Let $\Sigma$ be the union of all the surfaces $\mathcal{F}_i$'s and $\mathcal{P}_i$'s in the untelescoping. 
Suppose $M=\mathcal{W}\cup_\mathcal{T} \mathcal{V}$ and $\mathcal{T}$ is an incompressible torus.  A theorem of 
Bachman-Schleimer-Sedgwick \cite[Lemma 3.3 and Remark 3.4]{BSS} (also see \cite{L4, L14}) says that one can 
isotope $\Sigma$ to intersect $\mathcal{T}$ nicely. 

\begin{lemma}[Bachman-Schleimer-Sedgwick \cite{BSS}]\label{LBSS}
Let $M=\mathcal{W}\cup_\mathcal{T} \mathcal{V}$ and suppose $\mathcal{T}$ is an incompressible torus. Let $\Sigma$ be the a collection of incompressible and strongly irreducible surfaces in the decomposition above.  Then $\Sigma$ can be isotoped so that either
\begin{enumerate}
\item $\mathcal{T}$ is a component of some $\mathcal{F}_i$, or
\item $\Sigma$ transversely intersects $\mathcal{T}$, and each component of $\Sigma\cap \mathcal{W}$ and $\Sigma\cap \mathcal{V}$ is an essential or a strongly irreducible surface in $\mathcal{W}$ and $\mathcal{V}$ respectively.
\end{enumerate}
\end{lemma}

\begin{remark}\label{Rtorus}
If $\mathcal{T}$ is a component of $\Sigma$, then the minimal genus Heegaard splitting of $M$ is an amalgamation of Heegaard splittings of $\mathcal{W}$ and $\mathcal{V}$, see \cite{Sch}. In particular, $g(M)=g(\mathcal{W})+g(\mathcal{V})-g(\mathcal{T})=g(\mathcal{W})+g(\mathcal{V})-1\ge g(\mathcal{V})$, see \cite{La2, L4, L14}.  Since $N=\widehat{\mathcal{T}}\cup_{\mathcal{T}}\mathcal{V}$ is obtained by a Dehn filling on $\mathcal{V}$, we have $g(\mathcal{V})\ge g(N)$. Hence $g(M)\ge g(N)$ and Theorem~\ref{Ttorus} holds.  
Thus, to prove Theorem~\ref{Ttorus}, we only need to consider the case that $\Sigma$ transversely intersects $\mathcal{T}$, i.e.~the second conclusion of Lemma~\ref{LBSS}.
\end{remark}

The general strategy for the proof of Theorem~\ref{Ttorus} is to build a sequence of surfaces in the solid torus $\widehat{\mathcal{T}}$ according to $\Sigma\cap\mathcal{W}$, such that these surfaces
 merge with $\Sigma\cap\mathcal{V}$ and yield a generalized Heegaard splitting of $N$ with the same genus.

By Lemma~\ref{LBSS} and Remark~\ref{Rtorus}, we may assume that the intersection of the torus $\mathcal{T}$ with each compression body in the generalized Heegaard splitting consists of a collection of incompressible annuli.  These annuli divide each compression body into submanifolds with certain properties.  We describe these properties in the next section.

\section{Marked handlebodies and relative compression bodies}\label{Srelative compression body}

Let $\Sigma$ be the collection of surfaces in Lemma~\ref{LBSS}.  Our goal is to construct a collection of surfaces in the solid torus $\widehat{\mathcal{T}}$ so that these surfaces merge with $\Sigma\cap\mathcal{V}$ into a generalized Heeggard surface of $N=\widehat{\mathcal{T}}\cup_\mathcal{T} \mathcal{V}$. 
These surfaces in the solid torus $\widehat{\mathcal{T}}$ are mostly constructed by connecting $\partial$-parallel annuli using trivial tubes. 
We first describe such surfaces and the submanifolds bounded by these surfaces.

Before we proceed, we describe an operation on surfaces that we use many times throughout the paper. 
Let $S_1$ and $S_2$ be two surfaces in a 3-manifold and let $\alpha$ be an arc connecting $S_1$ to $S_2$.  
Let $H=D^2\times I$ be a 1-handle along $\alpha$ with $\alpha=\{x\}\times I$, and suppose $H\cap (S_1\cup S_2)=D^2\times\partial I$.  
Let $S'$ be the surface obtained by first removing the two disks $H\cap (S_1\cup S_2)$ from $S_1$ and $S_2$ and then connecting the resulting surfaces by the annulus $(\partial D^2)\times I$.  If $S_1\ne S_2$, we say that $S'$ is obtained by connecting $S_1$ and $S_2$ via a tube along $\alpha$ (topologically $S'\cong S_1\# S_2$).  If $S_1=S_2$ and $\alpha$ is a trivial arc (i.e.~parallel to an arc in $S_1$), then we say that $S'$ is obtained by adding a trivial tube.

\vspace{10pt}

\noindent
\textbf{Marked handlebodies}:

\vspace{10pt}

We call a genus-$g$ handlebody $X$ a \textbf{marked handlebody} if 
$\partial X$ contains $m$ ($m\le g$) disjoint essential annuli $A_1,\dots, A_m$ and 
$X$ contains $m$ disjoint compressing disks $\tau_1,\dots, \tau_m$ such that (1) $\tau_i\cap A_j=\emptyset$ if $i\ne j$, and (2) $\tau_i\cap A_i$ is a single essential arc of $A_i$ for each $i$. See Figure~\ref{Fex}(a) for a picture.  These annuli $A_i$ are called marked annuli in $\partial X$.

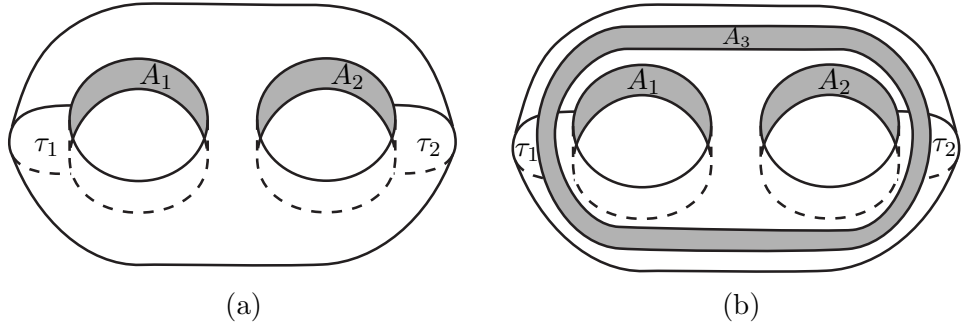
\begin{figure}[h]
	\begin{center}
		\begin{overpic}[width=5in]{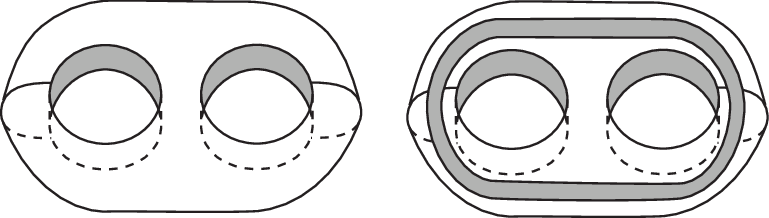}
			\put(23,-4){(a)}
			\put(75,-4){(b)}
			\put(14,20){$A_1$}				    \put(34,20){$A_2$}
			\put(65,19.5){$A_1$}
            \put(85,19.5){$A_2$}
            \put(75,24.2){\footnotesize{$A_3$}}
            \put(3,13){$\tau_1$}
            \put(43,13){$\tau_2$}
            \put(53.3,12.5){$\tau_1$}
            \put(97,13){$\tau_2$}
		\end{overpic}
		\vspace{7pt}
		\caption{Examples of a marked handlebody and a relative compression body}\label{Fex}
	\end{center}
\end{figure}

By connecting two parallel copies of $\tau_i$ via a band around $A_i$, 
we can construct a separating compressing disk in $X$ which cuts off a solid torus $T_i$ with
$A_i\subset \partial T_i$, see Figure~\ref{Fmono} for a picture.  
So we can give an equivalent but slightly different description of $X$: 
Start with a collection of solid tori $T_1,\dots, T_m$ and suppose $T_i=\tau_i\times S^1$, where $\tau_i$ is a bigon disk.  Let $a_i$ and $a_i'$ be the two boundary edges of the bigon $\tau_i$, and let $A_i=a_i\times S^1$ and $A_i'=a_i'\times S^1$ be the pair of annuli in $\partial T_i=\partial\tau_i\times S^1$.  Let $B$ be a 3-ball.  We first connect each $T_i$ to $B$ using a 1-handle $H_i$ ($i=1,\dots,m$) and require that each 1-handle $H_i$ ($i=1,\dots,m$) is attached to the annulus $A_i'$. Then we add $p$ 1-handles ($p=g-m$) $H_1',\dots, H_p'$ to $B$. The resulting handlebody is $X$, and  
the set of annuli $A_1,\dots A_m$ are our marked annuli on $\partial X$.  
We call $B$ in this construction the {\it central $3$-ball} of $X$.

The boundary of $X$ has two parts: (1) the vertical boundary $\partial_vX=\bigcup_{i=1}^m A_i$ and (2) the horizontal boundary $\partial_hX=\overline{\partial X\setminus \partial_vX}$. 

To put this in a broader context, one should view $X$ as a submanifold of the solid torus $\widehat{\mathcal{T}}$, $\partial_hX$ as a surface properly embedded in $\widehat{\mathcal{T}}$, and $\partial_vX$ as a collection of annuli in the torus $\mathcal{T}$.

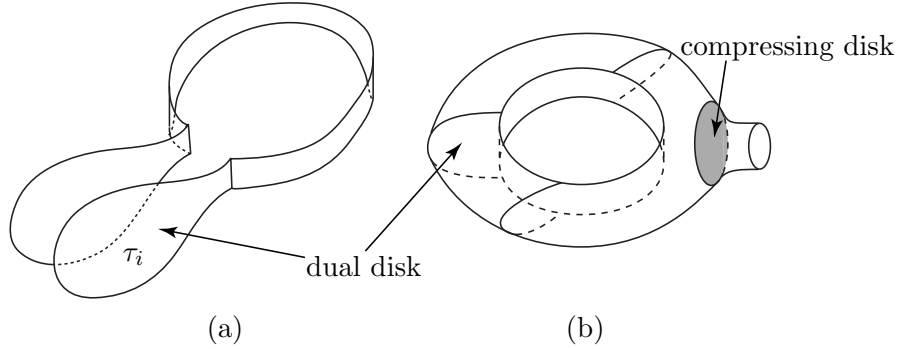
\begin{figure}[h]
	\begin{center}
		\begin{overpic}[width=4in]{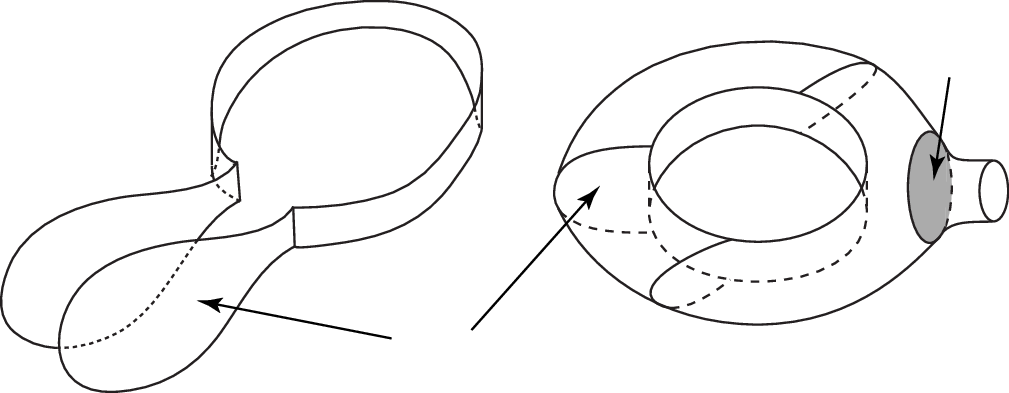}
			\put(26,-5){(a)}
			\put(73,-5){(b)}
			\put(15,5){$\tau_i$}
			\put(39,3){dual disk}
			\put(88,32){compressing disk}
		\end{overpic}
		\vspace{7pt}
		\caption{Obtain a separating compressing disk using two copies of a disk}\label{Fmono}
	\end{center}
\end{figure}

Next, we construct more surfaces in a marked handlebody $X$.

\vspace{10pt}

\noindent
\textbf{Cross-section disks, suspension surfaces, and standard surfaces}:  

\vspace{10pt}

Let $X$ be a marked handlebody with $T_i=\tau_i\times S^1$, $A_i$, $A_i'$, $B$ and $H_i$  ($i=1,\dots, m$) as above.  A {\bf cross-section disk} is a properly embedded disk in $X$ that cuts through each $H_i$ and each $T_i$, as shown in Figure~\ref{Ftype1}(a).  More precisely, a cross-section disk can be described as follows: we view the central 3-ball $B$ in the construction as a product $D\times I$, where $D$ is a disk, and view each 1-handle $H_i$ as a product $\Delta_i\times I$ with $\Delta_i\times\{0\}\subset A_i'\subset\partial T_i$ and $\Delta_i\times\{1\}\subset (\partial D)\times I\subset\partial B$.  
First take a disk $E_0=D\times\{t\}\subset D\times I=B$ and suppose $\partial E_0$ intersects each $\Delta_i\times\{1\}$ in a single arc $\delta_i\times\{1\}$, where $\delta_i$ is an arc in $\Delta_i$.  
Let $E_1$ be the union of $E_0$ and all the quadrilateral disks $\delta_i\times I$ in $H_i$ ($i=1,\dots,m$). 
In each solid torus $T_i=\tau_i\times S^1$, we take a meridional disk $\tau_i\times\{x\}$ and we suppose the intersection of the meridional disk $\tau_i\times\{x\}$ with the 1-handle $H_i$ is the arc $\delta_i\times\{0\}$.  Let $E$ be the union of $E_1$ and all the disks $\tau_i\times\{x\}$ ($i=1,\dots, m$).  So $E$ is a disk properly embedded in $X$ and we call $E$ a cross-section disk of $X$, see the shaded region in Figure~\ref{Ftype1}(a) for a picture.  

A cross-section disk $E$ has 3 parts: (1) $E_0=D\times\{t\}$ in the central 3-ball $B$, (2) the rectangles $\delta_i\times I$ in the 1-handles $H_i$ and (3) the bigons $\tau_i\times\{x\}$, $i=1,\dots,m$.  Note that by choosing $E_0=D\times\{t\}$ using different $t\in I$, we can construct a collection of parallel cross-section disks.

\begin{figure}[h]
	\begin{center}
		\begin{overpic}[width=4in]{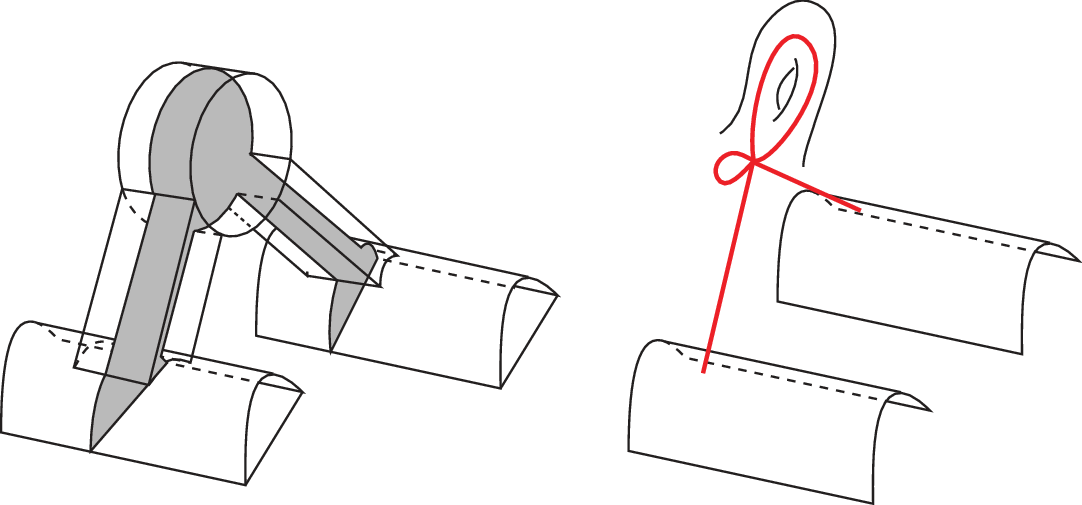}
			\put(20,-5){(a)}
			\put(78,-5){(b)}
			\put(27,37){$B$}
			\put(5,22){$H_i$}
			\put(19,8){$T_i$}
			\put(77,40){1-handle $H_j'$}
			\put(63,22){$G'$}
			\put(80,22){$\Gamma_i$}
		\end{overpic}
		\vspace{10pt}
		\caption{Cross-section disk and the construciton of a standard surface}\label{Ftype1}
	\end{center}
\end{figure}
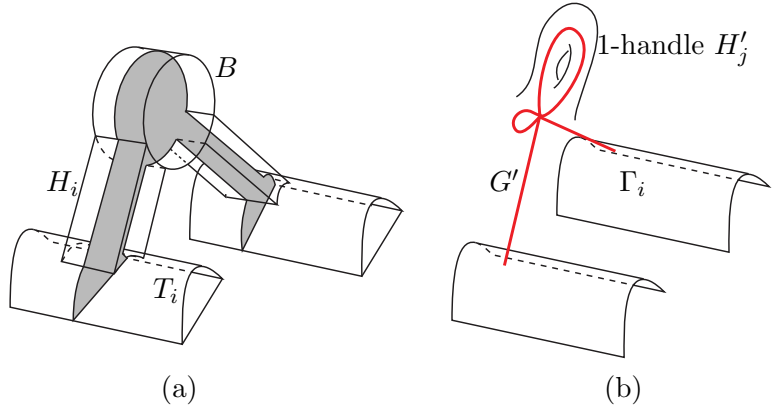

Next, we use the cross-section disk $E$ to describe a type of properly embedded planar surfaces in $X$ which we call suspension surfaces.  
Let $\Gamma_1,\dots,\Gamma_k$ be a collection of $\partial$-parallel annuli in $X$ with each $\partial\Gamma_i\subset \partial_vX$.  Let $\widehat{\Gamma}_i$ be the solid torus bounded by $\Gamma_i$ and the
subannulus of $\partial_vX$ between the two curves of $\partial\Gamma_i$.  
We say that these annuli $\Gamma_1,\dots,\Gamma_k$ are {\it non-nested} if these solid tori $\widehat{\Gamma}_1, \dots, \widehat{\Gamma}_k$ are non-nested/disjoint. 
Suppose these $\partial$-parallel annuli $\Gamma_1,\dots,\Gamma_k$ are non-nested. 
Let $E$ be a cross-section disk described above.  We may assume $E\cap \widehat{\Gamma}_i$ is a meridional disk of the solid torus $\widehat{\Gamma}_i$. 
Let $O$ be a point in the subdisk $E_0=D\times\{t\}$ of the cross-section disk $E$ and let $S$ be a 2-sphere around $O$. 
A  {\bf suspension surface} (over $\Gamma_1,\dots,\Gamma_k$ and based at the cross-section disk $E$) is the surface obtained by connecting each annulus $\Gamma_i$ to $S$  via a small tube along an arc in the cross-section disk.  We will call $S$ the {\bf central sphere} of the suspension surface.   
Alternatively, we can first connect each disk $E\cap\widehat{\Gamma}_i$ to $O$ using an arc in the cross-section disk $E$ and let $G$ be the union of these arcs.  
Let $N_\Gamma$ be a small neighborhood of $G\cup(\bigcup_{i=1}^k\widehat{\Gamma}_i)$. 
So $N_\Gamma$ is a marked handlebody whose marked annuli are the subannuli of $\partial_vX$ bounded by $\partial\Gamma_1,\dots, \partial\Gamma_k$.  
The frontier surface of $N_\Gamma$ in $X$ is a suspension surface over $\Gamma_1,\dots,\Gamma_k$.  In particular, a suspension surface is a planar surface properly embedded in $X$ with boundary in $\partial_vX$.  Moreover, we call $N_\Gamma$ the marked handlebody bounded by the suspension surface.
Note that topologically $X\setminus N_\Gamma$ is also a handlebody.

Recall that the marked handlebody $X$ in the definition has $p$ extra 1-handles $H_1',\dots, H_p'$ attached to the central 3-ball $B$, which makes $\partial_hX$ a genus-$p$ surface.  Let $F$ be a suspension surface constructed above.  By the construction, $F$ is a planar surface.  Next, we add some tubes to $F$ and change $F$ into a nonplanar surface. Each tube is either a trivial tube or a tube going through some 1-handle $H_j'$ and we require that each 1-handle $H_j'$ contains at most one tube.  
First, view the genus-$p$ handlebody $(\bigcup_{i=1}^pH_i')\cup B$ as a product $P\times I$, where $P$ is the planar surface obtained by adding $p$ 2-dimensional 1-handles to the disk $D$. 
Then extend the cross-section disk $E$ to a planar surface $E'$ by extending $E_0=D\times\{t\}$ to $P\times\{t\}$.   
Now extend the graph $G\subset E$ (in the construction above) to a graph $G'\subset E'$ by first adding a collection of trivial loops to the point $O$ in $E_0=D\times\{t\}$ and then adding $q$ loops ($q\le p$) in $P\times\{t\}$ going through the 1-handles $H_1',\dots, H_p'$, such that each $H_j'$ contains at most one arc, see Figure~\ref{Ftype1}(b) for a picture.  Let $N_\Gamma'$ be a small neighborhood of $G'\cup(\bigcup_{i=1}^k \widehat{\Gamma}_i)$.  Clearly, $N_\Gamma'$ is a marked handlebody.  We call the frontier surface of $N_\Gamma'$ in $X$ a \textbf{standard surface} in $X$ and call  $N_\Gamma'$ the marked handlebody bounded by the standard surface.  
Note that $X\setminus N_\Gamma'$ is also a topological handlebody.

\vspace{10pt}

\noindent
\textbf{Relative compression bodies}:

\vspace{10pt}

The collection of surfaces $\Sigma$ in Lemma~\ref{LBSS} divide $M=\mathcal{W}\cup_\mathcal{T} \mathcal{V}$ into a collection of compression bodies. The intersection of the torus $\mathcal{T}$ with each compression body is a collection of annuli and these annuli divide each compression body into a collection of submanifolds.  
These submanifolds have similar properties to compression bodies, plus some additional structures.

A compression body is a manifold obtained by adding 2- and 3-handles on the same side of $F\times I$, where $F$ is a closed and orientable surface.  Next, we allow the surface $F$ to have boundary and add some additional structures.

\begin{definition}\label{Drelative compression body}
Let $F$ be a compact connected and orientable surface.  Let $X$ be the 3-manifold obtained by adding 2- and 3-handles to $F\times I$ along $F\times\{0\}$.   Denote $(\partial F)\times I$ by $\partial_v^0 X$.  
There are two sets of disjoint annuli in $F\times\{1\}$, denoted by $\partial_v^+ X$ and $\partial_v^-X$  (it is possible that $\partial_v^\pm X=\emptyset$), with the following properties:  
\begin{enumerate}
  \item For each annulus $A$ in $\partial_v^-X$, there is a compressing disk $D_A$ such that $\partial D_A\subset F\times\{1\}$, $A\cap\partial D_A$ in a single essential arc in $A$, and $\partial D_A$ does not intersect any other annulus in $\partial_v^-X$.  We call $D_A$ a \textbf{dual disk} of $A$. 
  \item Each annulus $A$ in $\partial_v^\pm X$ is assigned a numerical order, denoted by $o(A)$.  A dual disk $D_A$ of an annulus $A$ in $\partial_v^-X$ may intersect annuli in $\partial_v^+X$ but it  intersects only annuli with order smaller than $o(A)$.  
\end{enumerate}
We call the manifold $X$ described above a \textbf{relative compression body}. Let $\partial_v X=\partial_v^0X\cup\partial_v^+X\cup\partial_v^-X$ and we call $\partial_vX$ the \textbf{vertical boundary} of $X$.   
Let $\partial_h^+X$ be the closure of $F\times\{1\}\setminus(\partial_v^+X\cup\partial_v^-X)$ and let $\partial_h^-X$ be the closure of $\partial X\setminus(F\times\{1\}\cup \partial_v^0X)$.  
Denote $\partial_hX=\partial_h^+X\cup\partial_h^-X$.  We call $\partial_hX$ the \textbf{horizontal boundary} of $X$.  Clearly $\partial X=\partial_hX\cup\partial_vX$.  Denote $\partial^+X=\partial_h^+X\cup\partial_v^+X$ and call $\partial^+X$  the \textbf{positive boundary} of $X$.  Similarly, denote $\partial^-X=\partial_h^-X\cup\partial_v^-X$ and call $\partial^-X$ the \textbf{negative boundary} of $X$.  Thus $\partial X=\partial^+X\cup\partial^-X\cup\partial_v^0X$. 
\end{definition} 
Figure~\ref{Fex}(b) is a picture of a relative compression body, where $A_1$ and $A_2$ are annuli in $\partial_v^-X$, $\tau_i$ is the dual disk of $A_i$ ($i=1,2$), $A_3$ is an annulus in $\partial_v^+X$, and $o(A_3)<o(A_i)$ for both $i=1,2$. In Figure~\ref{Fex}(b), $\partial_v^0X=\emptyset$ and $\partial_h^-X=\emptyset$.

To put this in a broader context, one should think of $\partial_hX$ as $\Sigma\cap \mathcal{W}$, where $\Sigma$ and $\mathcal{W}$ are as in Lemma~\ref{LBSS}, and view
 $\partial_v X$ as annuli in the torus $\mathcal{T}$. Moreover, one should view $\partial_h^+X$ and $\partial_h^-X$ as subsurfaces of the plus and minus boundary of a compression body determined by $\Sigma$.

\begin{remark}\label{Rproperty}
Below are some basic properties of a relative compression body. These properties are similar to the properties of a compression body.
\begin{enumerate}
  \item Let $\Delta$ be a compressing disk for $\partial_h^+X$. By cutting and pasting, we can construct a set of dual disks for annuli in $\partial_v^-X$ that are disjoint from $\Delta$. This means that if we compress $X$ along $\Delta$, the resulting manifold is still a relative compression body.
 \item Although we do not require all the dual disks to be disjoint in Definition~\ref{Drelative compression body}, similar to Property (1), we can perform cutting and pasting on the dual disks and obtain a set of disjoint dual disks for the annuli in $\partial_v^-X$.
 
 \item  In Definition~\ref{Drelative compression body}, each annulus $A$ in $\partial_v^-X$ has a dual disk $D_A$.  As shown in Figure~\ref{Fmono}(a), by connecting two parallel copies of $D_A$ using a band around $A$, we get a compressing disk $D$ of $X$ with $\partial D\subset\partial^+X$, such that $D$ cuts off a collar neighborhood of $A$ in $X$, see the shaded disk of Figure~\ref{Fmono}(b) for a picture of $D$.   Thus, similar to a compression body, if we maximally compress the positive boundary $\partial^+X$, the resulting manifold is a product neighborhood of the negative boundary $\partial^-X=\partial_h^-X\cup\partial_v^-X$.  

 \item  As a converse to Property (3), there is a core graph $G$ properly embedded in $X$ connecting all the components of $\partial^-X$ (similar to a core graph of a handlebody or compression body), such that $X\setminus N(G)$ is a product neighborhood of $\partial^+X$.  Moreover, for any nontrivial subgraph $G'$ of a core graph, $X\setminus N(G')$ is a relative compression body. 
  \item If we add a 1-handle to $\partial_h^+X$ or add a 2-handle to $\partial_h^-X$, the resulting manifold is still a relative compression body.  Moreover, if we add a 2-handle along any component of $\partial_v^-X$, the resulting manifold is still a relative compression body. 
 
  \item Let $A$ be a component of $\partial_v^0 X$. If one adds a 2-handle to $X$ along $A$, then the resulting manifold is still a relative compression body.  Conversely, since maximally compressing $\partial^+X$ results in a product neighborhood of $\partial^-X$, for each component of $\partial_h^-X$, there is a vertical arc $\beta$ (of the product structure above) connecting this component of $\partial_h^-X$ to $\partial_h^+X$, such that  $X\setminus N(\beta)$ is a relative compression body and the annulus around $\beta$ is a component of $\partial_v^0(X\setminus N(\beta))$.
 
 \item Let $\gamma$ be a properly embedded $\partial$-parallel arc in $X$.  Suppose $\partial\gamma\subset\partial_h^+ X$ and $\gamma$ is parallel to an arc $\beta$ in $\partial^+X$. So $\gamma\cup\beta$ bounds a disk.  Let $\mathcal{N}(\gamma)=\Delta\times I$ be a tubular neighborhood of $\gamma$ in $X$ and denote $A_\gamma=(\partial\Delta)\times I$.  Then the closure of $X\setminus \mathcal{N}(\gamma)$ is a relative compression body, if we set $A_\gamma$  as a component of $\partial_v^-(X\setminus \mathcal{N}(\gamma))$ and assign $A_\gamma$ an  order higher than the order of any other annulus (the disk bounded by $\gamma\cup\beta$ determines a dual disk for $A_\gamma$).
\end{enumerate}  
\end{remark}

\begin{lemma}\label{LBM1}
	Let $X$ be a marked handlebody.  Assign a plus sign to $\partial_hX$ and an arbitrary $\pm$-sign to each annulus in $\partial_vX$. Assign an arbitrary order to each annulus in $\partial_vX$.  Then $X$ is a relative compression body.
\end{lemma}
\begin{proof}
	Let $A_1,\dots,A_m$ be the components of $\partial_vX$.  In our definition, 
	$X$ contains a collection of compressing disks $\tau_1,\dots,\tau_m$ such that 
	$\tau_i\cap A_j=\emptyset$ if $i\ne j$, and $\tau_i\cap A_i$ is an essential arc of $A_i$. 
	If $A_i$ has a minus sign, then $\tau_i$ is a dual disk disjoint from all other annuli in $\partial_vX$.
	In particular,  $\tau_i$ is a dual disk for $A_i$ no matter what orders these annuli have.  
	Hence $X$ is a relative compression body.  
\end{proof}

\begin{lemma}\label{Lconnected}
For any relative compression body $X$, $\partial^+X$ is connected.
\end{lemma}
\begin{proof}
In Definition~\ref{Drelative compression body}, $\partial^+X$ is the closure of $F\times\{1\}\setminus\partial_v^-X$ and the surface $F$ is connected.  The existence of the dual disks for each annulus in $\partial_v^-X$ implies that $\partial_v^-X$ is non-separating in  $F\times\{1\}$. Thus $\partial^+X$ is connected.
\end{proof}

\begin{definition}\label{Dstack}
Let $X_1,\dots, X_n$ be a collection of relative compression bodies. Suppose that we can glue $X_1,\dots, X_n$ together by identifying some components of $\partial_h^+X_1,\dots, \partial_h^+X_n$ in pairs and some components of  $\partial_h^-X_1,\dots, \partial_h^-X_n$ in pairs via surface homeomorphisms.  We would like to emphasize that we only identify positive (resp.~negative) boundary to positive (resp.~negative) boundary and do not mix positive and negative boundaries. 
The resulting manifold $\widehat{X}$ is called a \textbf{stack} of relative compression bodies, or simply a stack.   
The vertical boundaries $\partial_v X_1,\dots,\partial_vX_n$ are glued into a collection of annuli and tori in $\partial \widehat{X}$. We call the union of these annuli and tori the vertical boundary of $\widehat{X}$ and denote it by $\partial_v\widehat{X}$.  
Let $\partial_h\widehat{X}$ be the closure of $\partial\widehat{X}\setminus\partial_v\widehat{X}$, and we call $\partial_h\widehat{X}$ the horizontal boundary of $\widehat{X}$.  
So $\partial_h\widehat{X}$  consists of components of $\partial_h^\pm X_1,\dots, \partial_h^\pm X_n$ that are not identified to other components. 

The components of $\partial_h^\pm X_1,\dots, \partial_h^\pm X_n$ that are identified to other components become surfaces properly embedded in $\widehat{X}$, and we call these surfaces {\bf horizontal surfaces} in $\widehat{X}$.  Let $P$ be a horizontal surface in $\widehat{X}$. We say $P$ is a positive (resp.~negative) horizontal surface if $P$ lies in the positive (resp.~negative) boundary of some $X_i$.  
\end{definition}

\begin{definition}\label{Dgood}
Let $M=\mathcal{W}\cup_\mathcal{T} \mathcal{V}$ be an amalgamation of two compact 3-manifolds $\mathcal{W}$ and $\mathcal{V}$ along a torus $\mathcal{T}$.  Let $\Sigma$ be a collection of surfaces in $M$ transversely intersecting $\mathcal{T}$.  We say that $\Sigma$ is in \textbf{good position} with respect to $\mathcal{T}$ if 
\begin{enumerate}
  \item $\Sigma$ divides $\mathcal{W}$ and $\mathcal{V}$ into two  stacks of relative compression bodies, where $\Sigma\cap \mathcal{W}$ and $\Sigma\cap \mathcal{V}$ are the horizontal surfaces in the respective stacks,
  \item positive (resp.~negative) horizontal surfaces in $\mathcal{W}$ are glued to positive (resp.~negative) horizontal surfaces in $\mathcal{V}$ along $\mathcal{T}$. 
  \item if an annulus $A\subset\mathcal{T}$ is shared by two relative compression bodies $X\subset\mathcal{W}$ and $Y\subset\mathcal{V}$, then $A\subset\partial_v^\pm X$ if and only if $A\subset\partial_v^\mp Y$, and $A$ has the same order in both $X$ and $Y$.
\end{enumerate}
\end{definition}

\begin{lemma}\label{Ldivide}
Let $M=\mathcal{W}\cup_\mathcal{T} \mathcal{V}$ be an amalgamation of two compact irreducible 3-manifolds $\mathcal{W}$ and $\mathcal{V}$ along an incompressible torus $\mathcal{T}$.  Let $\Sigma$ be the collection of incompressible and strongly irreducible surfaces in an untelescoping of a Heegaard splitting of $M$ as in section~\ref{SGHS}.  Then, after isotopy, either 
\begin{enumerate}
\item  $\mathcal{T}$ is isotopic to a component of $\Sigma$, or 
\item $\Sigma$ is in good position with respect to $\mathcal{T}$.
\end{enumerate}
\end{lemma}
\begin{proof}
By Lemma~\ref{LBSS}, we may assume $\mathcal{T}$ intersects each compression body in a collection of incompressible annuli. 
Let $X$ be a compression body in the untelescoping. Consider the set of annuli $\mathcal{T}\cap X$.  
First, note that no annulus in $\mathcal{T}\cap X$ can have both boundary circles in $\partial_-X$.  To see this, if $A$ is a component of $\mathcal{T}\cap X$ with $\partial A\subset\partial_-X$, then since $A$ is incompressible, for any compressing disk $D$ of $\partial_+X$, one can isotope $D$ so that $D\cap A=\emptyset$.  Recall that if one maximally compresses $\partial_+X$, the resulting manifold is a product $\partial_-X\times I$.  As $\partial A\subset\partial_-X$, this means that $A$ must be a $\partial$-parallel annulus, a contradiction to Lemma~\ref{LBSS}.  Thus each annulus in $\mathcal{T}\cap X$ has at least one boundary curve in $\partial_+X$.  If an incompressible annulus has one boundary circle in $\partial_+X$ and the other boundary circle in $\partial_-X$, we call $A$ a spanning annulus (or vertical annulus) in $X$.

Let $A_1,\dots, A_n$ be the annuli in $\mathcal{T}\cap X$ with both boundary curves in $\partial_+X$, and let $C_1,\dots, C_m$ be the spanning annuli in $\mathcal{T}\cap X$ with one boundary curve in $\partial_+X$ and the other boundary curve in $\partial_-X$.  Since $X$ is a compression body and $\partial A_i\subset\partial_+X$, each $A_i$ is $\partial$-compressible in $X$.  Thus we can assign a $\partial$-compressing disk $\Delta_i$ to each $A_i$, with $\partial\Delta_i=\alpha_i\cup\beta_i$, $\beta_i\subset\partial_+X$, and $\alpha_i$ being an essential arc in $A_i$.  Note that $\Delta_i$ may intersect other annuli in $\mathcal{T}\cap X$, but we require that $\Delta_i$ has minimal intersection with $\mathcal{T}\cap X$ among all $\partial$-compressing disks for $A_i$, for each $i$.

For each annulus $E$ in $\mathcal{T}\cap X$, if $E\cap\Delta_i$ contains a closed curve, then since $E$ is incompressible, this curve must be trivial in both $E$ and $\Delta_i$.  
By compressing $\Delta_i$ along an innermost such closed curve, we get a new $\partial$-compressing disk for $A_i$ with fewer intersection curves with $E$.  Similarly, if $E\cap\Delta_i$ contains an arc that is $\partial$-parallel in $E$, then by performing a $\partial$-compression on $\Delta_i$ along an outermost such arc, we get a new $\partial$-compressing disk for $A_i$ with fewer intersection arcs with $E$.  Since $\Delta_i$ is chosen to have minimal intersection with $\mathcal{T}\cap X$, this implies that $\Delta_i\cap E$ consists of arcs essential in $E$ for every annulus $E$ in $\mathcal{T}\cap X$ and for each $i$.  
 Since $\partial\Delta_i\cap\partial_-X=\emptyset$, this implies that $\Delta_i\cap C_k=\emptyset$ for all $i$, $k$.  

For each arc $\delta$ in $\Delta_i\cap A_k$, $\delta$ is an arc in $\Delta_i$ with both endpoints in $\beta_i$. Moreover, the subarc of $\beta_i$ between the two points of $\partial\delta$, together with $\delta$, bounds a subdisk of $\Delta_i$, which we denote by $\Delta_\delta$.  We say that $A_k$ is \emph{coherent} with $\Delta_i$ if (1) these subdisks $\Delta_\delta$ for all the arcs of $\Delta_i\cap A_k$ are non-nested in $\Delta_i$ and (2) each subdisk $\Delta_\delta$  is parallel to the $\partial$-compressing disk $\Delta_k$ of $A_k$.

\begin{claim*}
These $\partial$-compressing disks $\Delta_i$ may be chosen so that, for every $k$, $A_k$ is coherent with all the $\Delta_i$.
\end{claim*}
\begin{proof}[Proof of the Claim] Consider an annulus $A_i$ and its $\partial$-compressing disk $\Delta_i$. 
We say that $A_i$ is outermost if the $\partial$-compressing disk $\Delta_i$ is disjoint from all other annuli in $\mathcal{T}\cap X$. 
Note that we can choose these $\partial$-compressing disks $\Delta_i$ so that there is at least one outermost annulus.  To see this, suppose $A_i$ is not outermost, then the intersection of $\Delta_i$ with $\mathcal{T}\cap X$ is a collection of arcs with endpoints in $\beta_i$.  Let $\delta$ be an outermost such intersection arc in $\Delta_i$.  So $\delta$ is an essential arc of some annulus $A_j$ and $\delta$ cuts off a subdisk $\Delta'$ in $\Delta_i$.  Note that $\Delta'$ is a $\partial$-compressing disk for $A_j$ and $\Delta'$ is disjoint from all other annuli in $\mathcal{T}\cap X$. 
By choosing $\Delta_j=\Delta'$, we see that $A_j$ is an outermost annulus.

Without loss of generality, suppose $A_1$ is outermost.  As shown in Figure~\ref{Fmono}(a), by connecting two parallel copies of $\Delta_1$ using a band around $A_1$, we get a compressing disk $D_1$ for $\partial_+X$, see the shaded region in Figure~\ref{Fmono}(b) for a picture of $D_1$.  Since $\Delta_1$ is disjoint from other annuli in $\mathcal{T}\cap X$, after a small perturbation, we may assume $D_1$ is disjoint from $\mathcal{T}\cap X$.  
Note that $D_1\cap\Delta_1=\emptyset$.  If $D_1$ intersects other $\Delta_i$ ($i=2,\dots,n$), consider an arc $\delta'$ in $D_1\cap(\bigcup_{i=2}^n\Delta_i)$ that is outermost in $D_1$.  
Suppose $\delta'$ is an arc in $\Delta_i$.  Then we can perform a $\partial$-compression on $\Delta_i$ along the subdisk of $D_1$ cut off by $\delta'$.  
This $\partial$-compression on $\Delta_i$ produces a new $\partial$-compressing disk $\Delta_i'$ for $A_i$.  Since $D_1$ is disjoint from $\mathcal{T}\cap X$, $\Delta_i'$ still has minimal intersection with $\mathcal{T}\cap X$ and we can replace $\Delta_i$ with $\Delta_i'$. 
 After finitely many such $\partial$-compressions, we get a new set of $\partial$-compressing disks, which we still denote by $\Delta_1,\dots \Delta_n$, such that $D_1\cap\Delta_i=\emptyset$ for all $i$.

As illustrated in Figure~\ref{Fmono}(b), 
if we compress $X$ along $D_1$, $A_1$ becomes a $\partial$-parallel annulus in the resulting compression body.  Thus, if a $\partial$-compressing disk $\Delta_i$ intersects $A_1$, then the subdisks of $\Delta_i$ cut off by $A_1$ are all parallel copies of $\Delta_1$.  This means that $A_1$ is coherent with every $\Delta_i$.

We compress $X$ along $D_1$ and get a new compression body $X_1$.  Now we can ignore the annulus $A_1$ (since $A_1$ is $\partial$-parallel in $X_1$) and consider $A_2,\dots, A_n$ and their $\partial$-compressing disks $\Delta_2,\dots, \Delta_n$ in $X_1$.  We can find an annulus (without considering $A_1$) that is outermost in $X_1$ and repeat the argument above.  Therefore, we can inductively conclude that each $A_k$ is coherent with every $\Delta_i$ for all $k$.
\end{proof}

We define an order for each $\Delta_i$ as follows.  If $\Delta_i$ is disjoint from other annuli of $\mathcal{T}\cap X$, i.e.~$A_i$ is outermost, then set the order $o(\Delta_i)$ to be $0$.  Suppose $\Delta_i$ intersects other annuli of $\mathcal{T}\cap X$. 
Then the intersection consists of arcs with endpoints in $\beta_i\subset\partial\Delta_i$.  Each arc, together with a subarc of $\beta_i$, bounds a subdisk of $\Delta_i$.  Roughly speaking, the order $o(\Delta_i)$ is the maximal number of such subdisks that are nested to one another.  
More precisely, for any point $x\in \Int(\beta_i)$, we draw an arc $\gamma_x$ in $\Delta_i$ connecting $x$ to $\alpha_i$ ($\alpha_i= A_i\cap\partial\Delta_i$), see Figure~\ref{Forder}, and count the number of intersection points of $\Int(\gamma_x)$ with the annuli $\mathcal{T}\cap X$, i.e.~ the number $|\Int(\gamma_x)\cap \mathcal{T}|$. First define $o(x)$ to be the minimal number of such intersection points amount all arcs connecting $x$ to $\alpha_i$.  This means that if $\Delta_i$ intersects other annuli of $\mathcal{T}\cap X$, there is a point $x\in \beta_i$ with $o(x)\ge 1$.  As illustrated in Figure~\ref{Forder}, we define the order of $\Delta_i$ to be $o(\Delta_i)=\max_{x\in\beta_i}\{ o(x)\}$.  
 We define the order of the annulus $A_i$ to be $o(A_i)=o(\Delta_i)$.  Furthermore, for each annulus $A_i$, we assign a normal vector pointing into the $\partial$-compressing disk $\Delta_i$.  We assign $\pm$-signs to the two sides of $A_i$ so that this normal vector pointing from the plus-side to the minus-side (i.e.~$\Delta_i$ is on the minus-side of $A_i$).

\begin{figure}[h]
	\begin{center}
		\begin{overpic}[width=1.3in]{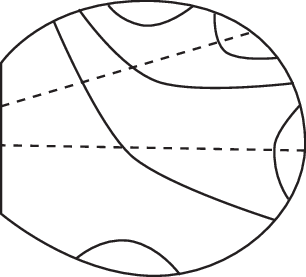}
			\put(-11,28){$\alpha_i$}
			\put(84.5,80){$x$}
			\put(101,40){$y$}
			\put(11, 64){$\gamma_x$}
			\put(18,18){$\Delta_i$}
			\put(110, 80){$o(x)=3$}
            \put(115, 40){$o(y)=2$}
            \put(110, 10){$o(\Delta_i)=3$}
		\end{overpic}
%		\vspace{3pt}
		\caption{The order of $\Delta_i$}\label{Forder}
	\end{center}
\end{figure}
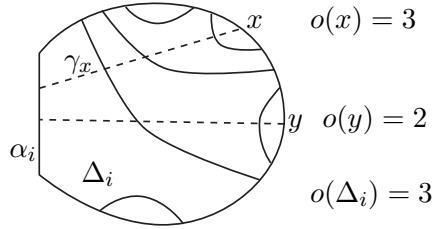

The annuli of $\mathcal{T}\cap X$ divide $X$ into submanifolds and let $N$ be one of these submanifolds.  
We will show next that each $N$ is a relative compression body.  First note that $\partial N$ has 3 parts:
\begin{enumerate}
 \item the annuli from $\mathcal{T}\cap X$, which we denote by $\partial_vN$, 
 \item a subsurface of $\partial_-X$, i.e.~$\partial N\cap\partial_-X$, denoted by $\partial_h^-N$, and 
 \item a subsurface of $\partial_+X$, i.e.~$\partial N\cap\partial_+X$, denoted by $\partial_h^+N$.  
\end{enumerate}
 We also divide the annuli in $\partial_vN$ also into 3 types: $\partial_v^0N$, $\partial_v^+N$ and $\partial_v^-N$, where 
 \begin{enumerate}
\item $\partial_v^0N$ consists of annuli connecting $\partial_+X$ to $\partial_-X$, i.e.~$\partial_v^0N$ corresponds to the spanning annuli $C_1,\dots, C_m$, 
\item $\partial_v^+N$ consists of annuli $A_i$ with normal direction pointing out of $N$, and 
\item $\partial_v^-N$ consists of annuli $A_i$ with normal direction pointing into $N$.  
 \end{enumerate}
Set the order of each annulus $A_i$ in $\partial_v^\pm N$ to be the order $o(A_i)$ defined for $\mathcal{T}\cap X$ above.  
 Each $\partial$-compressing disk $\Delta_i$ is cut into a collection of subdisks by $\mathcal{T}\cap X$, and let $\Delta_i'$ be the subdisk that contains the arc $\alpha_i=\partial\Delta_i\cap A_i$.  So, if $N$ is on the minus side of $A_i$, then the minus side of $A_i$ is an annulus of $\partial_v^-N$, $\Delta_i'\subset N$, and $\Delta_i'$ is the dual disk of $A_i$ in $N$. 
   By the claim above, we may assume each $A_k$ is coherent with each $\Delta_i$. Thus the definition of $o(A_i)$ above implies that the dual disk $\Delta_i'$ satisfies the conditions in Definition~\ref{Drelative compression body}.

Since each annulus of $\partial_v^-N$ has a dual disk, to prove that $N$ is a relative compression body, it suffices to show that, for each region $N$, after a sequence of compressions on $\partial^+N$, we obtain a product neighborhood of $\partial^-N$.  We prove this using the fact that $X$ is a compression body, which means that if we maximally compress $\partial_+X$, we obtain a product neighborhood of $\partial_-X$.

Start with $\partial$-compressing disks of order $0$. Let $\Delta_i$ be a  $\partial$-compressing disk of order $0$, i.e.~$A_i$ is an outermost annulus. Let $\widehat{\Delta}_i$ be the disk obtained by connecting two parallel copies of $\Delta_i$ using a band around $A_i$, as shown in Figure~\ref{Fmono}(a).  
As in the proof of the claim, we may assume that, after isotopy, $\widehat{\Delta}_i$ does not intersect the $\partial$-compressing disks.  

Instead of considering $N$ itself, we consider all the components of $\overline{X\setminus \mathcal{T}}$ at the same time. 
Denote the two components of $\overline{X\setminus \mathcal{T}}$ on the plus and minus sides of $A_i$ by $N_+$ and $N_-$ respectively.  
So the two sides of annulus $A_i$ can be viewed as annuli in $\partial_v^+N_+$ and $\partial_v^-N_-$.  Moreover $\widehat{\Delta}_i\subset N_-$ and $\widehat{\Delta}_i$ is a compressing disk for both $\partial_h^+ N_-$ and $\partial_+X$.  

Now compress $N_-$ along $\widehat{\Delta}_i$.  As in the proof of the Claim, after the compression along $\widehat{\Delta}_i$, $A_i$ becomes a $\partial$-parallel annulus in the resulting manifold.  So the compressing disk $\widehat{\Delta}_i$ divide $N_-$ into a collar neighborhood of $A_i$, which we denote by $E_i$, and a submanifold with fewer vertical boundary components.  Enlarge $N_+$ by including this solid torus $E_i$ into $N_+$ and then delete the annulus $A_i$.  Since $E_i$ is a collar neighborhood of $A_i$, this operation does not really change $N_+$.  The effect of this operation on $\partial N_+$ is equivalent to merging $A_i$ from $\partial_v^+N_+$ into $\partial_h^+N_+$.  We do these operations on all the disks $\Delta_i$ of order 0.  
Note that the compressions on $N_-$ are also compressions on $\partial_+X$, and  these compressions change $X$ into a new (possibly disconnected) compression body, and we can consider the remaining annuli $A_j$ in the new compression body.

Next, consider $\partial$-compressions disks $\Delta_j$ of order $1$.  Since the annuli of order $0$ are deleted, the disks $\Delta_j$ of order $1$ do not intersect other annuli.   Thus we can apply the same operations using the disks $\Delta_j$ of order 1, i.e., first compress the compression body along the disk illustrated in Figure~\ref{Fmono}, and then remove the annuli $A_j$ of order 1.
We can inductively repeat this operation using these $\partial$-compressing disks.   
After all these compressions and deleting the annuli $A_k$'s, the remaining annuli in the resulting (possibly disconnected) compression body are a collection of spanning annuli, i.e.~the annuli $C_1,\dots C_m$.  So we can perform more compressions disjoint from the spanning annuli $C_i$ and change the compression body into a product $\partial_-X\times I$.  
By restricting these operations on each component $N$ of $\overline{X\setminus \mathcal{T}}$, we see that if one maximally compresses $\partial^+N$, the resulting manifold is $\partial^- N\times I$.  Therefore, $N$ satisfies all the requirements of Definition~\ref{Drelative compression body} and is a relative compression body. 
Moreover, it follows from our construction that all the conditions in Definition~\ref{Dgood} are satisfied.
\end{proof}

Next, we prove a converse to Lemma~\ref{Ldivide}.

\begin{lemma}\label{Lmerge}
Let $\widehat{\mathcal{W}}$ and $\widehat{\mathcal{V}}$ be compact 3-manifolds with torus boundary. Suppose there are collections of horizontal surfaces dividing $\widehat{\mathcal{W}}$ and $\widehat{\mathcal{V}}$ into stacks of relative compression bodies.  Suppose the boundary curves of these horizontal surfaces of $\widehat{\mathcal{W}}$ and $\widehat{\mathcal{V}}$ divide the tori $\partial \mathcal{W}$ and $\partial \mathcal{V}$ into the same number of annuli $w_1,\dots,w_k$ and $v_1\dots, v_k$ respectively.  
Let $W_i$ and $V_i$ be the relative compression bodies in the stacks $\widehat{\mathcal{W}}$ and $\widehat{\mathcal{V}}$ containing $w_i$ and $v_i$ respectively (it is possible that $W_i=W_j$ and $V_i=V_j$ for different $i$, $j$).
Suppose that 
\begin{enumerate}
  \item  $w_i\subset\partial_v^\pm W_i$ if an only if $v_i\subset\partial_v^\mp V_i$, and $o(w_i)=o(v_i)$ for all $i$.
  \item  $w_i\subset\partial_v^0 W_i$ if and only if $v_i\subset\partial_v^0 V_i$.
\end{enumerate}  
Let $M$ be the closed 3-manifold obtained by gluing $\widehat{\mathcal{W}}$ to $\widehat{\mathcal{V}}$ and identifying $w_i$ to $v_i$ for all $i$, and suppose positive (resp.~negative) horizontal surfaces in $\widehat{\mathcal{W}}$ are glued to positive (resp.~negative) horizontal surfaces of $\widehat{\mathcal{V}}$.  Then the union of the horizontal surfaces of $\widehat{\mathcal{W}}$ and $\widehat{\mathcal{V}}$ divides $M$ into a collection of compression bodies.
\end{lemma}
\begin{proof}
The proof of this lemma is similar to the latter part of the proof of Lemma~\ref{Ldivide}. First note that the positive and negative horizontal surfaces in $\widehat{\mathcal{W}}$ and $\widehat{\mathcal{V}}$ match up and yield a collection of positive and negative surfaces in $M$.  The relative compression bodies in $\widehat{\mathcal{W}}$ and $\widehat{\mathcal{V}}$ also match up and become a collection of regions between these positive and negative surfaces.  
Denote these regions by $\mathcal{N}_1,\dots \mathcal{N}_k$.  We use $\partial_+\mathcal{N}_i$ and $\partial_-\mathcal{N}_i$ to denote the unions of the components of $\partial \mathcal{N}_i$ that are positive and negative surfaces respectively.  By the definition of relative compression body, $\partial_+\mathcal{N}_i\ne\emptyset$ for each $i$.

Our goal is to show that each $\mathcal{N}_i$ is a compression body and $\partial_+\mathcal{N}_i$ and $\partial_-\mathcal{N}_i$ are its plus and minus boundaries respectively.  

We view $M=\widehat{\mathcal{W}}\cup_{\mathcal{T}}\widehat{\mathcal{V}}$, where $\mathcal{T}= \partial\widehat{\mathcal{W}}=\partial\widehat{\mathcal{V}}$ is a torus in $M$. 
Let $X_1,\dots,X_m$ be the relative compression bodies of the stacks $\widehat{\mathcal{W}}$ and $\widehat{\mathcal{V}}$ that lie in $\mathcal{N}_i$.  
For each annulus $A$ of $\mathcal{T}\cap\mathcal{N}_i$, 
the two sides of $A$ are viewed as annuli in $\partial_v X_1,\dots,\partial_vX_m$. By the hypotheses, if one side of $A$ is a component of $\partial_v^\pm X_i$, then the other side of $A$ must be a component of $\partial_v^\mp X_j$ with the same order, similar to the conditions in Definition~\ref{Dgood}.

We first consider the annuli in $\partial_v^-X_j$ ($j=1,\dots,m$) and their dual disks.  Define the order $o(\Delta_i)$ of each dual disk $\Delta_i$ to be the order of the corresponding annulus.  Suppose $\Delta_1,\dots,\Delta_k$ are all the dual disks in the $X_i$'s and suppose $o(\Delta_1)\le o(\Delta_2)\le\cdots\le o(\Delta_k)$.  Since $\Delta_1$ has the smallest order among all dual disks in $\mathcal{N}_i$, and by the hypothesis that $o(w_j)=o(v_j)$ for all $j$ (i.e.~the orders of the two sides of the same annulus are the same), $\partial\Delta_1$ does not meet any other annulus in $\mathcal{T}\cap\mathcal{N}_i$.   
Hence, by connecting two parallel copies of $\Delta_1$ using a band as in Figure~\ref{Fmono}(a), we obtain a compressing disk $D_1$ for $\partial_+\mathcal{N}_i$ and $D_1$ is disjoint from the torus $\mathcal{T}$.  Similar to the proof of Lemma~\ref{Ldivide}, we may assume $D_1\cap\Delta_j=\emptyset$ for all $j$. Moreover, we can compress $\partial_+\mathcal{N}_i$ along $D_1$ and then push the resulting annulus that is parallel to $A_1$ to the other side of the torus $\mathcal{T}$.  
Since $A_1$ is eliminated from the compression body after this operation, $\partial\Delta_2$ does not meet any other annulus and we can repeat this operation using $\Delta_2$. 
So we successively repeat this operation using the disks $\Delta_2,\dots,\Delta_k$, and denote the resulting relative compression bodies by $X_1',\dots, X_m'$.   So, after we finish these compressions and isotopies, $\partial_v^\pm X_i'=\emptyset$ for all $i$. 

Similar to the proof of Lemma~\ref{Ldivide}, we can perform more compressions on each $\partial_h^+X_i'$ changing the manifold into a product $\partial_h^-X_i'\times I$. 
These products match up along $\mathcal{T}$ and yield a product $\partial_-\mathcal{N}_i\times I$.   This means that if we maximally compress $\partial_+\mathcal{N}_i$ in $\mathcal{N}_i$, the resulting manifold is $\partial_-\mathcal{N}_i\times I$.  Hence each $\mathcal{N}_i$ is a compression body.
\end{proof}

We end this section with certain conditions that describe the boundary of a relative compression body.

\begin{definition}\label{Dadmissible}
Let $S$ be a closed orientable surface.  Suppose $S$ contains a collection of  annuli, which we denote by  $\partial_vX$. Denote the closure of $S\setminus\partial_vX$ by  $\partial_hX$.  Assign each component of $\partial_hX$ a plus or minus sign and use $\partial_h^\pm X$ to denote the union of the components with $\pm$-sign. Assign each component of $\partial_vX$ a $\pm$-sign or no sign. Denote the union of the annuli in $\partial_vX$ with $\pm$-sign by $\partial_v^\pm X$, and denote the union of the annuli in $\partial_vX$ with no sign by  $\partial_v^0X$.  Moreover, we give each annulus $A$ in $\partial_v^\pm X$ a numerical order, denoted by $o(A)$.  Let $\partial^\pm X=\partial_h^\pm X\cup\partial_v^\pm X$.  We say that these subsurfaces of $S$ with the signs and orders are \textbf{admissible} if
\begin{enumerate}
  \item there is no annulus in $\partial_vX$ with both boundary circles in $\partial_h^-X$,
  \item each annulus in $\partial_v^0X$ has one boundary circle in $\partial_h^+X$ and the other boundary circle in $\partial_h^-X$,
  \item all the annuli in $\partial_v^+X$ and $\partial_v^-X$ have boundary curves in $\partial_h^+X$,
  \item  $\partial^+X$ is connected.
\item  For each annulus $A$ in $\partial_v^-X$, there is an arc $\alpha$ in $\partial^+X$ that connects the two boundary circles of $A$, such that $\alpha$ only intersects annuli in $\partial_v^+X$ with order smaller than $o(A)$.
\end{enumerate}  
\end{definition}

\begin{lemma}\label{Ladmissible}
Let $X$ be a relative compression body.  Then $\partial X$ and the signs and orders in Definition~\ref{Drelative compression body} are admissible.
\end{lemma}
\begin{proof}
By Lemma~\ref{Lconnected}, $\partial^+X$ is connected, so condition (4) of Definition~\ref{Dadmissible} is satisfied.  For any annulus $A$ in $\partial_v^-X$, let $D$ be its dual disk and let $\alpha=\partial D\cap\partial^+X$.  By our definition, $\alpha$ satisfies  condition (5) of Definition~\ref{Dadmissible}.  Other conditions of Definition~\ref{Dadmissible} follow directly from Definition~\ref{Drelative compression body}.
\end{proof}

In the later sections, for a certain surface $S$ with admissible decompositions as in Definition~\ref{Dadmissible}, we will construct a relative compression body $X$ with $\partial X=S$ and compatible boundary structure as in Lemma~\ref{Ladmissible}.

\section{Type I blocks}\label{StypeI}

In this section, we describe some building blocks of the construction.  Each building block is a topological handlebody with a special structure.  

\vspace{10pt}

\noindent
\textbf{Type I blocks}:  

\vspace{10pt}

Let $X$ be a marked handlebody as above.  Let $F_1,\dots, F_n$ be a collection of mutually disjoint standard surfaces and let $N_i$ be the marked handlebody bounded by $F_i$ (see the definition of standard surface in section~\ref{Srelative compression body}). 
Let $W$ be the closure of $X\setminus\cup_{i=1}^n N_i$ and we call $W$ a {\bf type I block}.  We call $(\cup_{i=1}^n F_i)\cup\partial_hX$ the horizontal boundary of $W$ and denote it by $\partial_h  W$.   
The vertical boundary $\partial_v W$ is the closure of $\partial X\setminus\partial_h W$. So $\partial_v W$ is a collection of subannuli of $\partial_vX$.

\vspace{10pt}

\noindent
\textbf{A basic construction:}

\vspace{10pt}

Let $X$ be a marked handlebody and let $D_1,\dots, D_k$ be a collection of parallel cross-section disks in $X$.  Let $\Gamma_1,\dots,\Gamma_q$ be a collection of disjoint non-nested $\partial$-parallel annuli in $X$ with boundary circles in $\partial_vX$.  We divide the annuli $\Gamma_1,\dots,\Gamma_q$ into $k$ disjoint sets of annuli $\widetilde{\Gamma}_1,\dots, \widetilde{\Gamma}_k$.  Let $F_i$ be the suspension surface over the set of annuli $\widetilde{\Gamma}_i$ and based at the cross-section disk $D_i$ ($i=1,\dots, k$).  
By taking disjoint cross-sections disks, we may assume that $F_1,\dots, F_k$ are disjoint. 
$F_1,\dots F_k$ divide $X$ into $k+1$ submanifolds $W$, $X_1,\dots X_k$, where 
each $X_i$ is the marked handlebody bounded by $F_i$ and $W$ is the type I block between $\partial_hX$ and these $F_i$'s.  
The boundary of $W$ has two parts: 
(1) the horizontal boundary $\partial_hW=\partial_hX\cup F_1\cup\cdots\cup F_k$,  and (2) the vertical boundary $\partial_vW$ consisting of subannuli of $\partial_vX$.

Consider the type I block $W$.  Assign each $F_i$ a plus sign and assign $\partial_hX$ either a plus or a minus sign.  We give each annulus in $\partial_vW$ a numerical order and either a $\pm$-sign or no sign.  We define $\partial_h^\pm W$ and $\partial_v^\pm W$ to be the union of components of $\partial_hW$ and $\partial_vW$ respectively with $\pm$-sign, and let $\partial_v^0W$ be the union of annuli in $\partial_vW$ with no sign.   We require that these surfaces with these signs and orders are admissible (see Definition~\ref{Dadmissible}).

Given any set of annuli in $\partial_v^+W$, we can first arrange the orders of the annuli in this set into a non-decreasing list and then consider the lexicographic order of this list.  
This gives an order on all subsets of annuli in $\partial_v^+W$.  
By condition (4) of Definition~\ref{Dadmissible}, $\partial^+W$ is connected. So there is a set of annuli in $\partial_v^+W$ connecting all the components of $\partial_h^+W$ together. 
We call a set of annuli in $\partial_v^+W$ a {\bf minimal set of annuli} connecting $\partial_h^+W$ if 
\begin{enumerate}
  \item  the union of $\partial_h^+W$ and the annuli in this set is connected and 
  \item  this set of annuli has minimal order among all such sets of annuli.
\end{enumerate}  
One may find it helpful to view the components of $\partial_h^+W$ as vertices and view the annuli in $\partial_v^+ W$ as edges connecting these vertices.  So a minimal set of annuli is necessarily a set of edges that connect these vertices into a tree. Figure~\ref{Fbasic} is a picture of two suspension surfaces $F_1$ and $F_2$ connected by an annulus $A_1$. 

\vspace{10pt}

We have two situations depending on the sign of $\partial_h X$:

If $\partial_h X$ has a minus sign, then $\partial_h^+W=F_1\cup\cdots\cup F_k$. 
As $\partial^+ W$ is connected, there is a minimal set of $k-1$ annuli $\{x_1,\dots, x_{k-1}\}$ in $\partial_v^+W$ connecting $F_1,\dots, F_k$ together.  
We may name the annuli $x_i$'s and the surfaces $F_1,\dots, F_k$ (or arrange the subscripts) so that
\begin{enumerate}
  \item $x_1$ has the smallest order among all the annuli in $\partial_v^+W$ that connect two distinct surfaces in $\{F_1,\dots,F_k\}$, and assume $F_1$ is attached to $x_1$. 
  \item $x_i$ connects $F_1\cup\cdots\cup F_i$ to $F_{i+1}$ for all $i=1,\dots, k-1$ 
  \item for each $i=1,\dots, k-1$,  $x_i$ has the smallest order among the annuli in the set $\{x_1,\dots, x_{k-1}\}$ that connect $F_1\cup\cdots\cup F_i$ to some $F_j$ with $j>i$.
\end{enumerate}

If $\partial_h X$ has a plus sign, then $\partial_h^+W=\partial_hX\cup F_1\cup\cdots\cup F_k$.  
So there is a minimal set of $k$ annuli $\{x_1,\dots, x_{k}\}$ in $\partial_v^+W$ connecting $\partial_hX, F_1,\dots, F_k$ together.  
Similar to the case above, we may name the annuli $x_i$'s and the surfaces $F_1,\dots, F_k$ (or arrange the subscripts) so that   
\begin{enumerate}[(i)]
  \item $x_1$ has the smallest order among all the annuli in $\partial_v^+W$ that connect two distinct surfaces in $\{\partial_hX, F_1,\dots,F_k\}$, and assume $F_1$ is attached to $x_1$.
  \item Let $\widehat{G}_i$ be the subset of the surfaces in $\{\partial_h X, F_1,\dots, F_k\}$ attached to the annuli $x_1,\dots, x_{i}$.  
The union of $x_1,\dots, x_{i}$ and the surfaces in $\widehat{G}_i$ is connected for each $i$, and the indices/subscripts are chosen so that, for any $F_p\in \widehat{G}_i$ and $F_q\notin \widehat{G}_i$,  $p<q$.
  \item for each $i=2,\dots, k$,  $x_i$ has the smallest order among all the annuli in $\{x_1,\dots, x_{k}\}$ that connect a surface in $\widehat{G}_{i-1}$  to a surface in $\{\partial_hX, F_1,\dots,F_k\}\setminus\widehat{G}_{i-1}$.
\end{enumerate}

In both cases, we say that these surfaces and cross-section disks are \textbf{well-positioned} if these $F_i$'s and $x_i$'s satisfy the respective conditions above and the cross-section disks $D_i$ and $D_{i+1}$ are adjacent for each $i$. 
Note that one can construct a suspension surface $F_i$ using an arbitrary cross-section disk, and different choices of cross-section disks do not affect whether or not the surfaces in $\partial W$ (with the fixed signs and orders of the $F_i$'s and $x_i$'s) are admissible, see Definition~\ref{Dadmissible}.

\begin{figure}[h]
	\begin{center}
		\begin{overpic}[width=3in]{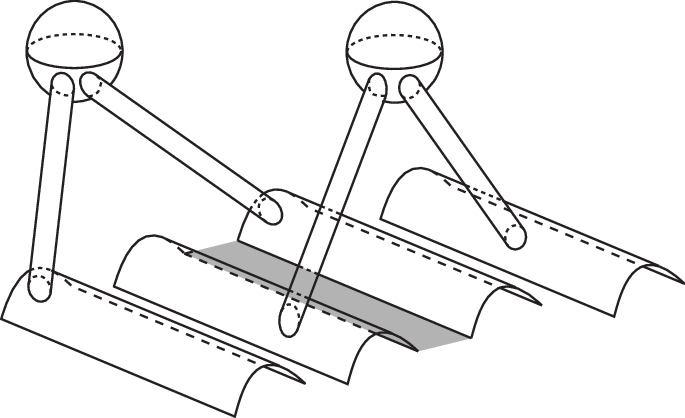}
			\put(18,55){$F_1$}
			\put(64,55){$F_2$}
			\put(63,8){$A_1$}
		\end{overpic}
%		\vspace{3pt}
		\caption{Connect two suspension surfaces with an annulus}\label{Fbasic}
	\end{center}
\end{figure}

\begin{lemma}\label{LW2}
Let $X$, $W$, $F_1,\dots, F_k$ be as above. 
In particular, each $F_i$ is a suspension surface in $X$ with a plus sign.  
Suppose $\partial_hX$ is a planar surface and has a minus sign. Suppose the signs and orders of the surfaces in $\partial W$ are admissible. Let $\{x_1,\dots, x_{k-1}\}$ be a minimal set of annuli connecting $F_1,\dots, F_k$ and  suppose these surfaces and the cross-section disks are well-positioned.  Then $W$ is a relative compression body with respect to these signs and orders.
\end{lemma}
\begin{proof}
Since $\partial_hX$ has a minus sign, $\partial_h^+W=\bigcup_{i=1}^k F_i$. 
Let $G_i$ be the union of $x_1,\dots, x_i$ and $F_1,\dots, F_{i+1}$. By  condition (2) on the $F_i$'s and $x_i$'s above, $G_i$ is a connected surface for each $i$.

Let $A$ be any annulus in $\partial_v^-W$.  
We first construct a dual disk for $A$. 
If both components of $\partial A$ lie in the same surface $F_j$, then $A$ has a dual disk $\Delta_A$, see the shaded disk in Figure~\ref{Fdual}(a), and $\partial \Delta_A$ does not intersect any annulus in $\partial_v^+W$.

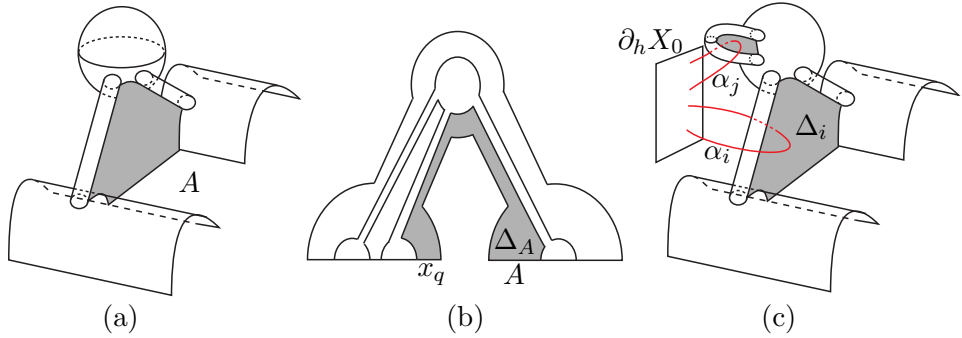
\begin{figure}[h]
\begin{center}
\begin{overpic}[width=5in]{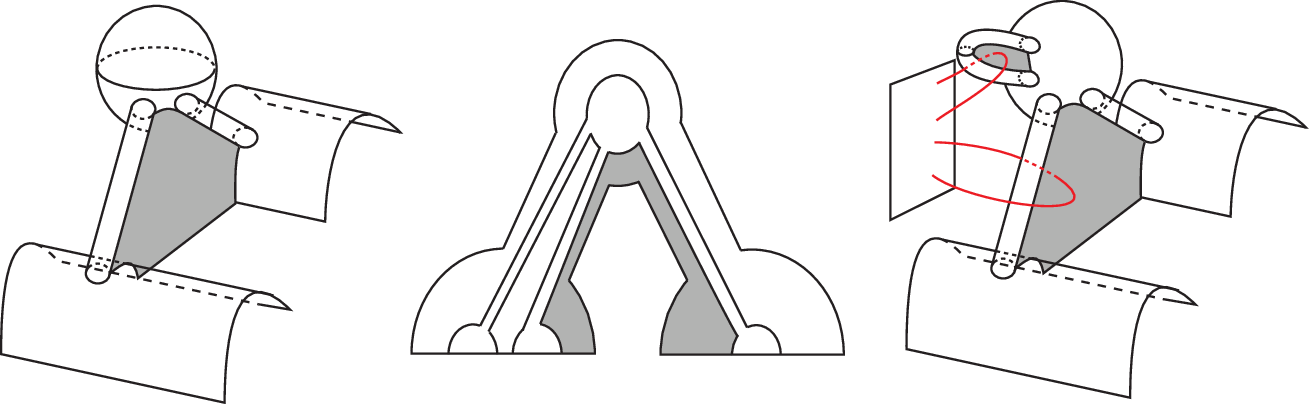}
\put(10,-3){(a)}
\put(46,-3){(b)}
\put(79,-3){(c)}
\put(18,11){$A$}
\put(52,1.4){$A$}
\put(51, 5){$\Delta_A$}
\put(43,2.2){$x_q$}
\put(82.5, 17){$\Delta_i$}
\put(73, 14.5){$\alpha_i$}
\put(73.8, 22){$\alpha_j$}
\put(64, 26){$\partial_h X_0$}
\end{overpic}
\vspace{6pt}
\caption{Dual disks}\label{Fdual}
\end{center}
\end{figure}

Suppose the two components of $\partial A$ lie in two different surfaces $F_s$ and $F_t$ with $s<t$.  By the construction of $G_i$, $\partial A\subset \partial G_{t-1}$.   We have two cases to discuss:

\vspace{5pt}

\noindent
{\it Case (1)}.  The order $o(A)$ is larger than the orders of all the annuli $x_1,\dots, x_{t-1}$ in $G_{t-1}$

\vspace{5pt}

Recall that each $F_i$ is constructed by connecting a collection of $\partial$-parallel annuli to its central 2-sphere along arcs in the cross-section disk $D_i$.  Denote the central  2-sphere of $F_i$ by $S_i$. 

Next we perform some isotopies on $W$.  Start with $F_1$, $F_2$, and $x_1$. Recall that $x_1$ connects $F_1$ to $F_2$.  
Let $\Gamma_1$ and $\Gamma_2$ be the two $\partial$-parallel annuli in the construction of $F_1$ and $F_{2}$ respectively that are attached to the annulus $x_1$.  So $\Gamma'=\Gamma_1\cup x_1\cup\Gamma_2$ is an annulus.   
Let $\mathfrak{t}_1$ and $\mathfrak{t}_2$ be the two tubes in $F_1$ and $F_{1}$ that connect $\Gamma_1$ and $\Gamma_2$ to the central 2-spheres $S_1$ and $S_{2}$ respectively.   
Next, we show that $G_1=F_1\cup x_1\cup F_{2}$ can be viewed as a suspension surface. 
To see this, the first step is to push $\Gamma'$ into a $\partial$-parallel annulus in $X$, see the change from Figure~\ref{Fslide}(a) to (b), where the shaded region denotes $x_1$. 
Now view $x_1$ as a subsurface in the interior of $\Gamma'$. 
Then perform a handle/tube slide on $\mathfrak{t}_1$, sliding $\mathfrak{t}_1$ across $x_1$ and then passing over $\mathfrak{t}_2$, which isotopes $\mathfrak{t}_1$  into a tube connecting the two central 2-spheres $S_1$ and $S_{2}$, see the isotopy from Figure~\ref{Fslide}(b) to (c). 
Now $\mathfrak{t}_1$ merges $S_1$ and $S_{2}$ into a single 2-sphere. 
Since the cross-section disks $D_1$ and $D_{2}$ are adjacent, we can then push all the tubes in $F_1$ into a neighborhood of the cross-section disk $D_{2}$. 
This operation changes  $G_1=F_1\cup x_1\cup F_{2}$ into a single suspension surface which we still denote by $G_1$.  
From the viewpoint of $W$, this operation is an isotopy and the annulus $x_1$ is isotoped into a subsurface of the new suspension surface $G_1$.  
Since the cross-section disk $D_1$ and $D_2$ are adjacent, this isotopy does not affect other $F_i$'s.

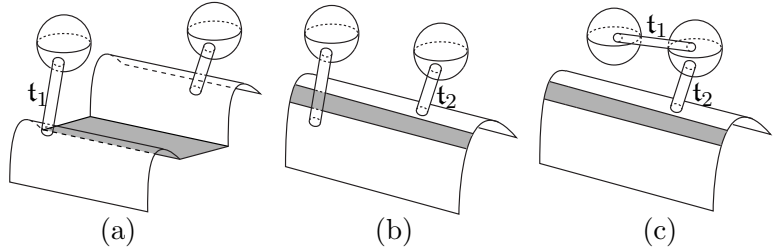
\begin{figure}[h]
\begin{center}
\begin{overpic}[width=4in]{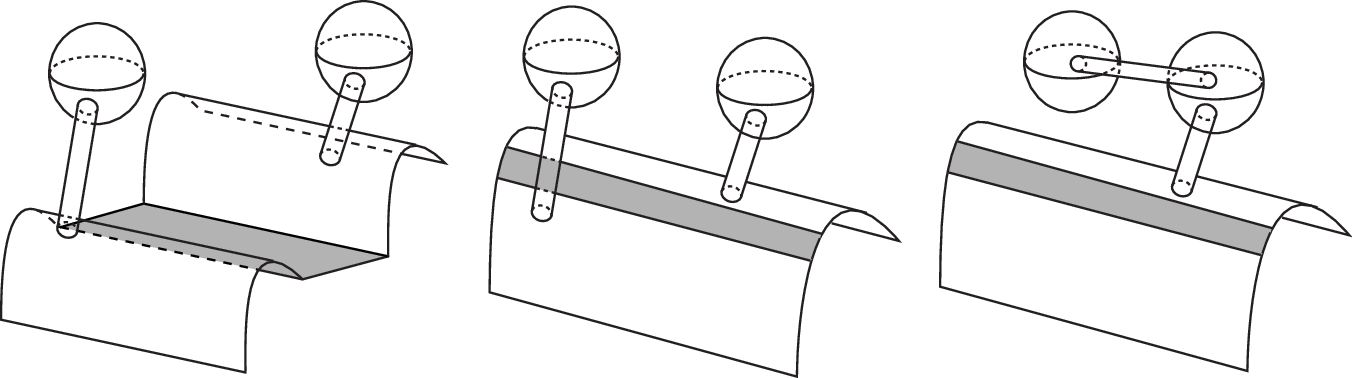}
\put(12,-3){(a)}
\put(48,-3){(b)}
\put(83,-3){(c)}
\put(2.5,15){$\mathfrak{t}_1$}
\put(83.5,24){$\mathfrak{t}_1$}
\put(56,14.7){$\mathfrak{t}_2$}
\put(89.5,15){$\mathfrak{t}_2$}
\end{overpic}
\vspace{6pt}
\caption{Slide a tube across an annulus to connect two central spheres}\label{Fslide}
\end{center}
\end{figure}

Note that the annulus $x_{2}$ connects $G_1$ to $F_{3}$. 
Thus the new set of suspension surfaces  $G_1, F_{3},\dots, F_t$ are connected  by $x_{2},\dots, x_{t-1}$.  By repeating the operation above, we can isotope the surface $G_{t-1}$ into a single suspension surface, and the annuli $x_1,\dots, x_{t-1}$ are viewed as subsurfaces of this suspension surface.  Since the two curves in $\partial A$ are both boundary curves of $G_{t-1}$, there is a disk $\Delta_A$, as shown in Figure~\ref{Fdual}(a), which intersects $A$ in a single essential arc.  As the annuli $x_1,\dots, x_{t-1}$ are now viewed as subsurfaces 
of the suspension surface $G_{t-1}$, $\partial\Delta_A$ may intersect $x_1,\dots, x_{t-1}$. 
We may reverse the isotopy and isotope $G_{t-1}$ back to its original position, and this isotopy changes $\Delta_A$ into a disk that possibly intersects the annuli $x_1,\dots, x_{t-1}$ in their original positions. 
Since the order $o(A)$ is larger than the orders $o(x_1),\dots, o(x_{t-1})$ in Case (1), $\Delta_A$ is a dual disk for $A$.  

\vspace{5pt}

\noindent
{\it Case (2)}.  The order $o(A)\le o(x_j)$ for some  $1\le j\le t-1$.

\vspace{5pt}

Without loss of generality, suppose $j$ is the largest index such that $o(A)\le o(x_j)$ and $1\le j\le t-1$. So $o(A)$ is larger than the orders of $x_{j+1},\dots, x_{t-1}$. 

Consider the surface $G_{k-1}$ which is the union of all the suspension surfaces $F_1,\dots, F_k$ and the annuli $x_1,\dots, x_{k-1}$.  $G_{k-1}$ is a connected surface. 
Now remove the annulus $x_j$ from $G_{k-1}$.  Since $\{x_1,\dots, x_{k-1}\}$ is a minimal set of annuli connecting $F_1,\dots, F_k$, $G_{k-1}\setminus x_j$ has two components, which we denote by $H_1$ and $H_2$. Without loss of generality, suppose $F_t\subset H_2$. Recall that $\partial A\subset \partial F_s\cup \partial F_t$ and $s<t$.  We have two subcases:

The first subcase is that $F_s\subset H_1$.  Since the signs and orders of the surfaces in $\partial W$ are admissible, by condition (5) in Definition~\ref{Dadmissible}, there is an arc $\alpha\subset\partial^+W$ connecting the two components of $\partial A$ such that $\alpha$ only intersects annuli in $\partial_v^+W$ with order smaller than $o(A)$.  Since  $o(A)\le o(x_j)$, $\alpha\cap x_j=\emptyset$.  Since $F_s\subset H_1$ and $F_t\subset H_2$,  one component of $\partial A$ lies in $H_1$ and the other component of $\partial A$ lies in $H_2$. 
So $\alpha$ is an arc in $\partial^+W$ connecting $H_1$ to $H_2$.  This means that one can find an annulus $x_j'$ in $\partial_v^+W$ such that (1) $\alpha$ intersects $x_j'$ and (2) $x_j'$ connects $H_1$ to $H_2$.  Since $\alpha$ intersects $x_j'$, by condition (5) in Definition~\ref{Dadmissible}, we have $o(x_j')< o(A)\le o(x_j)$.  By replacing $x_j$ with the annulus $x_j'$, we obtain a new set of annuli with smaller order and connecting the surfaces $F_1,\dots, F_k$, contradicting the hypothesis that $\{x_1,\dots, x_{k-1}\}$ is a minimal set of such annuli.

The second subcase is that $F_s\subset H_2$.  Consider the surface  $G_{j-1}$.  Recall that $G_{j-1}$ is connected and the annulus $x_j$ connects $F_{j+1}$ to $G_{j-1}$.   If there is another annulus $x_p$ ($j<p\le t-1$) that connects  $G_{j-1}$ to $F_{p+1}$, then by condition (3) on the $F_i$'s and $x_i$'s before the lemma, we have $o(x_j)\le o(x_p)$, and this contradicts the assumption at the beginning of Case (2) that $j$ is the largest index such that $o(A)\le o(x_j)$ and $1\le j\le t-1$.  Thus none of $x_{j+1},\dots, x_{t-1}$ are attached to $G_{j-1}$.  This implies that the annuli $x_{j+1},\dots, x_{t-1}$ connect $F_{j+1},\dots, F_t$ together into a connected surface, which we denote by $H'$.  In particular $F_t\subset H'$.  Since  $F_t\subset H_2$ and since $H'$ is connected, we have $H'\subset H_2$.  Since $G_{j-1}$ is connected and $F_s\subset H_2$, we must have  $G_{j-1}\subset H_1$ and $F_s\not\subset G_{j-1}$. This implies that $F_s\subset H'$ and hence $\partial A\subset\partial H'$.

Now we apply the argument in Case (1) on $H'$.  By performing tube slides, we can isotope $H'$ into a single suspension surface which we still denote by $H'$.  The annuli $x_{j+1},\dots, x_{t-1}$ are now subsurfaces of $H'$.  Moreover, since the cross-section disks $D_{j+1},\dots, D_t$ are adjacent to one another, this isotopy does not affect other surfaces $F_i$.  As $\partial A\subset\partial H'$, there is a disk $\Delta_A$, as shown in Figure~\ref{Fdual}(a), which intersects $A$ in a single essential arc.  Note that $x_{j+1},\dots, x_{t-1}$ are subsurfaces of $H'$, so $\partial\Delta_A$ may intersect $x_{j+1},\dots, x_{t-1}$. By our assumption on $j$ at the beginning of Case (2), the order $o(A)$ is larger than the orders of $x_{j+1},\dots, x_{t-1}$.  This means that, after isotope $H'$ back to its original
position, $\Delta_A$ is a dual disk for $A$.

Therefore, in both Case (1) and Case (2), there is a dual disk for any annulus $A$ in $\partial_v^-W$.  Moreover, since the surfaces in $\partial W$ are admissible, no annulus in $\partial_vW$ has both boundary curves in $\partial_h^-W=\partial_hX$.  So every annulus of $\partial_vX$ must contain boundary curves of some $F_i$.  \
Since $G_{k-1}$ can be isotoped into a single suspension surface and since $\partial_hX$ is planar, a maximal compression on $\partial^+W$ in $W$ yields a collar neighborhood of $\partial^-W=\partial_hX\cup\partial_v^-W$.  This means that $W$ is a relative compression body.
\end{proof}
\begin{remark}\label{Rgenus}
Given any relative compression body $W$, by Remark~\ref{Rproperty}(3), if we maximally compress $\partial^+W$, the resulting manifold is a product neighborhood of $\partial^-W$. Thus we always have $g(\partial^+W)\ge g(\partial^-W)$.  This is a major reason that we require $\partial_hX$ to be a planar surface in Lemma~\ref{LW2}.  
\end{remark}

\begin{lemma}\label{LW1}
Let $X$, $W$, $F_1,\dots, F_k$ be as above. 
In particular, each $F_i$ is a suspension surface in $X$ with a plus sign. 
Suppose $\partial_hX$ has a plus sign. Suppose the signs and orders of the surfaces in $\partial W$ are admissible.  Let $\{x_1,\dots, x_{k}\}$ be a minimal set of annuli connecting $\partial_hX, F_1,\dots, F_k$ as above and  suppose these surfaces and the cross-section disks  are well-positioned. Then $W$ is a relative compression body with respect to these signs and orders.
\end{lemma}
\begin{proof}
The proof is similar to the proof of Lemma~\ref{LW2} except that $\partial_hX$ is a component of $\partial_h^+W$ in this lemma.

Let $A$ be any annulus in $\partial_v^-W$.  
We first construct a dual disk for $A$. 
If both components of $\partial A$ lie in the same surface $F_i$, then $A$ has a dual disk $\Delta_A$ as shown in Figure~\ref{Fdual}(a) and $\partial \Delta_A$ does not intersect any annulus in $\partial_v^+W$. 
If both curves of $\partial A$ are boundary curves of $\partial_hX$, then $A$ must be a component of $\partial_vX$ and there is a dual disk $\Delta_A$, as in the proof of Lemma~\ref{LBM1}, disjoint from all the $F_i$'s. 
Now suppose that the two curves of $\partial A$ lie in different surfaces of $\{\partial_h X, F_1,\dots, F_k\}$.

Let $\widehat{G}_i$ be the subset of the surfaces in $\{\partial_h X, F_1,\dots, F_k\}$ attached to $x_1, \dots, x_{i}$.  
Let $G_i$ be the union of $x_1,\dots, x_i$ and the surfaces in $\widehat{G}_i$. 
By condition (ii) on the set of annuli $\{x_1,\dots, x_k\}$ before Lemma~\ref{LW2}, each $G_i$ is a connected surface. 

Suppose $\partial A\subset\partial G_t$ and suppose $t$ is the smallest such index (i.e.~$\partial A\not\subset G_{t-1}$).  Similar to Lemma~\ref{LW2}, we have two cases.

\vspace{5pt}

\noindent
{\it Case (1)}.  The order $o(A)$ is larger than the order of any annulus $x_i$ that lies in $G_{t}$, i.e. $o(A)>o(x_i)$ for any $1\le i\le t$.

\vspace{5pt}

If $\partial_hX$ is not in $\widehat{G}_t$, then the proof is the same as Case (1) of Lemma~\ref{LW2}.  By isotoping $G_t$ into a single suspension surface, we have a cross-section disk $\Delta_A$ for $A$, as shown in Figure~\ref{Fdual}(a).  So, it remains to consider the case that $\partial_hX\subset G_t$, i.e.~$G_t$ is the union of $x_1,\dots, x_t$, $\partial_hX$, and $F_1,\dots, F_t$.

\vspace{5pt}

\noindent
{\it Subcase (1a)}. Exactly one annulus in $\{x_1,\dots, x_t\}$ is attached to $\partial_hX$.

\vspace{5pt}

Suppose $x_q$ ($1\le q\le t$) is the annulus  attached to $\partial_hX$. 
Consider the subset of $t-1$ annuli $\{x_1,\dots, x_t\}\setminus x_q$.
In this subcase, $F_1,\dots, F_t$ are connected by this subset of $t-1$ annuli into a connected surface $F'$. Moreover, we can perform the tube slides as in the proof of Lemma~\ref{LW2}, which isotope $F'$ into a suspension surface.  So, in this subcase, we may view $G_t$ as the union of $F'$, $\partial_hX$, and the annulus $x_q$ between them.  

If both curves of $\partial A$ are boundary curves of $F'$, then as in Case (1) of Lemma~\ref{LW2}, $A$ has a dual disk $\Delta_A$ as shown in Figure~\ref{Fdual}(a).  If one curve of $\partial A$ is a boundary curve of $\partial_hX$ and the other curve of $\partial A$ is a curve in $\partial F'$, then as illustrated in the schematic picture in Figure~\ref{Fdual}(b), $A$ has a dual disk $\Delta_A$ that intersects the annulus $x_q$.   Note that we can isotope $F'$ back to the original position and this isotopy carries $\Delta_A$ into a  disk for $A$ that only intersects the annuli $x_1,\dots x_t$ of $\partial_v^+W$.   Since $o(A)$ is larger than the order of any annulus in $\{x_1,\dots, x_t\}$, $\Delta_A$ is a dual disk for $A$.

\vspace{5pt}

\noindent
{\it Subcase (1b)}.  More than one annulus in $\{x_1,\dots, x_t\}$ is attached to $\partial_hX$.

\vspace{5pt}

Similar to the proof of Lemma~\ref{LW2}, we try to isotope the $F_i$'s in $G_t$ into a single suspension surface. 
As before, for any two adjacent suspension surfaces connected by an annulus $x_i$ ($1\le i\le t$), we perform a tube slide and isotope them into a single suspension surface which contains $x_i$ as a subsurface.  Since more than one annulus in $\{x_1,\dots, x_t\}$ is attached to $\partial_hX$, after some tube slides and isotopies, we arrive at a situation that (1) $x_p$ and $x_q$ connect two suspension surfaces $F'$ and $F''$ to $\partial_h X$ and (2) the cross-section disks of $F'$ and $F''$ are adjacent.  Denote the cross-section disks of $F'$ and $F''$ by $D'$ and $D''$ respectively.

Let $\Gamma_1$ and $\Gamma_2$ be the $\partial$-parallel annuli (in the construction of the suspension surface $F'$ and $F''$) that are attached to $x_p$ and $x_q$ respectively.
Let $\mathfrak{t}_1$ and $\mathfrak{t}_2$ be the tubes connecting $\Gamma_1$ and $\Gamma_2$ to the central 2-spheres of $F'$ and $F''$ respectively.  
Let $H_1=\partial_hX\cup x_p\cup \Gamma_1$.  
As illustrated in the schematic pictures Figure~\ref{Fxslide1}(a, b), we perform an isotopy on $X$ by first pushing $x_p$ into the interior of $X$.  Then, as illustrated in Figure~\ref{Fxslide1}(c),
we slide $\mathfrak{t}_1$ across $x_p$ into a tube (in a neighborhood of the cross-section disk $D'$) connecting $\partial_hX$ to the central sphere of $F'$.  

\begin{figure}[h]
\begin{center}
\begin{overpic}[width=4in]{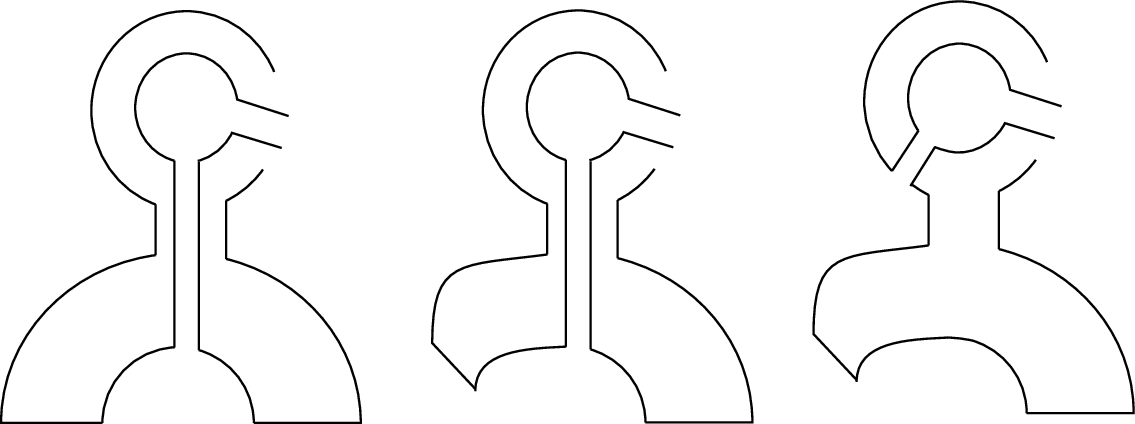}
\put(14,-3){(a)}
\put(50,-3){(b)}
\put(83,-3){(c)}
\put(4,1.5){$x_p$}
\put(12,3){$\Gamma_1$}
\put(12.7,10){$\mathfrak{t}_1$}
\put(47,10){$\mathfrak{t}_1$}
\put(81.5,21){$\mathfrak{t}_1$}
\put(3,14){$\partial_hX$}
\end{overpic}
\vspace{5pt}
\caption{Slide a tube to connect the central sphere to  $\partial_h X$}\label{Fxslide1}
\end{center}
\end{figure}
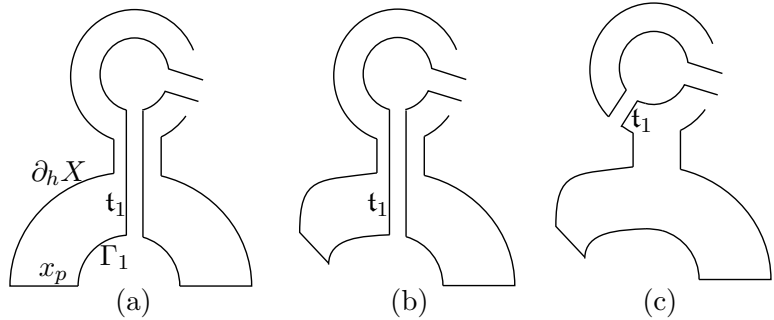

Next we show that we can slide the tube $\mathfrak{t}_1$ across $x_q$ and pass over $\mathfrak{t}_2$ into a tube connecting the two central spheres of $F'$ and $F''$.  To see this, note that we can similarly slide the tube $\mathfrak{t}_2$ across $x_q$ into a tube connecting $\partial_hX$ to the central sphere of $F''$.  This operation can be viewed as an isotopy of $W$. 
As illustrated in Figure~\ref{Fxslide2}(a, b), we can then slide $\mathfrak{t}_1$ over $\mathfrak{t}_2$ and into a tube connecting the two central spheres for $F'$ and $F''$. 
Since the cross-section disks for $F'$ and $F''$ are adjacent, this tube slide does not affect other suspension surfaces. 
Since these operations are all isotopies, this implies that, without isotoping the tube $\mathfrak{t}_2$ in the first step, we can slide $\mathfrak{t}_1$ across $x_q$  passing over $\mathfrak{t}_2$ into a tube connecting the two central spheres of $F'$ and $F''$.

Now we merge the central 2-spheres of $F'$ and $F''$ together along the tube $\mathfrak{t}_1$.  As in the proof of Lemma~\ref{LW2}, we can merge $F'$ and $F''$ into a single suspension surface.  Thus, after finite many such isotopes, we can merge $F_1,\dots, F_t$ into a single suspension surface $\widehat{F}$, and $\widehat{F}$ is connected to $\partial_hX$ by an annulus $x_q$.  Now the configuration is the same as Subcase (1a), and we can construct a dual disk for $A$ as in Subcase (1a).

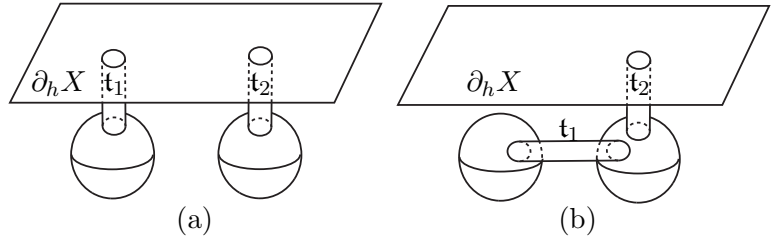
\begin{figure}[h]
\begin{center}
\begin{overpic}[width=4in]{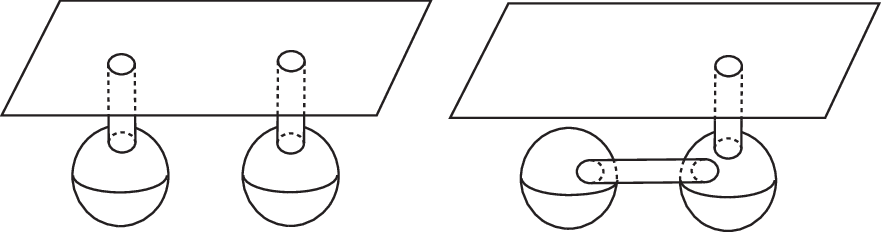}
\put(22,-3){(a)}
\put(72,-3){(b)}
\put(12.5,15){$\mathfrak{t}_1$}
\put(72,9){$\mathfrak{t}_1$}
\put(31.6,15){$\mathfrak{t}_2$}
\put(81.2,15){$\mathfrak{t}_2$}
\put(3,15){$\partial_hX$}
\put(60,15){$\partial_hX$}
\end{overpic}
\vspace{5pt}
\caption{Slide a tube to connect two central spheres}\label{Fxslide2}
\end{center}
\end{figure}

\vspace{5pt}

\noindent
{\it Case (2)}.  The order $o(A)\le o(x_j)$ for some annulus $x_j$ that lie in $G_{t}$, i.e.~$1\le j\le t$.

\vspace{5pt}

Without loss of generality, suppose $j$ is the largest index such that $o(A)\le o(x_j)$ and $1\le j\le t$. 
So $o(A)$ is larger than the orders of $x_{j+1},\dots, x_{t}$. 

Similar to Case (2) of Lemma~\ref{LW2}, consider $G_k$ and remove the annulus $x_j$ from $G_{k}$.  Since the $\{x_1,\dots, x_{k}\}$ is a minimal set of annuli connecting surfaces in $\partial_h^+W$, $G_{k}\setminus x_j$ has two components, denoted by $H_1$ and $H_2$. 

The first subcase is that the one component of $\partial A$ lies in $H_1$ and the other component of $\partial A$ lies in $H_2$.  This subcase is the same as the first subcase of Case (2) in Lemma~\ref{LW2}:
Since the signs and orders of the surfaces in $\partial W$ are admissible, there is an arc $\alpha\subset\partial^+W$ connecting the two components of $\partial A$. By replacing $x_j$ with an annulus that $\alpha$ intersects, as in Case (2) in Lemma~\ref{LW2}, we can obtain a new set of annuli with smaller order and connecting all the  surfaces of $\partial_h^+W$. 
This contradicts the hypothesis that $\{x_1,\dots, x_{k}\}$ is a minimal set of such annuli.

The second subcase is that both components of $\partial A$ lie in the same surface, say $H_2$. 
This subcase is similar to the second subcase of Case (2) in Lemma~\ref{LW2}, except that we have to take $\partial_hX$ into consideration. 

Consider the surface $G_{j-1}$. By the assumption that $\partial A\not\subset G_{t-1}$ before Case (1), the two boundary curves of $A$ cannot both be in $\partial G_{j-1}$.  If there is another annulus $x_p$ ($j<p\le t$) that connect  $G_{j-1}$ to a surface in $\widehat{G}_t\setminus\widehat{G}_{j-1}$, then by condition (iii) on $x_1,\dots, x_k$ before Lemma~\ref{LW2}, we have $o(x_j)\le o(x_p)$, and this contradicts the assumption above that $j$ is the largest index such that $o(A)\le o(x_j)$ and $1\le j\le t$.  
Thus none of $x_{j+1},\dots, x_{t}$ are attached to $G_{j-1}$.  
Let $\widehat{G}'$ be the collection of surfaces in $\widehat{G}_t\setminus\widehat{G}_{j-1}$, and let $H'$ be the union of $x_{j+1},\dots, x_{t}$ and all the surfaces in $\widehat{G}'$.
Similar to the second subcase of Case (2) in Lemma~\ref{LW2},
since none of $x_{j+1},\dots, x_{t}$ are attached to $G_{j-1}$, $H'$ must be a connected subsurface of $G_t$. 
Similarly, this implies that $H'\subset H_2$ and $\partial A\subset\partial H'$. 

Now we apply the argument in Case (1) on $H'$.  By performing tube slides, if $\partial_hX\not\subset H'$, we can isotope $H'$ into a single suspension surface, and if $\partial_hX\subset H'$, we can isotope $H'$ into the form of $\partial_hX\cup x_q\cup H''$, where $H''$ is a single suspension surface.  As in Case (1) above, this means that there is a disk $\Delta_A$ which intersects $A$ in a single essential arc, as illustrated in Figure~\ref{Fdual}(a or b).  
Note that $\partial\Delta_A$ may intersect $x_{j+1},\dots, x_{t}$, but by our assumption on $j$, the order $o(A)$ is larger than the orders of $x_{j+1},\dots, x_{t}$.  This means that, after isotope $H'$ back to its original
position, $\Delta_A$ is a dual disk for $A$.

Therefore, in both Case (1) and Case (2), there is a dual disk for any annulus $A$ in $\partial_v^-W$.  Moreover, since $\partial_h X$ and $F_1,\dots, F_k$ all have plus signs, 
 a maximal compression on $\partial^+W$ in $W$ yields a collar neighborhood of $\partial_v^-W$.  Hence $W$ is a relative compression body. 
\end{proof}

\begin{definition}\label{Dsimilar}
Let $X_1$ and $Y_1$ be two relative compression bodies.  We say that $X_1$ and $Y_1$ are {\bf $\partial$-similar} if there is a homeomorphism $f\colon \partial X_1\to\partial Y_1$ such that 
\begin{enumerate}
  \item $f$ maps $\partial_v^0X_1$, $\partial_v^\pm X_1$ and $\partial_h^\pm X_1$ homeomorphically to $\partial_v^0Y_1$, $\partial_v^\pm Y_1$ and $\partial_h^\pm Y_1$ respectively, and
  \item for each annulus $A$ in $\partial_v^\pm X_1$, $o(A)=o(f(A))$.
\end{enumerate}  
Note that the map $f$ is on the boundary only and it may not extend to a map from $\Int(X_1)$ to $\Int(Y_1)$.  

Let $X$ and $Y$ be two stacks of relative compression bodies.  Suppose $F_1,\dots, F_n$ are the horizontal surfaces in $X$ that divide $X$ into relative compression bodies $X_1,\dots, X_m$, and suppose $S_1,\dots, S_n$ are the horizontal surfaces in $Y$ that divide $Y$ into relative compression bodies $Y_1,\dots, Y_m$.  Suppose $F_i\cong S_i$ for all $i$ and let $h_i\colon F_i\to S_i$ be the homeomorphism.  Moreover, suppose these homeomorphisms $h_i$'s extend to homeomorphisms $f_j\colon \partial X_j\to \partial Y_j$ ($j=1,\dots, m$) which satisfy the two conditions above. In particular, $X_j$ and $Y_j$ are $\partial$-similar for all $j$. Then we say that the two stacks $X$ and $Y$ are {\bf similar}.
\end{definition}

Next, we use Lemma~\ref{LW2} and Lemma~\ref{LW1} to prove the main result on stacks of relative compression bodies.

\begin{lemma}\label{LtypeI}
Let $X$ be a marked handlebody.  Let $Y$ be a stack of relative compression bodies.  Suppose that there is a homeomorphism $f\colon (\partial Y, \partial_vY)\to (\partial X,\partial_vX)$ and suppose every horizontal surface in the stack $Y$ is separating. Then there is a collection of  properly embedded surfaces in $X$ dividing $X$ into a stack that is similar to the stack $Y$. 
\end{lemma}
\begin{proof}
We prove the lemma using induction on the number of relative compression bodies in the stack $Y$.  As $X$ is a marked handlebody, $\partial_hX$ is connected. 
Since $f\colon (\partial Y, \partial_vY)\to (\partial X,\partial_vX)$ is a homeomorphism, $\partial_hY$ is also connected. 
If the stack $Y$ has only one relative compression body, i.e.~$Y$ itself, since $\partial_hY$ is  connected and since $\partial_h^+ Y\ne\emptyset$ for any relative compression body $Y$, $\partial_hY$ must have a plus sign.  Thus we can assign $\partial_hX$ a plus sign and assign each component of $\partial_vX=f(\partial_vY)$ a sign and an order according to the sign and order of the corresponding component of $\partial_vY$.  By Lemma~\ref{LBM1}, $X$ is $\partial$-similar to $Y$.

Suppose the lemma is true if the number of relative compression bodies in the stack $Y$ is at most $n$. Now suppose $Y$ contains $n+1$ relative compression bodies.

Let $W_Y$ be the relative compression body in the stack $Y$ that contains $\partial_hY$.  Thus $\partial_hY$ is a component of $\partial_hW_Y$.  
Let $S_1,\dots,S_k$ be the components of $\partial_h W_Y\setminus\partial_hY$.  
So each $S_i$ is a properly embedded separating surface in $Y$. Let $y_1,\dots, y_p$ be the annuli in $\partial_vW_Y$. The complement $\overline{\partial_v Y-\cup_{i=1}^p y_i}$ is a collection of annuli, which we denote by $y_1',\dots, y_q'$.  Each $y_i'$ ($i=1,\dots, q$) lies outside $W_Y$.

\begin{claim}\label{Claim1}
For each $i$, both boundary curves of $y_i'$ belong to the same surface $S_j$ for some $j$. 
\end{claim}
\begin{proof}[Proof of the Claim] 
To see this, we first take an essential arc $\alpha$ of the annulus $y_i'$.  So $\alpha$ lies outside $W_Y$.  
Since $\partial\alpha\subset\partial W_Y$, we can connect the two endpoints of $\alpha$ using an arc properly embedded in $W_Y$.  Now we have a closed curve that intersects each component of $\partial y_i'$ in a single point. Suppose the two curves of $\partial y_i'$ belong to different components of $\partial_hW_Y$, say $S_s$ and $S_t$, then this closed curve intersects $S_t$ in a single point, which means that $S_t$ is non-separating, contradicting our hypothesis.
\end{proof}

Let $x_i=f(y_i)$ and $x_i'=f(y_i')$ be the corresponding annuli in $\partial_v X$.   Let $\Gamma_i$ be a properly embedded $\partial$-parallel annulus in $X$ with $\partial\Gamma_i=\partial x_i'$ ($i=1,\dots, q$).  By Claim~\ref{Claim1}, we can divide these $\Gamma_i$'s into $k$ sets of annuli $\widetilde{\Gamma}_1,\dots, \widetilde{\Gamma}_k$ such that $\Gamma_a\in\widetilde{\Gamma}_j$ if and only if the curves  $f^{-1}(\partial\Gamma_a)$ belong to the same surface $S_j$.    
Let $D_1,\dots,D_k$ be a sequence of parallel cross section disks and let $F_i$ be the suspension surface over the set of annuli $\widetilde{\Gamma}_i$ and based at the cross-section disk $D_i$.  
So a curve $\gamma$ is a boundary component of $F_i$ if and only if $f^{-1}(\gamma)$ is a boundary component of $S_i$.  Note that each suspension surface $F_i$ is planar, so $F_i$ may not be homeomorphic to $S_i$.  Nonetheless, $f(\partial S_i)=\partial F_i$.

Let $W_X$ be the submanifold of $X$ between $\partial_hX$ and these surfaces $F_1,\dots, F_k$. So $W_X$ is a type I block and $\partial_vW_X=\bigcup_{i=1}^p x_i$.  
We give a $\pm$-sign to $\partial_hX$ and each $F_i$ according to the signs of $\partial_hY$ and $S_i$ respectively.  Moreover, we give a sign and an order to each annulus $x_i$ according to the sign and order of $y_i$ ($i=1,\dots, p$).  Since $W_Y$ is a relative compression body, this one-to-one correspondence between surfaces in $\{\partial_hW_Y, \partial_vW_Y\}$ and $\{\partial_hW_X, \partial_vW_X\}$ implies that the surfaces $\partial_hW_X$, $\partial_vW_X$ with the signs and orders  are admissible, see Definition~\ref{Dadmissible}.

We have 4 cases to discuss depending on the signs of $S_1,\dots, S_k$, and $\partial_hY$.

\vspace{10pt}
\noindent
\emph{Case 1}. $S_1,\dots, S_k$ and $\partial_hY$ all have plus signs.
\vspace{10pt}

In this case, $\partial_h^+W_X=\partial_hX\cup F_1\cup\cdots\cup F_k$. 
Let $\{x_1,\dots,x_k\}$ be a minimal set of annuli connecting $\partial_hX$ and $F_1,\dots, F_k$ together, and we may choose these surfaces and cross-section disks $D_1,\dots, D_k$ to be well-positioned.  
By Lemma~\ref{LW1}, $W_X$ is a relative compression body.    

Let $F_i'$ be the standard surface obtained by adding $g(S_i)$ trivial tubes to $F_i$ ($i=1,\dots, k$).  So $F_i'\cong S_i$ for each $i$.  Let $W_X'$ be the submanifold of $W_X$ between $\partial_hX$ and $F_1'\dots, F_k'$.  Since $F_1,\dots, F_k$ and $\partial_hX$ all have plus signs and since $F_i'\cong S_i$, $W_X'$ is a relative compression body $\partial$-similar to $W_Y$.

By our construction, each standard surface $F_i'$ bounds a marked handlebody $X_i$ in $X$.  As each $S_i$ ($i=1,\dots, k$) is separating, $S_i$ cuts off a smaller stack $Y_i$ from $Y$.   By the induction hypotheses, there is a collection of surfaces in each $X_i$ that divides $X_i$ into a stack that is similar to the stack $Y_i$. 
As $W_X'$ is $\partial$-similar to $W_Y$, all these surfaces together divide $X$ into a stack that is similar to $Y$.

\vspace{10pt}
\noindent
\emph{Case 2}.  $S_1,\dots, S_k$ all have plus signs, but $\partial_hY$ has a minus sign.
\vspace{10pt}

In this case, $\partial_h^+W_X=F_1\cup\cdots\cup F_k$. 
Let $x_1,\dots,x_{k-1}$ be a minimal set of annuli connecting the $F_i$'s, and we may choose these surfaces and cross-section disks $D_1,\dots, D_k$ to be well-positioned.  
Note that if $\partial_hY$ is planar, then $\partial_hX$ is planar, and the lemma follows from Lemma~\ref{LW2} and the argument in Case 1.  So we suppose $g(\partial_hX)=g(\partial_hY)\ge 1$.

We may view $X$ as the manifold obtained by adding $g(\partial_hX)$ 1-handles to a ``smaller" marked handlebody $X_0$, where $P_0=\partial_hX_0$ is a planar surface and $\partial_v X_0=\partial_vX$.  We may view $F_1,\dots, F_{k}$ as suspension surfaces in $X_0$. 
 Let $W_0$ be the submanifold of $X_0$ between $P_0$ and $F_1,\dots, F_k$. 
Assign $P_0$ a minus sign. 
Since $P_0$ is a planar surface,  by Lemma~\ref{LW2}, $W_0$ is a relative compression body.

\vspace{10pt}
\noindent
\emph{Subcase 2a}.  $\sum_{i=1}^k g(S_i)\ge g(\partial_hY)$
\vspace{10pt}

Since $\partial_hX\cong\partial_hY$, $\sum_{i=1}^k g(S_i)\ge g(\partial_hX)$.  
Consider the $g(\partial_hX)$ 1-handles attached to $X_0$. 
Next, we add $g(S_i)$ tubes to each suspension surface $F_i$ ($i=1,\dots, k$), changing $F_i$ into a standard surface of genus $g(S_i)$. 
Since $\sum_{i=1}^k g(S_i)\ge g(\partial_hX)$, we can arrange $g(\partial_hX)$ of the total $\sum_{i=1}^k g(S_i)$ tubes going through the $g(\partial_hX)$ 1-handles attached $X_0$, exactly one tube for each 1-handle, and set the remaining $\sum_{i=1}^k g(S_i)-g(\partial_hX)$ tubes as trivial tubes.  Denote the resulting standard surfaces by $F_1',\dots, F_k'$.  So $g(F_i')=g(S_i)$ and by the construction of $F_i$, $F_i'\cong S_i$ for all $i$.  

Let $W_X'$ be the submanifold of $X$ between $\partial_hX$ and $F_1',\dots, F_k'$. So $\partial_h^+W_X'=\bigcup_{i=1}^k F_i'$ and $\partial_h^-W_X'=\partial_hX$. 
Since $W_0$ is a relative compression body, a maximal compression on $\partial^+W_0$ yields a product neighborhood of $\partial^-W_0$.  Since $P_0=\partial_h^-W_0$ and since each 1-handle of $X$ contains exactly one tube of $F_1',\dots, F_k'$, this implies that a maximal compression on $\partial^+W_X'$ yields a product neighborhood of $\partial^-W_X'$. 
Moreover, since $W_0$ is a relative compression body, each annulus of $\partial_v^- W_0$ has a dual disk in $W_0$.  As $\partial_v^- W_0=\partial_v^- W_X'$, each annulus of $\partial_v^- W_X'$ has an induced dual disk in $W_X'$.
Hence, $W_X'$ is a relative compression body $\partial$-similar to $W_Y$. Now the lemma follows from the induction as in Case 1.

\vspace{10pt}
\noindent
\emph{Subcase 2b}.  $\sum_{i=1}^k g(S_i)< g(\partial_hY)$
\vspace{10pt}

As $\partial_hY$ has a minus sign, we have $\partial_h^-W_Y=\partial_hY$ and $\partial_h^+W_Y=\cup_{i=1}^kS_i$.  As in Remark~\ref{Rgenus}, we have $g(\partial^- W_Y)\le g(\partial^+ W_Y)$, which implies that $g(\partial_hY)=g(\partial_h^-W_Y)\le g(\partial^+W_Y)$.  
Hence the hypothesis $\sum_{i=1}^k g(S_i)< g(\partial_hY)$ implies that $\sum_{i=1}^k g(S_i)< g(\partial^+W_Y)$.  
Since $\partial_h^+W_Y=\cup_{i=1}^kS_i$ and $\partial^+W_Y=\partial_h^+W_Y\cup \partial_v^+W_Y$, this inequality is possible because of the annuli in $\partial_v^+W_Y$.

Recall that $\{x_1,\dots,x_{k-1}\}$ is a minimal set of annuli connecting the $F_i$'s and $x_i=f(y_i)$.  So $y_1,\dots, y_{k-1}$ are $k-1$ annuli in $\partial_v^+W_Y$ connecting $S_1,\dots, S_k$ together.  Thus each additional annulus in $\partial_v^+W_Y\setminus\{y_1,\dots, y_{k-1}\}$ contributes an extra genus for $\partial^+W_Y=\partial_h^+W_Y\cup \partial_v^+W_Y$, so we have $g(\partial^+W_Y)=\sum_{i=1}^k g(S_i)+ |\partial_v^+W_Y|-(k-1)$. 
Let $g= |\partial_v^+W_Y|-(k-1)=g(\partial^+W_Y)-\sum_{i=1}^k g(S_i)$.  So $g(\partial^+W_Y)=\sum_{i=1}^k g(S_i)+ g$. Without loss of generality, suppose $y_k,\dots, y_{k+g-1}$ are the $g$ additional annuli in $\partial_v^+W_Y$. Hence $x_k,\dots, x_{k+g-1}$ are $g$ annuli in $\partial_v^+W_X$.

Next, we consider the marked handlebody $X_0$ and the relative compression body $W_0$ constructed at the beginning of Case 2.  

Let $\widehat{F}$ be the union of $x_1,\dots, x_{k-1}$ and the suspension surfaces $F_1,\dots, F_k$.  So $\widehat{F}$ is a connected planar surface.  
Similar to the proof of Lemma~\ref{LW2}, we can perform tube slides and isotope 
$\widehat{F}$ into a suspension surface in $X_0$. 
Next, we view $\widehat{F}$ as a suspension surface and view $F_1,\dots, F_k$ as subsurfaces of $\widehat{F}$.  
As in the proof of Lemma~\ref{LW2}, all the dual disks $\Delta_i$ in this configuration are as shown in Figure~\ref{Fdual}(a).  Since $x_k,\dots, x_{k+g-1}$ are $g$ annuli in $\partial_v^+W_0$, the corresponding disks $ \Delta_k,\dots, \Delta_{k+g-1}$, as in Figure~\ref{Fdual}(a), are $g$ compressing disks for $\partial^+W_0$.

Now we add $g(S_i)$ trivial tubes to each $F_i$ ($i=1,\dots, k$) and denote the resulting surface by $F_i'$.  So $F_i'\cong S_i$.  
As each $F_i$ is a subsurface of the suspension surface $\widehat{F}$, these trivial tubes can also be viewed as trivial tubes for $\widehat{F}$. 
Let $W_0'$ be the submanifold of $W_0$ bounded by $P_0=\partial_hX_0$, $F_1',\dots, F_k'$ and $\partial_vW_0$.  Let $\mathfrak{m}= \sum_{i=1}^k g(S_i)$. So $W_0'$ is obtained from $W_0$ by drilling out $\mathfrak{m}$  trivial tunnels.  Since each $F_i$ has a plus sign, $W_0'$ is also a relative compression body.  Moreover, each trivial tube determines a compressing disk for $F_i'$ in $W_0'$, see the top shaded disk in Figure~\ref{Fdual}(c).  Let $D_1',\dots, D_\mathfrak{m}'$ be the compressing disks in $W_0'$ corresponding to these $\mathfrak{m}$ trivial tubes. 

Let $\omega=g(\partial_hY)-\mathfrak{m}$ ($\mathfrak{m}= \sum_{i=1}^k g(S_i)$).  By the hypothesis $\sum_{i=1}^k g(S_i)< g(\partial_hY)$ in this subcase, we have $\omega> 0$. 
Since $g=g(\partial^+W_Y)-\mathfrak{m}$ and since 
$g(\partial_hY)=g(\partial_h^-W_Y)\le g(\partial^+W_Y)$, we have $g\ge \omega$.  
Note that $g(\partial_hY)=\omega+\mathfrak{m}$.  
Recall that each annulus $x_i$ ($i=k,\dots, g+k-1$) is associated with a disk $\Delta_i$, as illustrated in Figure~\ref{Fdual}(c), and since $x_k,\dots, x_{k+g-1}$ are $g$ annuli in $\partial_v^+W_X$, each $\Delta_i$ is a compressing disk for $\partial^+W_0'$. 
As $g\ge\omega$, we consider the first $\omega$ disks $\Delta_k,\dots,\Delta_{\omega+k-1}$.  

Consider the $\omega+\mathfrak{m}$ compressing disks $\Delta_k,\dots,\Delta_{\omega+k-1}$ and $D_1',\dots, D_\mathfrak{m}'$ for $\partial^+W_0'$ in $W_0'$. Recall that  $g(\partial_hY)=\omega+\mathfrak{m}$. 
Since $P_0=\partial_hX_0=\partial_h^-W_0'$, by Remark~\ref{Rproperty}(4), there are (core) arcs $\alpha_1,\dots,\alpha_{\omega+\mathfrak{m}}$ dual to the $\omega+\mathfrak{m}$ disks $\Delta_k,\dots,\Delta_{\omega+k-1}$ and $D_1',\dots, D_\mathfrak{m}'$ with $\partial\alpha_i\subset\partial_hX_0=P_0$, see Figure~\ref{Fdual}(c), such that after drilling out tunnels along these arcs $\alpha_i$, $W_X=W_0'\setminus\bigcup_{i=1}^{\omega+\mathfrak{m}}N(\alpha_i)$ is still a relative compression body.   

Note that there are $\omega$ disjoint circles $\gamma_1,\dots,\gamma_\omega$ in the central sphere of $\widehat{F}$ such that each $\gamma_j$ intersects the disk $\Delta_{j+k-1}$ exactly once and disjoint from other disks.  Similarly, each of the meridional circles of the trivial tubes meets the corresponding disk $D_i'$ exactly once, see Figure~\ref{Fdual}(c).  
So these arcs $\alpha_i$ basically go around  these circles. 
Moreover, if we ignore $\widehat{F}$, then these $\alpha_i$'s are arcs in a large 3-ball in $X_0$ that contain all the cross-section disks of $F_1,\dots, F_k$.  Thus, $\alpha_1,\dots,\alpha_{\omega+\mathfrak{m}}$ are trivial arcs in the marked handlebody $X_0$. 
 Hence $X=X_0\setminus \bigcup_{i=1}^{\omega+\mathfrak{m}}N(\alpha_i)$ is a marked handlebody with $\partial_vX=\partial_vX_0$ and $g(\partial_hX)=\omega+\mathfrak{m}=g(\partial_hY)$.  
This implies that $W_X$ is a submanifold of the marked handlebody $X$ and $W_X$ is $\partial$-similar to $W_Y$.   
Note that we may isotope $\widehat{F}$ back into the original positions of $F_1,\dots, F_i$ and $x_1,\dots, x_{k-1}$ and the isotopy carries these disks and arcs $\alpha_i$ to a different position. 
By our construction, after this isotopy, each $F_i'$, together with subannuli of $\partial_vX$, bounds a marked handlebody in $X$. 
Now by the induction as in Case 1, we can build a collection of surfaces that divide $X$ into a stack of relative compression bodies that is similar to the stack $Y$.

\begin{remark}\label{RfixH}
In the operation above, we fix $X_0$ and obtain the marked handlebody $X$ by drilling out $g(\partial_hX)$ trivial tunnels from $X_0$.  Instead of drilling tunnels, one can also carry out an equivalent (dual) operation by fixing $X$ while dragging some tubes of $F_i'$ through the $g(\partial_hX)$ 1-handles of $X$, in other words, re-embedding some tube of $F_i'$ into tubes that go through the 1-handles of $X$.  Note that, since the tubes of the suspension surface $F_i$ may be re-embedded into tubes through the 1-handles of $X$, the final surface $F_i'$ may not be a standard surface in $X$.  Nonetheless, each $F_i'$, together with subannuli of $\partial_vX$, still bounds a marked handlebody in $X$, and we can carry out the induction as in Case 1. 
\end{remark}

\vspace{10pt}
\noindent
\emph{Case 3}. $S_1,\dots, S_k$ all have minus signs.
\vspace{10pt}

By the definition of relative compression body, $\partial_h^+W_Y\ne\emptyset$. Hence $\partial_hY$ must have a plus sign and $\partial_h^+ W_Y=\partial_hY$.  

Let $y_1',\dots, y_q'$ be the collection of subannuli in $\partial_vY$ defined before Claim~\ref{Claim1}. 
We first show that each component of $\partial_vY$ contains at most one $y_i'$. 
 If two such annuli, say $y_i'$ and $y_j'$, lie in the same component of $\partial_vY$, then there is a component $y_t$ of $\partial_v W_Y$ between $y_i'$ and $y_j'$ with both boundary circles in $\partial y_i'\cup\partial y_j'$.  
This implies that $y_t$  has both boundary circles in $\bigcup_{i=1}^k \partial S_i$.  Since $S_1,\dots, S_k$ all have negative signs, $y_t$ is an annulus in $\partial_vW_Y$ with both boundary curves in $\partial_h^-W_Y$, contradicting the hypothesis that $W_Y$ is a relative compression body (by Definition~\ref{Drelative compression body}, there is no such annulus).  

Let $x_i'=f(y_i')$ ($i=1,\dots, q$). By the conclusion above, each component of $\partial_vX$ contains at most one $x_i'$. 
Consider the $\partial$-parallel annulus $\Gamma_i$ with $\partial\Gamma_i=\partial x_i'$ defined after Claim~\ref{Claim1}. Let $W_X'$ be the submanifold of $X$ obtained by deleting the solid tori bounded by $\Gamma_i\cup x_i'$ ($i=1,\dots, q$).  So $\partial W_X'$ consists of $\partial_hX$, $\Gamma_1,\dots,\Gamma_q$ and the annuli $x_1,\dots, x_p$, where $x_i=f(y_i)$.  We assign a sign and an order to each $x_i$ according to the sign and order of $y_i$.

As $X$ is a  marked handlebody, if we maximally compress $\partial_hX$ in $X$, the resulting manifold is a collar neighborhood of $\partial_vX$.   
Since each component of $\partial_vX$ contains at most one $x_i'$, if we maximally compress $\partial_hX$ in $W_X'$, we obtain a collection of solid tori, each of which is either a product neighborhood of some  $\Gamma_i$ or a collar neighborhood of a component of $\partial_vX$ that does not contain any $x_i'$.  
Assign a plus sign to $\partial_hX$ and a minus sign to each annulus $\Gamma_i$.  Similar to the proof of Lemma~\ref{LBM1}, $W_X'$ is a relative compression body with $\partial_h^-W_X'=\bigcup_{i=1}^q\Gamma_i$, $\partial_h^+W_X'=\partial_hX$ and $\partial_v W_X'=\bigcup_{i=1}^p x_i$.

By Remark~\ref{Rproperty}(4), there is a graph $Z$ in $W_X'$ connecting all the annuli $\Gamma_i$ together such that $W_X'\setminus N(Z)$ is a product neighborhood of $\partial^+W_X'$.  Moreover, by Remark~\ref{Rproperty}(4), for any nontrivial subgraph $Z'$ of $Z$, $W_X'\setminus N(Z')$ is also a relative compression body.  The graph $Z$ can be constructed to have $k$ disjoint subgraphs $Z_1,\dots Z_k$, such that 
\begin{enumerate}
  \item  each $Z_j$ connects a subset $\widetilde{\Gamma}_j$ of the annuli $\Gamma_i$'s, 
  \item    $\Gamma_a\in\widetilde{\Gamma}_j$ if and only if $\partial\Gamma_a\subset f(\partial S_j)$, $j=1,\dots, k$, 
  \item  let $F_j'$ be the frontier surface of $\widetilde{\Gamma}_j\cup N(Z_j)$ in $W_X'$, then $g(S_j)=g(F_j')$.
\end{enumerate}  
Let $W_X=W_X'\setminus \cup_{j=1}^kN(Z_j)$.  We set $\partial_h^+W_X=\partial_hX$, $\partial_vW_X=\partial_vW_X'$, and $\partial_h^-W_X=\bigcup_{i=1}^k F_i'$.  Conditions (2) and (3) above imply that each $F_j'$ is homeomorphic to $S_j$. So $W_X$ is a relative compression body $\partial$-similar to $W_Y$.  Moreover, each $F_i'$ (together with some subannuli of $\partial_vX$) bounds a marked handlebody in $X$. Thus we can proceed with the induction as in Case 1 and this proves Case 3.

\vspace{10pt}
\noindent
Case 4. $S_1,\dots, S_k$ do not have the same sign.
\vspace{10pt}

The proof for Case 4 is a mix of the proofs for Cases 1, 2, and 3.

Let $y_i$ $S_i$ and $\Gamma_i$ be as above.  Without loss of generality, we may suppose $S_i$ has a plus sign if $i\le t$ and has a minus sign if $t+1\le i\le k$.  We divides these $\partial$-parallel annuli $\Gamma_i$ into $k$ sets of annuli $\widetilde{\Gamma}_1,\dots, \widetilde{\Gamma}_k$ such that $\Gamma_a\in\widetilde{\Gamma}_j$ if and only if $\partial\Gamma_a\subset f(\partial S_j)$ ($j=1,\dots,k$).

Now we temporarily ignore $\widetilde{\Gamma}_{t+1},\dots,\widetilde{\Gamma}_k$ and proceed as in Case 1 and Case 2:  We first construct suspension surfaces $F_1,\dots, F_t$ over the sets of annuli $\widetilde{\Gamma}_1,\dots,\widetilde{\Gamma}_t$ respectively.   
Let $V$ be the submanifold of $X$ bounded by $\partial_hX$, $F_1,\dots, F_t$, and the collection of subannuli of $\partial_vX$ between these surfaces. 
So $\partial_hV=\partial_hX\cup(\bigcup_{i=1}^t F_i)$ and $\partial_v V$ consists of subannuli of $\partial_vX$.     
Denote the annuli in $\partial_vV$ by $x_1,\dots, x_m$, $w_1,\dots,w_s$, where no $x_i$ ($i=1,\dots,m$) contains boundary curves of any annulus in $\{\widetilde{\Gamma}_{t+1},\dots,\widetilde{\Gamma}_k\}$, and each $w_j$ ($j=1,\dots,s$) contains boundary  curves of at least one annulus in $\{\widetilde{\Gamma}_{t+1},\dots,\widetilde{\Gamma}_k\}$.  
Note that since the annuli $\Gamma_i$ are non-nested, 
each annulus in $\{\widetilde{\Gamma}_{t+1},\dots,\widetilde{\Gamma}_k\}$ must have both boundary curves in the same $w_j$ for some $j=1,\dots, s$.  

By our construction, $\partial F_i=f(\partial S_i)$ ($i=1,\dots,t$). We assign $F_1,\dots, F_t$ plus signs and assign $\partial_hX$ a $\pm$-sign according to the sign of $\partial_hY$. 
Note that each annulus $x_i$ ($i=1,\dots, m$) corresponds to a component $y_i$  of $\partial_vW_Y$ ($y_i=f(x_i)$).  
Assign a sign and an order to each $x_i$ according to the sign and order of $y_i$.  

\begin{claim}\label{Claim2}
Each $w_j$ contains the boundary curves of exactly one annulus in $\{\widetilde{\Gamma}_{t+1},\dots,\widetilde{\Gamma}_k\}$, and both curves of $\partial w_j$ are boundary curves of  $\partial_h^+ V$. 
\end{claim}
\begin{proof}[Proof of the Claim]
This proof is similar to the argument in Case 3.  
If two annuli $\Gamma_a$ and $\Gamma_b$ in $\{\widetilde{\Gamma}_{t+1},\dots,\widetilde{\Gamma}_k\}$ have boundary curves in the same $w_j$, then $w_j$ must contain a subannulus $x_c$ between $\Gamma_a$ and $\Gamma_b$ such that   $\partial x_c\subset\bigcup_{i=t+1}^k\widetilde{\Gamma}_{i}$.  Let $y_c = f^{-1}(x_c)$ be the corresponding annulus in $\partial_v W_Y$. 
Since $S_{t+1},\dots, S_k$ have minus signs, this means that $\partial y_c$ has both curves in $\partial_h^-W_Y$.  However, this is impossible because $W_Y$ is a relative compression body and, by Definition~\ref{Drelative compression body}, no annulus in $\partial_vW_Y$ has both boundary curves in  $\partial_h^-W_Y$.

Similarly, if a curve of $\partial w_j$ is a boundary curve of $\partial_h^- V$  (this occurs only if $\partial_h X$ has a minus sign, since $F_1,\dots, F_t$ all have plus signs), then since $w_j$ contains the boundary curves of an annulus in $\{\widetilde{\Gamma}_{t+1},\dots,\widetilde{\Gamma}_k\}$, a subannulus $x_d$ of $w_j$ has one boundary curve in some $\widetilde{\Gamma}_{p}$ ($t+1\le p\le k$) and the other boundary curve in $\partial_h^- V$.  Thus $y_d = f^{-1}(x_d)$ is a component of $\partial_v W_Y$ with both boundary curves in   $\partial_h^-W_Y$, which is impossible. 
\end{proof}

Now we assign each annulus $w_i$ a minus sign and an order larger than the order of any annulus in $\partial_v^+W_Y$, and we view each $w_i$ as an annulus in $\partial_v^-V$. 
Since $\partial^+ W_Y$ is connected (see Lemma~\ref{Lconnected}), the construction of $F_i$ implies that $\partial^+V=\partial_h^+V\cup\partial_v^+V$ is connected.  
Thus the assumptions on $W_Y$ and the $S_i$'s imply that these signs and orders on $\partial V$ are admissible, see Definition~\ref{Dadmissible}.  
As before, we can find a minimal set of annuli connecting the components of $\partial_h^+V$, and we can assume the $F_i$'s and their cross-section disks are well-positioned.  
As in Case 1 and Case 2, we can construct a surface $F_i'$ by modifying $F_i$ and 
 adding $g(S_i)$ tubes to $F_i$ ($i=1,\dots,t$), such that the submanifold $V'$ between $\partial_hX$ and $\cup_{i=1}^t F_i'$ is a relative compression body with each $w_i$ an annulus in $\partial_v^-V'$.  
 
By Claim~\ref{Claim2},  for each annulus $w_i$, there is exactly one annulus in $\{\widetilde{\Gamma}_{t+1},\dots,\widetilde{\Gamma}_k\}$ having boundary in $w_i$. 
Similar to Case 3, let $V''$ be the submanifold of $V'$ obtained by deleting the solid tori bounded by the $\partial$-parallel annuli in $\{\widetilde{\Gamma}_{t+1},\dots,\widetilde{\Gamma}_k\}$ and the corresponding subannuli of $w_1,\dots,w_s$.  Now view the annuli in $\{\widetilde{\Gamma}_{t+1},\dots,\widetilde{\Gamma}_k\}$ as components of $\partial_h^-V''$. 
Similar to Case 3, since each $w_i$ is a component of $\partial_v^- V'$ and since $V'$ is a relative compressing body, 
if we maximally compress $\partial^+V'$ in $V''$, the resulting manifold is a product neighborhood of $\partial^-V''$.  Therefore, $V''$ is a relative compression body (with orders and signs the same as corresponding components of $\partial_h^\pm V'$ and $\partial_v^\pm V'$).  
Note that, if $\partial_hX$ has a plus sign, then $\partial_h^- V''=\bigcup_{i=t+1}^k\widetilde{\Gamma}_{i}$, and if  $\partial_hX$ has a minus sign, then $\partial_h^- V''=(\bigcup_{i=t+1}^k\widetilde{\Gamma}_{i})\cup\partial_hX$.

By Remark~\ref{Rproperty}(4), there is a graph $Z$ in $V''$ connecting all the components of $\partial^- V''$, such that $V''\setminus N(Z)$ is a product neighborhood of $\partial^+V''$. 
Similar to Case 3, there are graphs $Z_{t+1},\dots, Z_k$ in $V''$ such that 
\begin{enumerate}
  \item  each $Z_j$ connects all the annuli in $\widetilde{\Gamma}_j$, and 
  \item    let $F_j'$ be the frontier surface of $\widetilde{\Gamma}_j\cup N(Z_j)$ in $V''$ ($j=t+1,\dots, k$), then $g(S_j)=g(F_j')$.
\end{enumerate}
Let $W_X$ be the manifold obtained by deleting a small neighborhood of $\cup_{i=t+1}^kZ_i$ from $V''$.  Note that after deleting $N(Z_j)$, the annuli in $\widetilde{\Gamma}_j$ merge into a connected surface $F_j'$ which is homeomorphic to $S_j$ ($j=t+1,\dots, k$).  Since $F_i'\cong S_i$ for $i=1,\dots, t$, this means that $W_X$ is a relative compression body $\partial$-similar to $W_Y$.  Thus we can proceed with the induction and this proves Case 4.
\end{proof}

\section{Separating surfaces}\label{Ssep}

Let $M=\mathcal{W}\cup_\mathcal{T} \mathcal{V}$ and $N=\widehat{\mathcal{T}}\cup_\mathcal{T} \mathcal{V}$ be as in Theorem~\ref{Ttorus}, where $H^1(\mathcal{W})=\mathbb{Z}$ and $\widehat{\mathcal{T}}$ is a solid torus.  
We view $\mathcal{V}$ as a submanifold of both $M$ and $N$. 
Let $\Sigma$ be the collection of incompressible and strongly irreducible surfaces in an untelescoping of a minimal genus Heegaard splitting of $M$.  By Lemma~\ref{Ldivide} and Remark~\ref{Rtorus}, we may assume that $\Sigma$ is in good position with respect to $\mathcal{T}$.

First, we would like to point out that we only need to consider the case that $\Sigma$ has no component entirely in $\mathcal{W}$.  
If $\Sigma $ has a strongly irreducible surface entirely in $\mathcal{W}$, then since $\mathcal{T}$ is incompressible,  by maximally compressing this component on either side, we obtain an incompressible component of $\Sigma$ entirely in $\mathcal{W}$.  
For homology reasons, each closed surface in $\mathcal{W}$ is separating.  
Thus there is a collection of incompressible components $\mathcal{F}_1,\dots, \mathcal{F}_n$ of $\Sigma$ entirely in $\mathcal{W}$, such that (1) each $\mathcal{F}_i$ bounds a submanifold $Z_i$ of $\mathcal{W}$, (2) these $Z_i$'s are non-nested, and (3) $\Sigma$ has no component entirely in $\mathcal{W}\setminus\bigcup_{i=1}^n Z_i$.  
Let $\Sigma'$ be the components of $\Sigma$ lying outside $\bigcup_{i=1}^n Z_i$.  So no component of $\Sigma'$ is entirely in $\mathcal{W}$.  Since $\mathcal{F}_i$ is an incompressible component in the untelescoping,  $\Sigma\cap \Int(Z_i)$ gives a generalized Heegaard splitting of $Z_i$. 
Let $\mathcal{A}_i$ be the compression body in the untelescoping that lies outside $Z_i$ and with $\mathcal{F}_i\subset\partial_-\mathcal{A}_i$. Let $\mathcal{P}_i=\partial_+\mathcal{A}_i$. So $\mathcal{P}_i$ is a component of $\Sigma'$. 
By rearranging handles in $Z_i$ (as a converse of the untelescoping), we first convert $\Sigma\cap \Int(Z_i)$ into a Heegaard surface of $Z_i$. Then we amalgamate this Heegaard surface of $Z_i$ with $\mathcal{P}_i$ along $\mathcal{F}_i$, see \cite{Sch}.  As in \cite{Sch}, also see \cite{La2,L4,L14}, this amalgamation operation is basically using a tube across $\mathcal{F}_i$ to connect $\mathcal{P}_i$ with a closed surface in $Z_i$. In particular, it does not affect how these surfaces intersect $\mathcal{T}$. 
We perform this amalgamation operation for each $Z_i$ and let $\Sigma''$ be the resulting collection of surfaces. $\Sigma''$ decomposes $M$ into a generalized Heegaard splitting with the same genus (it can be viewed as rearranging the same set of handles).  
Moreover, no component of $\Sigma''$ is entirely in $\mathcal{W}$.  This amalgamation operation is away from the torus $\mathcal{T}$, so the intersection of $\mathcal{T}$ with each compression body is still a collection of incompressible annuli. Hence it follows from the proof of Lemma~\ref{Ldivide} that $\Sigma''$ is still in good position with respect to $\mathcal{T}$.  
We can continue the proof using $\Sigma''$ instead of $\Sigma$ and the proof is the same. 

By the argument above, we may 
suppose that $\Sigma$ has no component entirely in $\mathcal{W}$.

For homology reasons, there is a nontrivial simple closed curve in $\mathcal{T}$ that is null-homologous in $\mathcal{W}$ (this curve bounds a non-separating Seifert surface in $\mathcal{W}$).  
Choose a framing for $H_1(\mathcal{T})$ by setting the slope of this curve to be $1/0$ (or $\infty$).  We also call this slope a meridional slope.  Note that $N$ is obtained from $\mathcal{V}$ by a Dehn filling along the $\infty$-slope.

In this section, we prove Theorem~\ref{Ttorus} in the case that every component of $\Sigma\cap\mathcal{W}$ is separating in $\mathcal{W}$. 
Suppose every component of $\Sigma\cap\mathcal{W}$ is separating in $\mathcal{W}$. We have 3 cases to discuss.

\vspace{10pt}

\noindent

\textbf{Case (a)}.  A component of $\Sigma\cap \mathcal{W}$ has genus at least one.

\vspace{10pt}

Let $\Sigma_0$ be a component of $\Sigma\cap \mathcal{W}$ with $g(\Sigma_0)\ge 1$. 
Since every component of $\Sigma\cap\mathcal{W}$ is separating in $\mathcal{W}$, $\partial\Sigma_0$ has an even number of components.  Let $c_1,\dots, c_{2k}$ be the boundary curves of $\Sigma_0$.  
The curves $c_1,\dots, c_{2k}$ divide $\mathcal{T}$ into $2k$ annuli $A_1,\dots, A_{2k}$. Suppose $A_i$ is adjacent to $A_{i+1}$ for each $i$.  As $\Sigma_0$ is separating in $\mathcal{W}$, 
$A_{2i}$ and $A_{2i-1}$ lie on different sides of $\Sigma_0$ for each $i$.  

Now consider $N= \widehat{\mathcal{T}} \cup_\mathcal{T} \mathcal{V}$ and view $\mathcal{T}$ as the boundary of the solid torus $\widehat{\mathcal{T}}$.

We first construct a surface $\Sigma_0'$ properly embedded in $\widehat{\mathcal{T}}$ such that 
\begin{enumerate}
  \item  $\partial\Sigma_0'=c_1\cup\cdots\cup c_{2k}$, 
  \item  $g(\Sigma_0')=g(\Sigma_0)$ and 
  \item  $\Sigma_0'$ divides the solid torus $\widehat{\mathcal{T}}$ into two marked handlebodies.
\end{enumerate}
The construction is fairly straightforward.  Let $\Gamma_i$ be a $\partial$-parallel annulus in the solid torus $\widehat{\mathcal{T}}$ with $\partial\Gamma_i=\partial A_{2i}$ ($i=1,\dots, k$).  Let $N_i$ be the solid torus bounded by $\Gamma_i\cup A_{2i}$.  
Let $c$ be a core curve of the solid torus $\widehat{\mathcal{T}}$ and suppose $c$ lies outside each $N_i$.  
So $\widehat{\mathcal{T}}\setminus N(c)\cong T^2\times I$. 
Let $J=\alpha\times I$ be a vertical annulus in  $\widehat{\mathcal{T}}\setminus N(c)\cong T^2\times I$ and we  choose the slope of $\alpha$ so that $J$ meets each curve $c_i$ in a single point. 
Hence $J\cap N_i$ is a bigon meridional disk of $N_i$. 
Let $h_1,\dots h_k$ be a collection of 1-handles connecting $\overline{N(c)}$  to $N_1,\dots, N_k$ respectively along vertical arcs in the vertical annulus $J$.  Let $X_0=\overline{N(c)}\cup (\bigcup_{i=1}^k (N_i\cup h_i)) $.  
So $X_0$ is a marked handlebody with $\partial_v X_0=\cup_{i=1}^k A_{2i}$ and $\partial_hX_0$ a genus-one surface properly embedded in $\widehat{\mathcal{T}}$.  
Since these 1-handles $h_i$ are unknotted, the closure of $\widehat{\mathcal{T}}\setminus X_0$ is a marked handlebody whose vertical boundary is $\cup_{i=1}^k A_{2i-1}$.  

Now we add $g$ trivial 1-handles to $X_0$, where $g=g(\Sigma_0)-1$, and denote the resulting manifold by $X_0'$.  Hence $X_0'$ is a marked handlebody with $\partial_h X_0'\cong\Sigma_0$.  
Let $X^c$ be the closure of the $\widehat{\mathcal{T}}\setminus X_0'$.
Since these $g$ 1-handles are trivial 1-handles, $X^c$ is a marked handlebody with $\partial_v X^c=\cup_{i=1}^k A_{2i-1}$.  
Set $\Sigma_0'=\partial_hX_0'=\partial_hX^c$, so we have $\Sigma_0'\cong\Sigma_0$

$\Sigma_0'$ divides $\widehat{\mathcal{T}}$ into two marked handlebodies and 
$\Sigma_0$ divides $\mathcal{W}$ into two stacks of relative compression bodies.  By applying Lemma~\ref{LtypeI} to each of the two stacks, we can construct a stack of relative compression bodies in $\widehat{\mathcal{T}}$ that is similar to the stack $\mathcal{W}$.  By Lemma~\ref{Lmerge}, we can connect the horizontal surfaces in $\widehat{\mathcal{T}}$ to $\Sigma\cap\mathcal{V}$ and obtain a generalized Heegaard splitting for $N$ with the same genus.  Thus $g(N)\le g(W)$ and Theorem~\ref{Ttorus} holds in Case (a).

\vspace{10pt}

\noindent
\textbf{Case (b)}.  The curves in $\Sigma\cap \mathcal{T}$ have an integer slope. 
\vspace{10pt}

This case is similar to Case (a).  Let $\Sigma_0$ be a component of $\Sigma\cap \mathcal{W}$.  
Since $\Sigma_0$ is separating in $\mathcal{W}$, $\partial\Sigma_0$ has an even number of components.  Let $c_1,\dots, c_{2k}$ be the boundary curves of $\Sigma_0$.  Let $A_i$ ($i=1,\dots, 2k$), $\Gamma_j$ and $N_j$  ($j=1,\dots, k$) be as in Case (a). 

Consider $N= \widehat{\mathcal{T}} \cup_\mathcal{T} \mathcal{V}$ and view $\mathcal{T}$ as the boundary of the solid torus $\widehat{\mathcal{T}}$.  
Let $\Delta$ be a compressing disk in the solid torus  $\widehat{\mathcal{T}}$. 
Since the curves in $\Sigma\cap \mathcal{T}$ have an integer slope, $\partial\Delta$ intersects each curve $c_i$ in a single point.
Hence $\Delta\cap N_i$ is a bigon compressing disk of $N_i$.  
Let $O$ be a point in $\Int(\Delta)$ and outside each $N_i$.  Let $h_1,\dots, h_k$ be a collection of 1-handles connecting a 3-ball $N(O)$ to $N_1,\dots, N_k$ respectively along arcs in $\Delta$.  Let $X_0$ be the union of $N(O)$, the solid tori $N_i$'s, and the 1-handles $h_i$'s.  So $X_0$ is a marked handlebody with $\partial_v X_0=\cup_{i=1}^k A_{2i}$. 
Since $\Delta$ is a compressing disk in the solid torus, the closure of $\widehat{\mathcal{T}}\setminus X_0$ is a marked handlebody whose vertical boundary is $\cup_{i=1}^k A_{2i-1}$.  
Then we add $g$ trivial 1-handles to $X_0$, where $g=g(\Sigma_0)$, and denote the resulting manifold by $X_0'$.  Similar to Case (a), $X_0'$ and its complement are both marked handlebodies.  Let $\Sigma_0'=\partial_h X_0'$.  Thus $\Sigma_0'\cong \Sigma_0$. 
Now the proof is the same as Case (a).

\vspace{10pt}

\noindent

\textbf{Case (c)}.  Every component of $\Sigma\cap \mathcal{W}$ is a planar surface, but there is a relative compression body  $Y$ in the stack $\mathcal{W}$ with $g(\partial^+Y)\ge 1$.

\vspace{10pt}

Recall that if one maximally compresses $\partial^+Y$ in $Y$, the resulting manifold is a product neighborhood of $\partial^-Y$.  
As in Remark~\ref{Rproperty}(4), this means that there is a graph $G$ connecting all the components of $\partial^-Y$ such that $Y\setminus N(G\cup \partial^-Y)\cong (\partial^+Y)\times I$.   
Denote $Y'=Y\setminus N(G\cup \partial^-Y)$ and let $\Sigma_0$ be the frontier surface of $\overline{N(G\cup\partial^-Y)}$ in $Y$.  So $\Sigma_0$ is a properly embedded surface in $Y$ and $\Sigma_0\subset\partial Y'$. 
 We may view $Y'$ as a (trivial) relative compression body with $\partial^+Y'=\partial^+Y$ and $\partial_h^-Y'=\Sigma_0$.  As $Y'\cong  (\partial^+Y)\times I$, we have $\Sigma_0\cong\partial^+Y$. By Lemma~\ref{Lconnected}, $\partial^+Y$ is connected. Hence $\Sigma_0$ is connected.  Since $g(\partial^+Y)\ge 1$, $g(\Sigma_0)\ge 1$.

$\Sigma_0$ divides $Y$ into two submanifolds $Y'$ and $Y''$, where $Y''$ is  a neighborhood of $G\cup\partial^-Y$.   If we maximally compress $\Sigma_0$ in $Y''$, the resulting manifold is a product neighborhood of $\partial^-Y$. Hence we may view $Y''$ as a relative compression body with $\partial_h^+Y''=\Sigma_0$ and $\partial^-Y''=\partial^-Y$.  Although $Y'$ and $Y''$ are both relative compression bodies, $Y=Y'\cup_{\Sigma_0}Y''$ is not a stack by definition because $\Sigma_0$ has a minus sign in $Y'$ but has a plus sign in $Y''$.  Nonetheless, the closure of each component of $\mathcal{W}\setminus \Sigma_0$ is a stack.

As $g(\Sigma_0)\ge 1$, we can repeat the construction in Case (a) and construct a surface $\Sigma_0'$ in the solid torus $\widehat{\mathcal{T}}$, such that $\Sigma_0'\cong\Sigma_0$, $\partial\Sigma_0'=\partial\Sigma_0$ in $\mathcal{T}$, and $\Sigma_0'$ divides $\widehat{\mathcal{T}}$ into two marked handlebodies. 
 Now the proof is the same as Case (a): in each of the two marked handlebodies, we apply Lemma~\ref{LtypeI} and construct a collection of surfaces, dividing the marked handlebody into a stack that is similar to the corresponding stack of $\overline{\mathcal{W}\setminus \Sigma_0}$.  After removing $\Sigma_0'$, we obtain a stack $\widehat{\mathcal{T}}$ that is similar to the stack $\mathcal{W}$.  Now the proof is the same as Case (a) and this finishes Case (c).

\vspace{5pt}

\begin{claim*}
Suppose every component of $\Sigma\cap \mathcal{W}$ is a planar surface and, for each relative compression body  $Y$ in the stack $\mathcal{W}$, $\partial^+Y$ is a planar surface.  Then the slope of the curves in $\Sigma\cap \mathcal{T}$ must be an integer.
\end{claim*}
\begin{proof}[Proof of the claim]
$\Sigma\cap \mathcal{W}$ divides $\mathcal{W}$ into a stack of relative compression bodies.  Let $Y$ be any relative compression body in this stack.  
By Remark~\ref{Rproperty}(5), the manifold $\widehat{Y}$ obtained by adding a 2-handle along each annulus in $\partial_v^-Y$ is also a relative compression body.  Next, add  2-handles to $\widehat{Y}$ along every annulus in $\partial_v^0Y$. By Remark~\ref{Rproperty}(6), the resulting manifold, denoted by $\widehat{Y}'$, is also a relative compression body.  Since each component of $\Sigma\cap \mathcal{W}$ is a planar surface, $\partial^-\widehat{Y}'$ is a collection of 2-spheres.  
By the hypothesis of the claim, $\partial^+Y$ is a planar surface, which implies that $\partial^+\widehat{Y}'$ is a 2-sphere.  Hence 
$\widehat{Y}'$ is a punctured 3-ball (i.e. a 3-ball with a collection of 3-balls removed from its interior).  Let $\widehat{Y}''$ be the manifold obtained from $\widehat{Y}'$ by adding a 2-handle along each annulus in $\partial_v^+Y$.  Since $\widehat{Y}'$ is a punctured 3-ball, so is $\widehat{Y}''$.

Let $\mathcal{W}'=\mathcal{W}\cup (D^2\times S^1)$ be the manifold obtained by a Dehn filling on $\mathcal{W}$ with each curve of $\Sigma\cap \mathcal{T}$ bounding a disk in the solid torus $D^2\times S^1$.  The argument above implies that each relative compression body in the stack $\mathcal{W}$ extends to a punctured 3-ball in $\mathcal{W}'$.  Hence $\mathcal{W}'$ is either $S^3$ or a connected sum of $S^2\times S^1$.

Now consider the slope of the curves in $\Sigma\cap \mathcal{T}$. If the slope is neither an integer nor the $\infty$-slope, then $H_1(\mathcal{W}')$ is a nontrivial finite cyclic group, contradicting our conclusion that $\mathcal{W}'$ is either $S^3$ or a connected sum of $S^2\times S^1$. If the slope is the $\infty$-slope, then $H_1(\mathcal{W}')=\mathbb{Z}$ and $\mathcal{W}'$ must be $S^2\times S^1$.  
Since  each relative compression body in the stack $\mathcal{W}$ extends to a punctured 3-ball in $\mathcal{W}'$, $\Sigma\cap\mathcal{W}$ extends into a collection of 2-spheres that divide $\mathcal{W}'$ into a collection of punctured 3-balls. 
This implies that there must be a component of $\Sigma\cap\mathcal{W}$ that extends to a non-separating 2-sphere in $\mathcal{W}'=S^2\times S^1$.  This means that a component of $\Sigma\cap\mathcal{W}$ has an odd number of boundary curves and is non-separating in $\mathcal{W}$, a contradiction to the hypothesis of this section. 
\end{proof}

Theorem~\ref{Ttorus} in the case that every surface in $\Sigma\cap\mathcal{W}$ is separating now follows from the claim above and the discussions in the 3 cases:  By Cases (a), we may assume every component of $\Sigma\cap\mathcal{W}$ is a planar surface.  By Case (c), we may assume that for each relative compression body $Y$ in the stack $\mathcal{W}$, $\partial^+Y$ is a planar surface.  So it follows from the Claim that the slope of the curves in $\Sigma\cap\mathcal{T}$ is an integer.  Now Theorem~\ref{Ttorus} follows from Case (b).

Next, we consider the case that a component of $\Sigma\cap\mathcal{W}$ is separating.

\section{Type II blocks}\label{StypeII}

In this section, we describe the second type of building block in the construction.  

We call a genus-$g$ handlebody $X$ an \textbf{extended marked handlebody} if $\partial X$ contains $m+1$ ($m\le g$) disjoint annuli $A_0, A_1,\dots, A_m$, such that 
\begin{enumerate}
\item $A_0$ is trivial  (i.e.~$\partial A_0$ is trivial in $\partial X$) and each $A_i$ ($i\ge 1$) is an essential annulus in $\partial X$
\item $X$ contains $m$ disjoint meridional disks $\tau_1,\dots, \tau_m$ such that $\tau_i\cap A_j=\emptyset$ if $i\ne j$, and $\tau_i\cap A_i$ is a single essential arc of $A_i$ for each $i\ge 1$.
\end{enumerate}
The only difference between a marked handlebody and an extended marked handlebody is that one marked annulus is trivial in an extended marked handlebody.

Similar to the description of a marked handlebody in section~\ref{StypeI}, we can describe $X$ as follows: 
Start with a marked handlebody $X_1$ constructed  by first connecting a central 3-ball $B$ to a collection of solid tori $T_1,\dots, T_m$ using 1-handles, and then adding $g-m$ 1-handles to $B$ (here we use the same notation as in section~\ref{StypeI}).  
Let $A_1, \dots, A_m$ be the annuli in $\partial T_1,\dots, \partial T_m$ with $\partial_vX_1=\bigcup_{i=1}^m A_i$. 
Next, connect a 3-ball $B'=D^2\times I$ to $B$ using a 1-handle $H_0$ and require $H_0$ is attached to $D^2\times\{1\}\subset\partial B'$.  Denote the resulting manifold by $X$ and denote $(\partial D^2)\times I$ by $A_0$.  
The set of annuli $A_0, A_1,\dots A_m$ are our marked annuli on $\partial X$.   

The boundary of $X$ has two parts: (1) the vertical boundary $\partial_vX=\bigcup_{i=0}^m A_i$ and (2) the horizontal boundary $\partial_hX=\overline{\partial X\setminus\partial_vX}$.   So $\partial_hX$ has two components, one of which is the disk $D^2\times\{0\}\subset\partial B'$.

Next, we describe the cross-section disks and suspension surfaces for $X$.  
We view $X$ as the manifold obtained by connecting a marked handlebody $X_1$ to a 3-ball $B'=D^2\times I$ via a 1-handle $H_0$ as above. 
A cross-section disk for $X$ is simply an extension of a cross-section disk for $X_1$ into the 1-handle $H_0$ and the 3-ball $B'$.  
More precisely, start with a cross-section disk $E'$ for the marked handlebody $X_1$.  
Let $H_0=\Delta_0\times I$ be the 1-handle connecting $X_1$ to $B'=D^2\times I$ with $\Delta_0\times\{0\}\subset D^2\times\{1\}\subset\partial B'$.  
We may suppose that there is a disk $E_H$ properly embedded in the 1-handle $H_0=\Delta_0\times I$ in the form of $\delta_0\times I$ such that $E_H\cap E'$ is the arc $\delta_0\times \{1\}$ in $\partial H_0$.  Moreover, we may suppose that there is a disk $E_D$  properly embedded in $B'=D^2\times I$ in the form of $\delta\times I$, such that $E_D\cap E_H$ is the arc $\delta_0\times \{0\}$ in $\partial H_0$.  
Let $E=E'\cup E_H\cup E_D$.  We call $E$ a {\bf cross-section disk} for $X$.

A {\bf suspension surface} in $X$ is constructed in the same way as in section~\ref{StypeI}:  
Start with a collection of non-nested $\partial$-parallel annuli, each of which is parallel to a subannulus of some $A_i$ ($i=0, 1,\dots, m$).  
Then connect these $\partial$-parallel annuli to a central 2-sphere along arcs in a cross-section disk.

Denote the two components of $\partial_hX$ by $F_0$ and $F_1$, where $F_0=D^2\times\{0\}\subset\partial B'$ is the disk component.  
Let $S$ be a suspension surface in $X$ over a set of $\partial$-parallel annuli.  
Let $\hat{S}$ be the surface obtained by connecting $S$ and the disk $F_0$ by a tube along an arc in the cross-section disk, see Figure~\ref{Ftype2}(a) for a 1-dimensional schematic picture. 
We call $\hat{S}$ an {\bf extended suspension surface}.   
Let $\hat{S}'$ be the surface obtained by adding some (possibly none) trivial tubes to $\hat{S}$, and 
let $X_0$ be the submanifold of $X$ bounded by $F_1$, $\hat{S}'$, and the subannuli of $\partial_vX$ between them. 
We call $X_0$ a {\bf type II block}.  Similarly, define the horizontal and vertical boundaries of $X_0$ as $\partial_hX_0=F_1\cup \hat{S}'$ and $\partial_vX_0=\overline{\partial X_0\setminus\partial_hX_0}$ respectively.

\vspace{10pt}

\noindent
{\bf  Conversion between type I and type II blocks:}

\vspace{10pt}

Let $T_0, T_1,\dots, T_m$ be a collection of solid tori and 
let $\mathfrak{X}$ be a marked handlebody constructed by connecting these $T_i$'s to a central 3-ball $B$ using 1-handles, as in section~\ref{StypeI}.   
Suppose $\partial_v\mathfrak{X}$ consists of annuli $A_0,\dots, A_m$, where each $A_i$ is an annulus in $\partial T_i$.  Let $\Gamma_0,\dots,\Gamma_n$ be a collection of non-nested $\partial$-parallel annuli in $\mathfrak{X}$ with $\partial\Gamma_i\subset\partial_v\mathfrak{X}$, and let $S$ be a suspension surface over these $\Gamma_i$'s.  
Let $\mathfrak{X}_0$ be the submanifold of $\mathfrak{X}$ bounded by $S$, $\partial_h\mathfrak{X}$, and the subannuli of $\partial_v\mathfrak{X}$ between them.  So $\mathfrak{X}_0$ is a type I block by definition. 

Suppose $\partial\Gamma_0\subset A_0$, and suppose a component of $\partial A_0$ and a component of $\partial\Gamma_0$ bounds a subannulus $A_0'\subset A_0$ such that $\Int(A_0')$ does not contain boundary curves of any $\Gamma_i$.  
In particular, $A_0'$ is a component of $\partial_v \mathfrak{X}_0$. 
Now add a 2-handle to $\mathfrak{X}_0$ along the annulus $A_0'$ and denote the resulting manifold by $X_0$.  We claim that $X_0$ is a type II block.  To see this, start with the solid torus $T_0$ in the construction of $\mathfrak{X}$ and view $\Gamma_0$ as a $\partial$-parallel annulus in $T_0$.  
Let $\widehat{\Gamma}_0$ be the solid torus bounded by $\Gamma_0$ and the subannulus of $A_0$ between $\partial\Gamma_0$.
We may view the closure of $T_0\setminus\widehat{\Gamma}_0$ as a product $A\times I$, where $A$ is an annulus, $\Gamma_0=A\times\{0\}$, and the annulus $A_0'$ described above is a component of $(\partial A)\times I$.  So adding a 2-handle along $A_0'$ changes $A\times I$ to a 3-ball $B'=D^2\times I$, and $\Gamma_0$ extends to the disk $D^2\times\{0\}$.  Moreover, the set of 1-handles connect $B'$, $T_1,\dots, T_m$ and $B$ into an extended marked handlebody.  Further, the suspension surface $S$ in $\mathfrak{X}$ becomes an extended suspension surface in this extended marked handlebody.  Thus $X_0$ is a type II block.

Conversely, given a type II block $X_0$ as above, let $\alpha$ be an arc of the form $\{x\}\times I\subset D^2\times I=B'$ and suppose $\alpha$ is disjoint from the cross-section disk of $\hat{S}$.  Then $\mathfrak{X}_0=X_0\setminus N(\alpha)$ is a type I block (one can view  $N(\alpha)$ as the 2-handle above and view $\alpha$ to be a co-core of the 2-handle).

\begin{remark}\label{Rconversion}
Let $X_0$ be a type II block and let $\alpha$ be the arc of the form $\{x\}\times I\subset D^2\times I=B'$ as above.  
Since $\alpha$ is unknotted, $X_0\setminus N(\alpha)$ is a topological handlebody.  We may view the closure of $N(\alpha)$ as a cylinder $\Delta_\alpha\times I$, where $\Delta_\alpha$ is a disk and $\Delta_\alpha\times \partial I\subset\partial_hX_0$. 
In the description above, we view $\mathfrak{X}_0=X_0\setminus N(\alpha)$ as a type I block and view $A_0'=\partial\Delta_\alpha\times I$ as a component of $\partial_v\mathfrak{X}_0$. 
Moreover, $\partial_h \mathfrak{X}_0=\partial_h X_0\setminus (\Delta_\alpha\times \partial I)$, and $\partial_h \mathfrak{X}_0$ has two components which we denote by $S_1$ and $S_2$.
Note that we may view $\mathfrak{X}_0$ and its boundary surface in a slightly different way:  
If we include the annulus $A_0'$ as part of $\partial_h \mathfrak{X}_0$ instead, in other words, if we define  $\partial_h\mathfrak{X}_0=S_1\cup A_0'\cup S_2$, then it follows from the proof of Lemma~\ref{LW1} that $\mathfrak{X}_0$ is a marked handlebody with $\partial_h\mathfrak{X}_0=S_1\cup A_0'\cup S_2$.
\end{remark}

\begin{figure}[h]
\begin{center}
\begin{overpic}[width=4in]{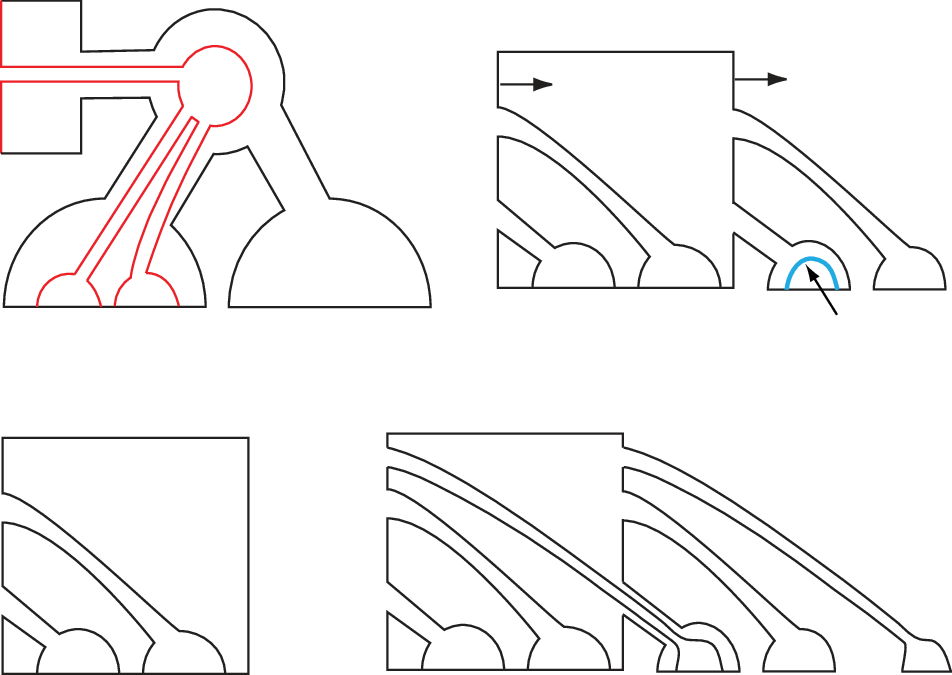}
\put(20,-5){(c)}
\put(67,-5){(d)}
\put(20,34){(a)}
\put(67,34){(b)}
\put(23,60){$\hat{S}$}
\put(34,51){$F_1$}
\put(0,21){$D^-$}
\put(26,16){$D^+$}
\put(12,11){$P_{k-1}$}
\put(60,55){$P_{k-1}^-$}
\put(84,55){$P_{k-1}^+$}
\put(86.5,35){$\Gamma_{k-2}$}
\put(48,21){$P_{k-2}^-$}
\put(74,20){$P_{k-2}^+$}
\put(58,17.5){$\mathfrak{X}_{k-2}$}
\end{overpic}
\vspace{6pt}
\caption{Extended suspension surfaces and type II blocks}\label{Ftype2}
\end{center}
\end{figure}
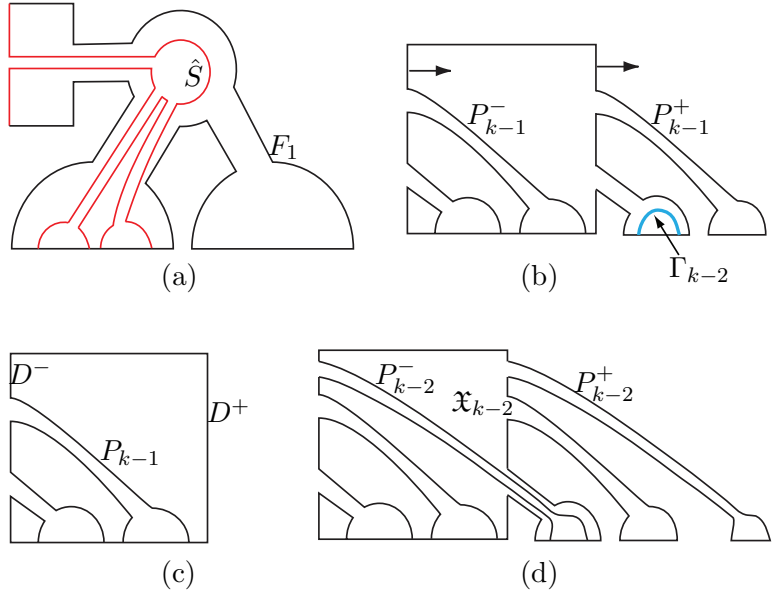

\section{Non-separating surfaces}\label{Snonsep}

Let $M=\mathcal{W}\cup_\mathcal{T} \mathcal{V}$ and $N=\widehat{\mathcal{T}}\cup_\mathcal{T} \mathcal{V}$ be as in Theorem~\ref{Ttorus}.  Let $\Sigma$ be the collection of surfaces in an untelescoping of a minimal genus Heegaard splitting of $M$.  In this section, we prove Theorem~\ref{Ttorus} in the case that $\Sigma\cap\mathcal{W}$ contains a component that is non-separating in $\mathcal{W}$.  This, combined with section~\ref{Ssep}, gives a full proof of Theorem~\ref{Ttorus}.

As in section~\ref{Ssep}, we view $\mathcal{T}$ as a torus in both $M$ and $N$, and we may assume that $\Sigma$ has no component that is entirely in $\mathcal{W}$.  Suppose at least one component of $\Sigma\cap \mathcal{W}$ is non-separating in $\mathcal{W}$.  By our choice of framing, the curves $\Sigma\cap \mathcal{T}$ in $\mathcal{T}$ must have the $\infty$-slope.

Let $S$ be a non-separating component of $\Sigma\cap \mathcal{W}$.     Denote the set of boundary curves of $S$ by $C_0$.  We fix a normal direction for $S$ which induces a normal direction for each circle of $C_0$ in $\mathcal{T}$.  
By the homology assumption on $\mathcal{W}$, $C_0$ (with the induced orientation) represents a primitive element in $H_1(\mathcal{T})$. In particular, $C_0$ contains an odd number of curves. 
Curves in $C_0$ divide the torus $\mathcal{T}$ into an odd number of annuli and we denote this set of annuli by $\mathcal{A}_0$.  We call an annulus in $\mathcal{A}_0$ an inner annulus if the normal directions of both of its boundary curves point out of this annulus.  The homology conclusion on $C_0$ implies that, if $|C_0|>1$, then there must be at least one inner annulus in $\mathcal{A}_0$.  Let $\mathcal{A}_0'$ be the set of all inner annuli in $\mathcal{A}_0$.

Next, we remove the boundary curves of the annuli in $\mathcal{A}_0'$ from $C_0$ and let $C_1$ denote the remaining set of curves. 
Clearly $C_0$ and $C_1$ represent the same element in $H_1(\mathcal{T})$.  Now we apply the same procedure on $C_1$: $C_1$ divides $\mathcal{T}$ into a collection of annuli, which we denote by $\mathcal{A}_1$, and we define inner annuli for $\mathcal{A}_1$ similarly.  If  $|C_1|>1$, there must be at least one inner annulus in $\mathcal{A}_1$. 
Let $\mathcal{A}_1'$ be the set of inner annuli in $\mathcal{A}_1$. 
Note that, by our construction, each inner annulus in $\mathcal{A}_1'$ must contain at least one inner annulus of $\mathcal{A}_0'$ as a subannulus.  
Similarly, let $C_2$ be the subset of $C_1$ obtained by deleting all the boundary curves of the annuli in $\mathcal{A}_1'$.  This procedure produces a sequence of sets of curves $C_0\supset C_1\supset\cdots\supset C_k$ and eventually $C_k$ is a single curve.  Moreover, each annulus in $\mathcal{A}_i'$ must contain at least one annulus of $\mathcal{A}_{i-1}'$ as a subannulus.

Next, we build a planar surface $P$ in $\widehat{\mathcal{T}}$, such that $\partial P$ is the same set of curves as $\partial S$ in $\mathcal{T}$. 
Let $D$ be a meridional disk of $\widehat{\mathcal{T}}$ bounded by $C_k$ (recall that $C_k$ is a single curve).  
For each inner annulus $A$ in $\mathcal{A}_i'$, build a $\partial$-parallel annulus $\Gamma$ in $\widehat{\mathcal{T}}$ parallel to $A$ and with $\partial\Gamma=\partial A$. 
Let $\Gamma_i$ be the set of such $\partial$-parallel annuli parallel to the set of inner annuli in $\mathcal{A}_i'$.  
We may suppose $D$ and all the annuli in the $\Gamma_i$'s are pairwise disjoint. 
Assign normal directions to $D$ and these $\partial$-parallel annuli according to the normal directions of their boundary curves.  So the union of $D$ and these $\partial$-parallel annuli is a disconnected surface whose boundary is the same as $\partial S$.

Now we use tubes to connect $D$ and the annuli in the $\Gamma_i$'s together in a way that preserves the induced normal directions.  The meridional disk $D$ cuts the solid torus $\widehat{\mathcal{T}}$ into a product $D\times I$.  
Denote $D\times\{0\}$ by $D^-$ and $D\times\{1\}$ by $D^+$, and view $D^\pm$ as the $\pm$-side of $D$.  Suppose the induced normal direction of $D^+$ points out of $D\times I$ and the direction of $D^-$ points into $D\times I$. 

As illustrated in the 1-dimensional schematic picture in Figure~\ref{Ftype2}(c), we first connect each annulus in $\Gamma_{k-1}$ to $D^-$ using a trivial arc.   Then we add tubes along these arcs to connect $D^-$ and the annuli in $\Gamma_{k-1}$. Denote the resulting planar surface by $P_{k-1}$.  By the definition of inner annulus, the induced normal directions for $\Gamma_{k-1}$ and $D^-$ are compatible in $P_{k-1}$.
We may view $P_{k-1}$ as a non-separating surface in $\widehat{\mathcal{T}}$. 

As illustrated in Figure~\ref{Ftype2}(b), if we cut $\widehat{\mathcal{T}}$ open along $P_{k-1}$, the resulting manifold, denoted by $\mathfrak{X}_{k-1}$, is a type II block.  Denote the two components of $\partial_h\mathfrak{X}_{k-1}$ by $P_{k-1}^-$ and $P_{k-1}^+$, and suppose the induced normal direction at $P_{k-1}^-$ points into $\mathfrak{X}_{k-1}$, see  Figure~\ref{Ftype2}(b).  
Moreover,  $\Gamma_{k-2}$ can be viewed as a collection of $\partial$-parallel annuli in $\mathfrak{X}_{k-1}$.

Similarly, we connect $P_{k-1}^-$ to the annuli in $\Gamma_{k-2}$ using tubes along unknotted arcs, and denote the resulting planar surface by $P_{k-2}$.    We may view $P_{k-2}$ as a non-separating surface in $\widehat{\mathcal{T}}$ and let $\mathfrak{X}_{k-2}$ be the manifold obtained by  cutting $\widehat{\mathcal{T}}$ open along $P_{k-2}$.  
As illustrated in the schematic picture in Figure~\ref{Ftype2}(d), $\mathfrak{X}_{k-2}$ is also a type II block with the two components of $\partial_h\mathfrak{X}_{k-2}$ being the two sides of $P_{k-2}$. Denote the two sides of $P_{k-2}$ by $P_{k-2}^\pm$.

We continue this procedure:  Let $P_{j-1}$ be the planar surface obtained by connecting the annuli in $\Gamma_{j-1}$ to $P_j^-$ using tubes along unknotted arcs for each $j$.  Denote the final planar surface $P_0$ by $P$.  So $P$ is a non-separating planar surface with $\partial P=C_0=\partial S$ and with compatible normal direction.  Moreover, the manifold obtained by cutting $\widehat{\mathcal{T}}$ open along $P$ is a type II block.

Suppose $\Sigma\cap \mathcal{W}$ contains another non-separating surface $S'$.  
Next we build a non-separating planar surface $P'$ in the solid torus $\widehat{\mathcal{T}}$ with $\partial P'=\partial S'$ in $\mathcal{T}$ and such that $P'$ and $P$ divide $\widehat{\mathcal{T}}$ into a pair of type II blocks.

Assign a normal direction to $S'$ compatible with the normal direction of $S$, i.e.~$S$ and $S'$ (with this orientation) represent the same element in $H_2(\mathcal{W},\mathcal{T})$.  The normal direction of $S'$ induces an orientation for each boundary curve.  We say two curves in $\partial S$ and $\partial S'$ have the same orientation if they represent the same element in $H_1(\mathcal{T})$, otherwise, they have opposite orientations.

Let $Y_P$ be the type II block obtained by cutting the solid torus $\widehat{\mathcal{T}}$ open along $P$. 
We may view $\partial S'$ as a set of curves in $\partial_vY_P$. 

\begin{claims}\label{Claimcurves}
Let $A$ be a component of $\partial_vY_P$, and let $\gamma'$ and $\gamma''$ be the two boundary curves of $A$. 
Let $\gamma_1,\dots,\gamma_n$ be the components of $\partial S'$ that lie in $A$.  
Suppose $\gamma_i$ is adjacent to $\gamma_{i+1}$ for any $i=1,\dots, n-1$, and suppose $\gamma'$ (resp.~$\gamma''$) is adjacent to $\gamma_1$ (resp.~$\gamma_n$) in $A$. Then
\begin{enumerate}
  \item if $\gamma'$ and $\gamma''$ have the same orientation, then $A$ contains at least one curve of $\partial S'$, 
  \item any two adjacent curves  $\gamma_i$ and $\gamma_{i+1}$ have opposite orientations,
  \item $\gamma'$ and $\gamma_1$ have the same orientation, and
  \item $\gamma''$ and $\gamma_n$ have the same orientation
\end{enumerate}
\end{claims}
\begin{proof}[Proof of Claim~\ref{Claimcurves}]
Suppose part (1) is false, i.e.~$\gamma'$ and $\gamma''$ have the same orientation and $A\cap\partial S'=\emptyset$. 
Let $\tau$ be an essential arc in $A$ connecting $\gamma'$ to $\gamma''$.
Since $\gamma'$ and $\gamma''$ have the same orientation, $\tau$ is an arc connecting the plus side of $S$ to its minus side and $\tau\cap S'=\emptyset$. 
However, since $S$ and $S'$ represent the same element in $H_2(\mathcal{W}, \mathcal{T})\cong\mathbb{Z}$, any arc connecting the plus side to the minus side of $S$ must intersect $S'$. We have a contradiction.

Similarly, if two adjacent curves  $\gamma_i$ and $\gamma_{i+1}$ have the same orientation, then an arc in $A$ connecting  $\gamma_i$ to $\gamma_{i+1}$ is an arc connecting the plus side of $S'$ to its negative side without intersecting $S$.  Again, this is impossible since $H_2(\mathcal{W}, \mathcal{T})\cong\mathbb{Z}$.  Hence part (2) of the claim holds.

Now consider $\gamma'$ and $\gamma_1$. Since $\gamma'$ and $\gamma_1$ are adjacent in $A$, there is an arc $\tau$ in $A$ connecting $\gamma'$ to $\gamma_1$ and disjoint from all other curves. Without loss of generality, suppose $\tau$ is on the plus side of $S$. 
If $\gamma'$ and $\gamma_1$ have opposite orientations, then $\tau$ connects the plus side of $S$ to the plus side of $S'$. This is also impossible since $S$ and $S'$ have compatible orientations and represent the same element in $H_2(\mathcal{W},\mathcal{T})$, which means that the plus side of $S$ faces the minus side of $S'$.  Thus part (3) of the claim holds. Part (4) is similar to part (3).
\end{proof}

Next, we construct a non-separating planar surface $P'$ in the solid torus $\widehat{\mathcal{T}}$, such that $\partial P'=\partial S'$, and $P$ and $P'$ divide $\widehat{\mathcal{T}}$ into two type II blocks.

Before we proceed, we describe the type II block $Y_P$ (obtained by cutting the solid torus $\widehat{\mathcal{T}}$ open along $P$) as in the definition of type II block in section~\ref{StypeII}:  
First let $X_P$ be an extended marked handlebody obtained by connecting $D^2\times I$ and a collection of solid tori $T_1,\dots, T_m$ to a 3-ball $B$ using 1-handles.  
Set $A_0=(\partial D^2)\times I$, $A_i\subset\partial T_i$ and $\partial_vX_P=\bigcup_{i=0}^mA_i$ as in section~\ref{StypeII}.  
The two components of $\partial_hX_P$ consist of the disk $D^2\times\{0\}$ and a surface that we denote by $P^+$. Suppose $P^+\cong P$ ($P^+$ denotes the plus side of $P$).
Let $P^-$ be an extended suspension surface obtained by tubing together $D^2\times\{0\}$ and a collection of non-nested $\partial$-parallel annuli $\Gamma_1,\dots,\Gamma_s$. 
 $Y_P$ can be defined to be the submanifold of $X_P$ bounded by $P^-$, $P^+$, and a collection of subannuli in $\partial_vX_P$.  We view $P^\pm$ as the two sides of $P$.  
Suppose the induced normal direction for $P^+$ points out of $Y_P$ and the direction for $P^-$ points into $Y_P$.

Consider the curves in $\partial S'$.  Since $Y_P\subset X_P$, we may view $\partial_vY_P$ as a collection of subannuli of $\partial_vX_P$ and $\partial S'\subset\partial_vY_P\subset \partial_vX_P$.  In particular, we also view $\partial S'$ as curves in $\partial_vX_P=\bigcup_{i=0}^mA_i$.

\begin{claims}\label{ClaimS'}
Let $X_P$ be the extended marked handlebody as above and $\partial S'\subset\partial_vX_P$. Then
\begin{enumerate}
  \item any two adjacent curves of $\partial S'$ in each $A_i$ ($i=0,\dots, m$) have opposite orientations
  \item $A_0\cap \partial S'$ contains an odd number of curves, and
  \item $A_i\cap \partial S'$ contains an even number of curves for each $i\ge 1$.
\end{enumerate}
\end{claims}
\begin{proof}[Proof of Claim~\ref{ClaimS'}]
We first prove part (1). 
By Claim~\ref{Claimcurves}(2), any two adjacent curves of  $\partial S'$ in each component of $\partial_vY_P$  have opposite signs.  So either part (1) holds or there are two curves  $\gamma_1$, $\gamma_2$ of $\partial S'$ that are adjacent in $A_i$ ($i=0,\dots, m$) but lie in different components of $\partial_vY_P$.  
Let $A_\gamma$ be the subannulus of $A_i\subset \partial_vX_P$ bounded by $\gamma_1$ and $\gamma_2$.  
Since $\partial S'\subset\partial_vY_P$ and since $\gamma_1$ and $\gamma_2$ lie in different components of $\partial_vY_P$, $A_\gamma$ contains two subannuli of $\partial_vY_P$ next to $\partial A_\gamma$ and possibly some whole components of $\partial_vY_P$. In particular, $A_\gamma$ contains an even number of curves of $\partial S$.   

If $A_\gamma$ contains a whole component $E$ of $\partial_vY_P$, then since $\gamma_1$ and $\gamma_2$ are adjacent in  $A_i$, $E$ does not contain any curve of $\partial S'$, and by  Claim~\ref{Claimcurves}(1), the two boundary curves of $E$ must have opposite orientations. Recall that we view $P^-$ as an extended suspension surface obtained by tubing together $\Gamma_1,\dots,\Gamma_s$ and $D^2\times\{0\}$. So the two boundary curves of each $\partial$-parallel annulus $\Gamma_i$ have opposite orientations.  Any two adjacent curves of $\partial S$ in $A_\gamma$ are either boundary curves of a component $E$ of $\partial_vY_P$ or the boundary curves of $\Gamma_i$ for some $i=1,\dots, s$.  This implies that any two adjacent curves of $\partial S$ in $A_\gamma$ have opposite orientations.
Since $A_\gamma$ contains an even number of curves of $\partial S$, it follows from Claim~\ref{Claimcurves}(3, 4) that $\gamma_1$ and $\gamma_2$ must have opposite orientations. Hence part (1) of the claim holds.

An essential arc of the annulus $A_0$ is an arc connecting $P^-$ to $P^+$, which gives rise to an arc in $\mathcal{W}$ connecting the plus side of $S$ to its minus side. 
Since $S$ and $S'$ represent the same element in homology, this arc must intersect $S'$ an odd number of times, and this implies that $A_0\cap\partial S'$ contains an odd number of curves. Hence part (2) of the claim holds.  Similarly, an essential arc of $A_i$ ($i\ge 1$) is an arc connecting $P^+$ to $P^+$, so part (3) holds for the same homological reason.
\end{proof}

Next we temporarily ignore $P^-$ and construct a surface $P'$ in $X_P$ with $\partial P'=\partial S'$.

Let $\delta_0$ be the component of $A_0\cap \partial S'$ that is adjacent to the curve $\partial(D^2\times\{0\})$ in $A_0$, and let $\hat{\delta}_0$ be the disk in $B'=D^2\times I$ in the form of $D^2\times\{x\}$ bounded by $\delta_0$, see the 1-dimensional schematic picture in Figure~\ref{Fnonsep}(a).  By Claim~\ref{ClaimS'}(1), any two adjacent curves of $A_i\cap\partial S'$ have opposite orientations.  So we can use a collection of non-nested $\partial$-parallel annuli $\Gamma_1',\dots,\Gamma_t'$ to connect the remaining curves $\partial S'\setminus\delta_0$ in pairs, see Figure~\ref{Fnonsep}(a). Then we use tubes to connect the disk $\hat{\delta}_0$ and these $\partial$-parallel annuli $\Gamma_1',\dots,\Gamma_t'$ together in a manner similar to the construction of the extended suspension surface in section~\ref{StypeII}.  Denote the resulting surface by $P'$, so $\partial P'=\partial S'$.  See the red curves in Figure~\ref{Fnonsep}(b) for a schematic picture of $P'$.  Moreover, $P'$ has a normal direction compatible with the orientations of the curves in $\partial S'$.

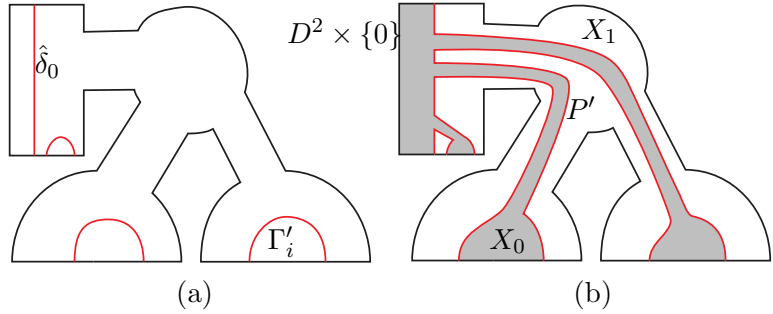
\begin{figure}[h]
\begin{center}
\begin{overpic}[width=4in]{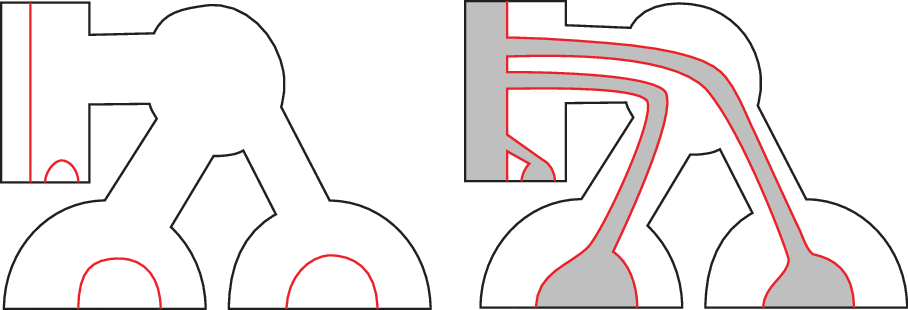}
\put(22,-5){(a)}
\put(74,-5){(b)}
\put(3.5, 25){$\hat{\delta}_0$} 
\put(34, 2){$\Gamma_i'$} 
\put(73, 19){$P'$}
\put(36.5, 29){$D^2\times\{0\}$}
\put(63, 2){$X_0$}    
\put(75, 29.5){$X_1$}    
\end{overpic}
\vspace{6pt}
\caption{Construct a suface $P'$ with $\partial P'=\partial S'$}\label{Fnonsep}
\end{center}
\end{figure}

As illustrated in Figure~\ref{Fnonsep}(b), $P'$ divides $X_P$ into two submanifolds $X_0$ and $X_1$ with $D^2\times\{0\}\subset \partial X_0$ and $P^+\subset\partial X_1$.  By the construction of $P'$, we may view $X_0$ as an extended marked handlebody and view $X_1$ as a type II block.    Since $\partial S$ and $\partial S'$ represents the same element in $H_1(\mathcal{T})$, the normal direction of $P'$ is compatible with the direction of $P^+$.  So the direction of $P'$ points into $X_1$ and out of $X_0$.  

Now we consider $P^-$. Recall that  the extended suspension surface $P^-$ was constructed by connecting $D^2\times\{0\}$ and 
a collection of $\partial$-parallel annuli $\Gamma_1,\dots,\Gamma_s$ using tubes.  Let $\widehat{\Gamma}_i$ be the solid torus bounded by $\Gamma_i$ and the subannulus of $\partial_vX_P$ between $\partial\Gamma_i$ .  The type II block $Y_P$ and each $\widehat{\Gamma}_i$ lie on different sides of $P^-$ in $X_P$.  
Since $\partial P'$ lies in $\partial_vY_P$, this implies that no curve of $\partial P'=\partial S'$ lies inside the solid torus $\widehat{\Gamma}_i$.  Thus, each solid torus $\widehat{\Gamma}_i$  ($i=1,\dots, s$) lies entirely in either $X_0$ or $X_1$ (possibly after shrinking $\widehat{\Gamma}_i$ if necessary).

\begin{claims}\label{Claimannuli}
$\Gamma_1,\dots,\Gamma_s$ all lie in $X_0$ (possibly after an isotopy relative to $\partial\Gamma_i$).  
\end{claims}
\begin{proof}[Proof of Claim~\ref{Claimannuli}]
Suppose the claim is false. Then, since each solid torus $\widehat{\Gamma}_i$  ($i=1,\dots, s$) lies entirely in either $X_0$ or $X_1$, there must be some annulus $\Gamma_i$ with $\widehat{\Gamma}_i\subset X_1$.  
We may choose such a $\Gamma_i$ to be adjacent to some curve $\gamma$ of $\partial P'=\partial S'$.  In other words, choose $\Gamma_i$ so that there is an annulus $C$ between $\Gamma_i$ and a curve $\gamma$ of $\partial P'$ with the following properties: 
\begin{enumerate}
  \item the two components of $\partial C$ are $\gamma$ ($\gamma\subset\partial P'$) and a component of $\partial\Gamma_i$,
  \item  $C\subset X_1$,  
  \item $\Int(C)$ does not contain a boundary curve of any $\Gamma_j$  ($j=1,\dots, s$).
\end{enumerate}  
By the construction of $P^-$, the normal direction of $\Gamma_i$ points out of the solid torus $\widehat{\Gamma}_i$.  We have concluded earlier that the induced direction of $P'$ points into $X_1$.   
This implies that the normal directions at the two boundary curves of $C$ must both point into $C$. However, this means that the two curves of $\partial C$ have opposite orientations, contradicting Claim~\ref{Claimcurves}(3, 4). 
\end{proof}

Since $P^-$ is an extended suspension surface obtained by tubing together the $\partial$-parallel annuli $\Gamma_1,\dots,\Gamma_s$ and the disk $D^2\times\{0\}$, 
 Claim~\ref{Claimannuli} implies that $P^-$ can be constructed as an extended suspension surface in the extended marked handlebody $X_0$.  Hence the submanifold of $X_0$ between $P^-$ and $P'$ is a type II block.  In other words, $P'$ is a surface in $Y_P$ with $\partial P'=\partial S'$ and $P'$ divides $Y_P$ into two type II blocks.
 By gluing $P^-$ to $P^+$, we get back the solid torus $\widehat{\mathcal{T}}$.  Hence, $P$ and $P'$ divide $\widehat{\mathcal{T}}$ into a pair of type II blocks.

Let $S_1,\dots, S_n$ be the components of $\Sigma\cap \mathcal{W}$ that are non-separating in $\mathcal{W}$.  
As before, we may assign each $S_i$ a compatible normal direction so that they represent the same element in $H_2(\mathcal{W},\mathcal{T})$. 
By repeating the arguments above, we get a collection of non-separating planar surfaces $P_1,\dots, P_n$ in the solid torus $\widehat{\mathcal{T}}$ such that $\partial P_i=\partial S_i$ and $P_1,\dots, P_n$ divide $\widehat{\mathcal{T}}$ into a collection of type II blocks.

Assign each $P_i$ a normal direction compatible with the induced normal direction of $\partial P_i=\partial S_i$.  We refer to the two sides of each $P_i$ as the left and right sides with the normal direction of $P_i$ pointing from the left to the right side.

\vspace{10pt}

\noindent
{\bf Case (1)}.  $n=1$, i.e.~$\Sigma\cap\mathcal{W}$ contains only one non-separating surface $S_1$. 

\vspace{10pt}

Let $\widetilde{Y}$ be the manifold obtained by cutting $\mathcal{W}$ open along $S_1$, and let $\widetilde{X}_0$ be the type II block obtained by cutting $\widehat{\mathcal{T}}$ open along $P_1$.  
Let $P^+$ and $P^-$ be the two surfaces in $\partial_h\widetilde{X}_0$ representing the two sides of $P_1$, and suppose the induced normal direction of $P^-$ points into  $\widetilde{X}_0$. 
Let $P_1'$ be the surface obtained from $P_1$ by adding $g_1$  trivial tubes ($g_1=g(S_1)$)  on the right side of $P_1$, and let $\widetilde{X}$ be the manifold obtained by cutting $\widehat{\mathcal{T}}$ open along $P_1'$.  So $\widetilde{X}$ can be obtained from $\widetilde{X}_0$ by first adding $g_1$ 1-handles at $P^+$ and then drilling out $g_1$ trivial tunnels at $P^-$. $\widetilde{X}$ is also a type II block. 
Let $P_l$ and $P_r$ be the two components of $\partial_h\widetilde{X}$  representing the two sides of $P_1'$.  
Similarly, denote the two surfaces in $\widetilde{Y}$ representing the two sides of $S_1$  by $S_l$ and $S_r$ respectively. 

The surfaces of $\Sigma\cap\mathcal{W}$ that lie in $\Int(\widetilde{Y})$ divide $\widetilde{Y}$ into a stack of relative compression bodies.  
Since $S_1$ is the only non-separating component of $\Sigma\cap\mathcal{W}$,  there is a relative compression body $Y$ in the stack $\widetilde{Y}$ that contains both $S_l$ and $S_{r}$.  So $\partial_hY$ consists of $S_l$, $S_{r}$, and a collection of separating surfaces $Q_1,\dots,Q_k$ in $\mathcal{W}$.  

Next, we construct a collection of surfaces $Q_1',\dots, Q_k'$ in $\widetilde{X}$  such that (1)  $\partial Q_j'=\partial Q_j$ in $\mathcal{T}$, (2) $Q_j'\cong Q_j$  for all $j$, and (3) the submanifold $X$ bounded by $P_l$, $P_{r}$, $Q_1',\dots, Q_k'$ (and subannuli of $\partial_v \widetilde{X}$) is a relative compression body $\partial$-similar to $Y$.

The argument is similar to section~\ref{StypeI}, and the only difference here is that $\widetilde{X}$ is a type II block.  Instead of repeating the argument in section~\ref{StypeI}, we convert $\widetilde{X}$ into a type I block or a marked handlebody and then apply Lemma~\ref{LtypeI}. 

In section~\ref{StypeII}, we described a way of converting a type II block into a type I block by drilling out a tunnel $N(\alpha)$, where $\alpha$ is of the form $\{x\}\times I\subset D^2\times I=B'$ in the construction of an extended marked handlebody.  
We view $N(\alpha)=\Delta_\alpha\times I$, where $\Delta_\alpha$ is a disk, and suppose $\Delta_\alpha\times\{0\}\subset P_l$ and $\Delta_\alpha\times\{1\}\subset P_{r}$.  Let $P_l'=P_l\setminus(\Delta_\alpha\times\{0\})$, $P_{r}'=P_{r}\setminus (\Delta_\alpha\times\{1\})$, and $\mathcal{A}_\alpha=(\partial\Delta_\alpha)\times I$. 
Let $\widetilde{X}'=\widetilde{X}\setminus N(\alpha)$.

Next, we construct a relative compression body $X$ that is $\partial$-similar to $Y$. 

\vspace{10pt}

\noindent
{\it Subcase (1a)}.  $S_1$ has a minus sign.

\vspace{10pt}

In this case, $P_l$ and $P_r$ have minus signs. 
As in Remark~\ref{Rconversion}, instead of viewing $\widetilde{X}'=\widetilde{X}\setminus N(\alpha)$ as a type I block, we may view $\widetilde{X}'$ as a marked handlebody with $\partial_h \widetilde{X}'=P_l'\cup P_{r}'\cup \mathcal{A}_\alpha$. 

Now consider the relative compression body $Y$.  In this subcase, both $S_l$ and $S_{r}$ are components of $\partial_h^-Y$.  By Remark~\ref{Rproperty}(4), there is a core arc $\beta$ in $Y$ connecting $S_l$ to $S_{r}$, such that (1) $Y\setminus N(\beta)$ is a relative compression body and (2) $S_l$ and $S_{r}$ are tubed together and become a component of $\partial_h^-(Y\setminus N(\beta))$.

By Lemma~\ref{LtypeI}, we can build a collection of surfaces $Q_1',\dots, Q_k'$ in the marked handlebody $\widetilde{X}'=\widetilde{X}\setminus N(\alpha)$ such that (1) $\partial Q_j'=\partial Q_j$ in $\mathcal{T}$ and $Q_j'\cong Q_j$  for all $j$ and (2) the type I block, denoted by $X'$, between $\partial_h \widetilde{X}'$ and $Q_1',\dots, Q_k'$  is a relative compression body $\partial$-similar to $Y\setminus N(\beta)$.  Let $X$ be the manifold obtained by adding a 2-handle to $X'$ along the annulus $\mathcal{A}_\alpha$ (i.e.~filling the tunnel $N(\alpha)$).   Since the 2-handle is added to $\partial_h^- X'$, by Remark~\ref{Rproperty}(5), $X$ is a relative compression body $\partial$-similar to $Y$. 

\vspace{10pt}

\noindent
{\it Subcase (1b)}.  $S_1$ has a plus sign.

\vspace{10pt}

In this case, both $S_l$ and $S_r$ have plus signs. 
Let $\gamma$ be a $\partial$-parallel arc in $Y$ connecting $S_l$ to $S_r$ and parallel to an arc in $\partial^+Y$.  Let $N(\gamma)$ be a small neighborhood of $\gamma$ in $Y$.  We view $N(\gamma)=\Delta_\gamma\times I$ with $\Delta_\gamma\times\partial I\subset S_l\cup S_{r}$.  Let $\mathcal{A}_\gamma=(\partial\Delta_\gamma)\times I$.  Let $S_l'=S_l\setminus (\Delta_\gamma\times\partial I)$ and $S_r'=S_r\setminus (\Delta_\gamma\times\partial I)$. 
As in Remark~\ref{Rproperty}(7), after setting $\mathcal{A}_\gamma$ as a component of $\partial_v^- (Y\setminus N(\gamma))$ with a large order, $Y\setminus N(\gamma)$ is a relative compression body, where $S_l'$ and $S_r'$ are two components of $\partial_h^+(Y\setminus N(\gamma))$.

Now consider $\widetilde{X}'=\widetilde{X}\setminus N(\alpha)$ defined before Subcase (1a). The difference from Subcase (1a) is that, in this subcase, we view $\widetilde{X}'=\widetilde{X}\setminus N(\alpha)$ as a type I block (see section~\ref{StypeII}) and view $\mathcal{A}_\alpha$ as a component of $\partial_v^-\widetilde{X}'$ with induced order from $\mathcal{A}_\gamma\subset\partial_v^- (Y\setminus N(\gamma))$. Moreover, in this subcase, $\partial_h\widetilde{X}'=P_l'\cup P_r'$.

Since $\widetilde{X}'$ is a type I block, we may view $\widetilde{X}'$ as a submanifold of a marked handlebody as in the definition of type I block in section~\ref{StypeI}:  First take a marked handlebody $Z$ with $\partial_hZ=P_{r}'$, then view $P_l'$ as a standard surface in $Z$ such that the submanifold of $Z$ between $P_r'$ and $P_{l}'$ is the type I block  $\widetilde{X}'$.    

Now we apply Lemma~\ref{LtypeI} to the marked handlebody $Z$. 
By the construction in Lemma~\ref{LtypeI}, we can build a collection of standard surfaces $Q_1',\dots, Q_k'$ in $Z$, such that the cross-section disks for $P_l'$, $Q_1',\dots, Q_k'$ are well-positioned in $Z$, $\partial Q_i'=\partial Q_i$, and $Q_i'\cong Q_i$.  Moreover, the type I block, denoted by $X'$, bounded by $P_{r}'$, $P_l'$, $Q_1',\dots, Q_k'$ (and subannuli of $\partial_v\widetilde{X}'$) is a relative compression body that is $\partial$-similar to $Y-N(\gamma)$.  
Let $X$ be the manifold obtained by adding a 2-handle to $X'$ along the annulus $\mathcal{A}_\alpha$ (i.e.~filling the tunnel $N(\alpha)$). 
Since  $\mathcal{A}_\alpha$ is an annulus in  $\partial_v^- X'$, by Remark~\ref{Rproperty}(5), $X$ is a relative compression body $\partial$-similar to $Y$. 

In both subcases, each $Q_i'$ in $\widetilde{X}$ cuts off a marked handlebody from $\widetilde{X}$.   
By Lemma~\ref{LtypeI}, we can construct a collection of surfaces in each of these marked handlebodies, such that these surfaces, together with $Q_1',\dots, Q_k'$, divide $\widetilde{X}$ into a stack that is similar to the stack $\widetilde{Y}$.
After gluing $P_l$ to $P_r$, we have a collection of surfaces that divide $\widehat{\mathcal{T}}$ into a stack of relative compression bodies similar to the stack $\mathcal{W}$.

\vspace{10pt}

\noindent
{\bf Case (2)}. $n>1$.

\vspace{10pt}

Consider the non-separating surfaces $S_1,\dots, S_n$ in  $\Sigma\cap\mathcal{W}$ and suppose each $S_i$ is adjacent to $S_{i+1}$. 
We first show that there must be two adjacent non-separating surfaces $S_i$ and $S_{i+1}$ (set $S_{n+1}=S_1$) with the same sign.  The reason for this comes from the untelescoping construction. Recall in section~\ref{SGHS}, $\Sigma$ divides $M$ into a collection of compression bodies $\mathcal{A}_0,\dots,\mathcal{A}_m$ and $\mathcal{B}_0,\dots,\mathcal{B}_m$, where each $\mathcal{A}_i$ is obtained by adding 1-handles either to 0-handles or to $\mathcal{B}_{i-1}$ along $\mathcal{F}_i$, and each $\mathcal{B}_i$ is obtained by adding 2- and 3-handles to $\mathcal{A}_i$ along $\mathcal{P}_i$ ($i=0,\dots , m$), see Figure~\ref{Funtel}.  We can set a direction of ``adding handles" on these surfaces as follows: the direction for each $\mathcal{F}_i$ points from $\mathcal{B}_{i-1}$ to $\mathcal{A}_i$, and the direction for each $\mathcal{P}_i$ points from $\mathcal{A}_i$ to $\mathcal{B}_i$.  This is the direction of the handle addition, where one starts from the handlebody $\mathcal{A}_0$ and ends at the handlebody $\mathcal{B}_m$ by adding handles in this direction, see Figure~\ref{Funtel}. Note that Figure~\ref{Funtel} is the simplest diagram describing generalized Heegaard splittings, see \cite{SSS} for more complicated diagrams.  Therefore, if any two adjacent non-separating surfaces $S_i$ and $S_{i+1}$ (set $S_{n+1}=S_1$) have different signs, then the surfaces $\mathcal{F}_i$'s and $\mathcal{P}_i$'s appear alternately in $\mathcal{W}$, which implies that the handle-addition direction gives an oriented cycle. This is impossible because this is a direction for adding handles.

Without loss of generality, suppose $S_1$ and $S_2$ have the same sign.

The non-separating surfaces $S_1,\dots, S_n$ divide $\mathcal{W}$ into submanifolds $\widetilde{Y}_1,\dots, \widetilde{Y}_n$, where $\widetilde{Y}_i$ is the submanifold between $S_i$ and $S_{i+1}$.  
Similarly, $P_1,\dots, P_n$ divide $\widehat{\mathcal{T}}$ into a collection of type II blocks $\widetilde{X}_1,\dots, \widetilde{X}_n$, where $\widetilde{X}_i$ is the type II block between $P_i$ and $P_{i+1}$. 

First, add $g_1$ trivial tubes to $P_1$ on its right side (i.e.~in $\widetilde{X}_1$), where $g_1=g(S_1)$. Denote the resulting surface by $P_1'$. So $P_1'\cong S_1$.  Now replace $P_1$ by $P_1'$ and consider the collection of non-separating surfaces $P_1', P_2,\dots, P_n$ which divide $\widehat{\mathcal{T}}$ into submanifolds $\widetilde{X}_1', \widetilde{X}_2, \dots, \widetilde{X}_{n-1}, \widetilde{X}_n'$, where $\widetilde{X}_n'$ is obtained by adding $g_1$ 1-handles to $\widetilde{X}_n$,  
and $\widetilde{X}_1'$ is obtained from $\widetilde{X}_1$ by drilling out $g_1$ trivial  tunnels.  

Next we add $g_n$ tubes to $P_n$ on its right side ($g_n=g(S_n)$) and modify $\widetilde{X}_n'$.  We divide the discussion into two subcases.

\vspace{5pt}

\noindent
{\it Subcase (2a)}.  $S_n$ and $S_1$ have the same sign.

\vspace{5pt}

In this subcase, we simply add $g_n$ trivial tubes to $P_{n}$ on its right side and denote the resulting surface by $P_{n}'$.  So $P_n'\cong S_n$.  Now we replace $P_{n}$ by $P_{n}'$ and consider the new collection of surfaces $P_1', P_2,\dots, P_{n-1}, P_n'$ which divide $\widehat{\mathcal{T}}$ into  submanifolds $\widetilde{X}_1', \widetilde{X}_2, \dots, \widetilde{X}_{n-1}', \widetilde{X}_n''$, where $\widetilde{X}_n''$ is the submanifold between $P_n'$ and $P_1'$.  So $\widetilde{X}_n''$ is obtained from $\widetilde{X}_n'$ by drilling out  $g_n$ trivial tunnels, and $\widetilde{X}_{n-1}'$ is obtained by adding $g_n$ 1-handles to $\widetilde{X}_{n-1}$.

Since $P_n'$ and $P_1'$ have the same sign in this subcase,  $\widetilde{X}_n''$ has a similar structure as the manifold  $\widetilde{X}$ in Case (1).  So we can apply the argument on $\widetilde{X}$ in Case (1) to $\widetilde{X}_n''$ and construct a sequence of separating surfaces in $\widetilde{X}_n''$, such that these surfaces divide $\widetilde{X}_n''$ into a stack of relative compression bodies that is similar to the stack $\widetilde{Y}_n$.

\vspace{5pt}

\noindent
{\it Subcase (2b)}.  $S_n$ and $S_1$ have opposite signs.

\vspace{5pt}

Let $Y$ be the relative compression body in the stack $\widetilde{Y}_n$ that contains both $S_n$ and $S_1$. 
Since $S_n$ and $S_1$ have different signs, by Remark~\ref{Rproperty}(6), there is a vertical arc $\beta$ in $Y$ that connects $S_n$ to $S_{1}$, such that $Y-N(\beta)$ is a relative compression body.  More precisely, let $N(\beta)$ be a small neighborhood of $\beta$ in $Y$ and we view $N(\beta)=\Delta_\beta\times I$ with $\Delta_\beta\times\partial I\subset S_n\cup S_{1}$.  Let $\mathcal{A}_\beta=(\partial\Delta_\beta)\times I$.  Let $S_n'=S_n\setminus (\Delta_\beta\times\partial I)$ and $S_1'=S_1\setminus (\Delta_\beta\times\partial I)$. 
As in Remark~\ref{Rproperty}(6), we can choose an arc $\beta$ such that $Y\setminus N(\beta)$ is a relative compression body with $S_n'$ and $S_1'$ being two components of $\partial_h^\pm (Y\setminus N(\beta))$ and $\mathcal{A}_\beta$ a component of $\partial_v^0 (Y\setminus N(\beta))$.

Now we convert the type II block $\widetilde{X}_n'$ into a type I block, as in 
 section~\ref{StypeII}, by drilling out a tunnel $N(\alpha)$, where $\alpha$ is of the form $\{x\}\times I\subset D^2\times I=B'$ and $B'$ is the 3-ball in the construction of an extended marked handlebody.  
As in Case (1), we view $N(\alpha)=\Delta_\alpha\times I$, where $\Delta_\alpha$ is a disk, $\Delta_\alpha\times\{0\}\subset P_n$ and $\Delta_\alpha\times\{1\}\subset P_{1}'$.  
Let $P_n^c=P_n\setminus(\Delta_\alpha\times\{0\})$ and $P_{1}^c=P_{1}'\setminus (\Delta_\alpha\times\{1\})$.
Let $\widetilde{X}_n^c=\widetilde{X}_n'\setminus N(\alpha)$.    Let $\mathcal{A}_\alpha=(\partial\Delta_\alpha)\times I$. 
 So $\widetilde{X}_n^c$ is a type I block with $\partial_h \widetilde{X}_n^c=P_n^c\cup P_{1}^c$ and $\partial_v \widetilde{X}_n^c=\partial_v\widetilde{X}_n'\cup\mathcal{A}_\alpha$.

Similar to Subcase (1b), we may view the type I block $\widetilde{X}_n^c$ as a submanifold of a marked handlebody in section~\ref{StypeI} as follows:  First take a marked handlebody $Z$ with $\partial_hZ=P_{1}^c$, then view $P_n^c$ as a suspension surface in $Z$ such that the submanifold of $Z$ between $P_n^c$ and $P_{1}^c$ is the type I block  $\widetilde{X}_n^c$.   

Now we apply Lemma~\ref{LtypeI} to the marked handlebody $Z$. 
By the construction in Lemma~\ref{LtypeI}, we can build a collection of  surfaces $Q_1',\dots, Q_k'$ in $Z$, such that the cross-section disks for $P_n^c$ and $Q_1',\dots, Q_k'$ are well-positioned.  
We also add $g_n$ tubes to $P_n^c$ ($g_n=g(S_n)$) and denote the resulting surface by $P_n^d$.
Moreover, as  in  Lemma~\ref{LtypeI}, we can construct these surfaces and add these $g_n$ tubes so that the submanifold $Z^c$ between $P_1^c$, $P_n^d$, $Q_1',\dots, Q_k'$ 
 is a relative compression body $\partial$-similar to the relative compression body $Y-N(\beta)$.   

Note that, in the construction of Lemma~\ref{LtypeI} and as explained in Remark~\ref{RfixH}, some tubes of $P_n^d$ may be dragged through the 1-handles of $Z$, in other words, some tube of $P_n^d$ may be re-embedded into tubes that go through the $g_1$ tubes of $P_1'$.  
This is the main difference between Subcase (2a) and Subcase (2b).  

Let $X_n$ be the relative compression body obtained by adding a 2-handle to $Z^c$ along the annulus $\mathcal{A}_\alpha$ (i.e.~filling the tunnel $N(\alpha)$).  By Remark~\ref{Rproperty}(6), $X_n$ a relative compression body $\partial$-similar to $Y$.

Adding the 2-handle also extends the surface $P_n^d$ to a surface $P_n'$ with $P_n'\cong S_n$. 
Denote the resulting manifold between $P_n'$ and $P_1'$ by $\widetilde{X}_n''$.   
Similarly, we view $Q_1',\dots, Q_k'$ as separating surfaces in $\widetilde{X}_n''$. 
Since each $Q_i'$ cuts off a marked handlebody from $\widetilde{X}_n''$, as in Case (1), we can construct surfaces in $\widetilde{X}_n''$ which divide $\widetilde{X}_n''$ into a stack of relative compression bodies that is similar to the stack $\widetilde{Y}_n$.  

So, in both subcases, we can change $P_n$ to $P_n'$ and constructed surfaces that divide $\widetilde{X}_n''$ into a stack of relative compression bodies that is similar to the stack $\widetilde{Y}_n$.   

The operation of changing $P_n$ to $P_n'$ also affects the type II block $\widetilde{X}_{n-1}$.  For $\widetilde{X}_{n-1}$, the effect of adding tubes to $P_n$ on its right side is simply adding 1-handles to $\widetilde{X}_{n-1}$, and 
the topological type of the resulting manifold $\widetilde{X}_{n-1}'$ does not depend on whether or not the tubes that we added to $P_n$ are trivial tubes. Moreover, the possible tube re-embedding explained above and in Remark~\ref{RfixH} does not change the topological type of $\widetilde{X}_{n-1}'$ either. Thus $\widetilde{X}_{n-1}'$ can be viewed as the manifold obtained by adding $g_n$  1-handles to $\widetilde{X}_{n-1}$ at $P_n$.

Consider the new collection of surfaces $P_1', P_2,\dots, P_{n-1}, P_n'$ which divide $\widehat{\mathcal{T}}$ into submanifolds $\widetilde{X}_1', \widetilde{X}_2,\dots, \widetilde{X}_{n-2},\widetilde{X}_{n-1}', \widetilde{X}_n''$. 
Next, we inductively carry out this construction on $P_{n-1},\dots, P_2$ and successively modify the planar surfaces $P_{n-1},\dots, P_2$ into $P_{n-1}',\dots, P_2'$, with $P_i'\cong S_i$ for each $i$.  
$P_1',\dots, P_n'$ divide $\widehat{\mathcal{T}}$ into submanifolds $\widetilde{X}_{1}'',\dots, \widetilde{X}_n''$.  Similarly, we construct separating surfaces in $\widetilde{X}_{n-1}'',\dots, \widetilde{X}_2''$ which divide each $\widetilde{X}_{i}''$ into a stack of relative compression bodies  similar to the stack $\widetilde{Y}_{i}$ ($i=2,\dots, n-1$).

The last step is to consider the type II block $\widetilde{X}_{1}''$ between $P_1'$ and $P_2'$.  Recall that $P_1'$ is obtained by adding trivial tubes to $P_1$, so the topological structure of $\widetilde{X}_{1}''$ is similar to that of $\widetilde{X}$ in Case (1).  Since both $S_1$ and $S_2$ have the same sign, we can apply the construction on $\widetilde{X}$ in Case (1) to $\widetilde{X}_{1}''$ which makes $\widetilde{X}_{1}''$ a stack of relative compression bodies that is similar to the stack $\widetilde{Y}_{1}$. 
By gluing these stacks together, we obtain a stack in $\widehat{\mathcal{T}}$ similar to the stack $\mathcal{W}$.

In both Case (1) and Case (2) above, by gluing the stack of relative compression bodies in $\widehat{\mathcal{T}}$ to the stack of relative compression bodies in $\mathcal{V}$, as in section~\ref{Ssep}, we obtain a generalized Heegaard splitting for $N=\widehat{\mathcal{T}}\cup_\mathcal{T} \mathcal{V}$ with genus equal to $g(M)$.  This means that $g(N)\le g(M)$ and Theorem~\ref{Ttorus} holds.

%%%%%%%%%%%%%%

\end{document}